\documentclass[11.5pt,twoside,openany,UTF8,CJK]{article}
\usepackage{amssymb,amsthm}
\usepackage[tbtags]{amsmath}
\usepackage{cite}
\usepackage{indentfirst}
\usepackage{cases}
\usepackage{graphicx}
\usepackage{subfigure}
\usepackage{lineno}
\usepackage{enumerate}
\usepackage{threeparttable}
\usepackage{multirow}
\usepackage{booktabs}
\usepackage{upgreek}  
\usepackage{bm}
\usepackage{CJK}
\usepackage{fancyhdr}
\usepackage{ifthen}
\usepackage{lipsum}
\allowdisplaybreaks[4]
\usepackage[misc]{ifsym}
\usepackage{caption}
\usepackage{float}
\usepackage{appendix}
\usepackage[format=hang]{caption}
\usepackage{algorithm,algorithmic}

\usepackage{charter}  
\usepackage{indentfirst}  
\usepackage{soul}
\usepackage{authblk}
\allowdisplaybreaks[4]
\newcommand{\Rmnum}[1]{\uppercase\expandafter{\romannumeral #1}} 

\usepackage[colorlinks, linkcolor=blue, anchorcolor=blue, citecolor=red]{hyperref}

\numberwithin{equation}{section}

\newtheorem{Lemma}{Lemma}[section]
\newtheorem{Theorem}{Theorem}[section]

\newcounter{saveeqn}

\makeatletter
\def\@maketitle{%
	\newpage
	\null
	\vskip 2em%
	\begin{center}%
		\let \footnote \thanks
		{\LARGE \@title \par}%
		\vskip 1.5em%
		{\large
			\lineskip .5em%
			\begin{tabular}[t]{c}%
				\@author
			\end{tabular}\par}%
	\end{center}%
	\par
	\vskip 1.5em}
\makeatother

\makeatletter
\evensidemargin
\oddsidemargin
\makeatother

\title{
	\textbf{Optimal convergence analysis of fully discrete SAVs-FEM for the Cahn-Hilliard-Navier-Stokes equations}\thanks{Supported by the National Natural Science Foundation of China (No.11971337) 
	and the Natural Science Foundation of Sichuan Province (No. 2025ZNSFSC0070)}
}
\author[1]{Haijun Gao}
\author[2]{Xi Li}
\author[3]{Cheng Wang}
\author[1]{Minfu Feng\thanks{Corresponding author. 
		E-mail: 
			gaohijun@163.com, lixi@cdut.edu.cn, 
			cwang1@umassd.edu, 
			fmf@scu.edu.cn.}}
\affil[1]{School of Mathematics, Sichuan University, 610065, Chengdu, China.}
\affil[2]{School of Mathematical Sciences, Chengdu University of Technology, Chengdu, Sichuan 610059, China.}
\affil[3]{Mathematics Department, University of Massachusetts Dartmouth, North Dartmouth, MA 02747, USA.}

\date{}

\begin{document}
	\maketitle
	\begin{abstract}
		\indent
		We construct a fully discrete numerical scheme that is linear, decoupled, and unconditionally energy stable, and analyze its optimal error estimates for the Cahn-Hilliard-Navier-Stokes equations. For time discretization, we employ the two scalar auxiliary variables (SAVs) and the pressure-correction projection method. For spatial discretization, we choose the $P_r \times P_r \times \mathbf{P}_{r+1} \times P_r$ finite element spaces, where $r$ is the degree of the local polynomials, and derive the optimal $L^2$ error estimates for the phase-field variable, chemical potential, and pressure in the case of $r \geq 1$, and for the velocity when $r \geq 2$, without relying on the quasi-projection operator technique proposed in \textit{[Cai et al. SIAM J Numer Anal, 2023]}. Numerical experiments validate the theoretical results, confirming the unconditional energy stability and optimal convergence rates of the proposed scheme. Additionally, we numerically demonstrate the optimal $L^2$ convergence rate for the velocity when $r=1$.
		\\
		
	\end{abstract}
	\thispagestyle{empty}
	
\section{Introduction}\label{section_introduction}
In this article, we consider the Cahn-Hilliard-Navier-Stokes (CHNS) equations \cite{2006_FengXiaobing_FullydiscretefiniteelementapproximationsoftheNavierStokesCahnHilliarddiffuseinterfacemodelfortwophasefluidflows} as follows:
\begin{subequations}
	\label{eqCHNS01}
	\begin{align}
		\frac{\partial \phi}{\partial t}+(\mathbf{u}\cdot \nabla)\phi-M\Delta \mu =0~~&\text{in}~~ \Omega\times (0,T],\\
		\mu+\lambda \Delta \phi-\lambda G'(\phi)=0~~&\text{in}~~ \Omega\times (0,T],\\
		\frac{\partial \mathbf{u}}{\partial t}+\mathbf{u}\cdot\nabla \mathbf{u}-\nu\Delta \mathbf{u}+\nabla p-\mu\nabla\phi=0~~&\text{in}~~ \Omega\times (0,T],\\
		\nabla\cdot \mathbf{u}=0~~&\text{in}~~ \Omega\times (0,T].
	\end{align}
\end{subequations}
The $G(\phi)=\frac{1}{4 \epsilon^2}(\phi^2-1)^2$ is a nonlinear free energy density, where $\epsilon$ denotes the interface width and $M,\lambda,\nu>0,$ describes the mobility, mixing coefficient, and fluid viscosity, respectively.
We consider the following  no-flux or no-flow boundary and initial conditions of 
\eqref{eqCHNS01}:
\begin{subequations}
	\label{eq_boundary_initial_conditions_equations}
	\begin{align}
		\label{eq_boundary_conditions_equation}
		\frac{\partial \phi}{\partial \mathbf{n}}=\frac{\partial\mu}{\partial \mathbf{n}}=0,\quad\mathbf{u}=0,&\quad\text{on}\quad \partial\Omega\times  (0,T],\\
		\label{eq_initial_conditions_equation}
		\phi(\mathbf{x},0)=\phi^0,\quad\mathbf{u}(\mathbf{x},0)=\mathbf{u}^0,&\quad\text{in}\quad \Omega,
	\end{align}
\end{subequations}
where $\Omega$ is a bounded domain in $\mathbb{R}^2$ 
with boundary $\partial \Omega$, 
and $\mathbf{n}$ denotes the unit outward normal vector on $\partial\Omega$.
The unknowns are the phase field function $\phi$ and the chemical potential $\mu$, 
the velocity $\mathbf{u}$, 
the pressure $p$. 
It is well-known that the energy dissipation law for system  \eqref{eqCHNS01}-\eqref{eq_boundary_initial_conditions_equations} is as follows:
\begin{equation}
	\frac{d E(\phi, \mathbf{u})}{d t}  =-M\|\nabla \mu\|^2-\nu\|\nabla \mathbf{u}\|^2,
\end{equation}
where $E(\phi, \mathbf{u}) 
=\int_{\Omega}\left\{\frac{1}{2}|\mathbf{u}|^2+\frac{\lambda}{2}|\nabla \phi|^2+\lambda G(\phi)\right\} d \mathbf{x}$ is the total energy, and $\|\cdot\|$  is $L^2$ norm.

In the last two decades, for numerical algorithms for the CHNS equation, researchers constructed various different discrete schemes. 
Kay et al. \cite{2007_Kay_David_and_Welford_Richard_Efficient_numerical_solution_of_Cahn_Hilliard_Navier_Stokes_fluids_in_2D} proposed the decoupled semi-discrete scheme by using explicit treatment for the coupled terms of the Cahn-Hilliard (CH) equation. Their discrete scheme computes CH and Navier-Stokes (NS) separately, but the whole discrete system is still a nonlinear scheme. For the CHNS equations, many researchers mainly proposed corresponding solutions for the nonlinear and coupled terms, such as convex splitting techniques \cite{2012_ShenJie_Second_order_convex_splitting_schemes_for_gradient_flows_with_Ehrlich_Schwoebel_type_energy_application_to_thin_film_epitaxy,2013_Baskaran_Convergence_analysis_of_a_second_order_convex_splitting_scheme_for_the_modified_phase_field_crystal_equation}, stabilization methods \cite{2015_ShenJie_Decoupled_energy_stable_schemes_for_phase_field_models_of_two_phase_incompressible_flows,2018_CaiYongyong_Error_estimates_for_a_fully_discretized_scheme_to_a_Cahn_Hilliard_phase_field_model_for_two_phase_incompressible_flows} , the Invariant Energy Quadratization (IEQ) methods \cite{2017_YangXiaofeng_Numerical_approximations_for_the_molecular_beam_epitaxial_growth_model_based_on_the_invariant_energy_quadratization_method,2017_ZhaoJia_Numerical_approximations_for_a_phase_field_dendritic_crystal_growth_model_based_on_the_invariant_energy_quadratization_approach}, and the Scalar Auxiliary Variable (SAV) methods \cite{2018_Shenjie_Xujie_SAV,2018_Shenjie_Xujie_CAEAFTSAVSYGF}, etc. The SAV method is used for the CHNS equation\cite{2020_LiXiaoli_On_a_SAV_MAC_scheme_for_the_Cahn_Hilliard_Navier_Stokes_phase_field_model_and_its_error_analysis_for_the_corresponding_Cahn_Hilliard_Stokes_case_,2021_JiangNan_Stabilized_scalar_auxiliary_variable_ensemble_algorithms_for_parameterized_flow_problems,2021_LiMinghui_New_efficient_time_stepping_schemes_for_the_Navier_Stokes_Cahn_Hilliard_equations,2022JieShen_LiXiaoliMSAVCHNStwo_phase_incompressible_flows,2022_ChenYaoyao_CHNS_2022_AMC,2023_LiYibao_Consistency_enhanced_SAV_BDF2_time_marching_method_with_relaxation_for_the_incompressible_Cahn_Hilliard_Navier_Stokes_binary_fluid_model,2024_WangCheng_Efficient_finite_element_schemes_for_a_phase_field_model_of_two_phase_incompressible_flows_with_different_densities} to realize decoupled and linear numerical discretization schemes.
In the time semi-discrete scheme, Shen et al. \cite{2010_ShenJie_Energy_stable_schemes_for_Cahn_Hilliard_phase_field_model_of_two_phase_incompressible_flows,2010_ShenJie_A_phase_field_model_and_its_numerical_approximation_for_two_phase_incompressible_flows_with_different_densities_and_viscosities,2015_ShenJie_Decoupled_energy_stable_schemes_for_phase_field_models_of_two_phase_incompressible_flows} proposed a series of linear, first-order, unconditionally energy-stable, weakly decoupled, and fully decoupled schemes for the CHNS equations.
 Li et al. \cite{2022JieShen_LiXiaoliMSAVCHNStwo_phase_incompressible_flows} proposed first-order and second-order time-discrete schemes for the CHNS equations based on multiple SAV methods, but they did not provide the optimal error analysis for the fully discrete scheme. 
More numerical discretization schemes for CHNS type equations can be found in  \cite{2015_Diegel_Analysis_of_a_mixed_finite_element_method_for_a_Cahn_Hilliard_Darcy_Stokes_system,2015_HeYinnian_Unconditional_convergence_of_the_Euler_semi_implicit_scheme_for_the_three_dimensional_incompressible_MHD_equations,2017_CaiYongyong_Error_estimates_for_time_discretizations_of_Cahn_Hilliard_and_Allen_Cahn_phase_field_models_for_two_phase_incompressible_flows,2017_Diegel_Convergence_analysis_and_error_estimates_for_a_second_order_accurate_finite_element_method_for_the_Cahn_Hilliard_Navier_Stokes_system,2019_YangXiaofeng_Convergence_analysis_of_an_unconditionally_energy_stable_projection_scheme_for_magneto_hydrodynamic_equations,2020_WangLiupeng_Error_analysis_of_SAV_finite_element_method_to_phase_field_crystal_model}. The CHNS model is a multivariate, strongly nonlinear and strongly coupled system, and there is cumbersome and difficult error analysis, especially for optimal error estimation, for fully discrete numerical schemes.

Feng et al. \cite{2006_FengXiaobing_FullydiscretefiniteelementapproximationsoftheNavierStokesCahnHilliarddiffuseinterfacemodelfortwophasefluidflows,2007_FengXiaobing_Analysis_of_finite_element_approximations_of_a_phase_field_model_for_two_phase_fluids} proposed fully discrete finite element numerical schemes, which were the nonlinear and coupled system, and analyzed the error estimates. For the error analysis of the fully discrete schemes of the coupled models, Feng et al. \cite{2015_Diegel_Analysis_of_a_mixed_finite_element_method_for_a_Cahn_Hilliard_Darcy_Stokes_system} proposed the convex-splitting finite element method for the Cahn-Hilliard-Darcy-Stokes equations to demonstrate the unconditional unique solvability, unconditional energy stability, and optimal error estimation of the fully discrete finite element scheme in the three-dimensional case. Recently, optimal error estimation of the CHNS-type  systems in discrete scheme has received attention from researchers \cite{2017_Diegel_Convergence_analysis_and_error_estimates_for_a_second_order_accurate_finite_element_method_for_the_Cahn_Hilliard_Navier_Stokes_system,2023_CaiWentao_Optimal_L2_error_estimates_of_unconditionally_stable_finite_element_schemes_for_the_Cahn_Hilliard_Navier_Stokes_system,2024_ChenYaoyao_SAVFEM,2024_YiNianyu_Convergence_analysis_of_a_decoupled_pressure_correction_SAV_FEM_for_the_Cahn_Hilliard_Navier_Stokes_model,2022_WangCheng_A_positivity_preserving_energy_stable_finite_difference_scheme_for_the_Flory_Huggins_Cahn_Hilliard_Navier_Stokes_system,2024_ChenWenbin_Convergence_analysis_of_a_second_order_numerical_scheme_for_the_Flory_Huggins_Cahn_Hilliard_Navier_Stokes_system,2024_WangCheng_A_second_order_numerical_scheme_of_the_Cahn_Hilliard_Navier_Stokes_system_with_Flory_Huggins_potential,2024_Wangcheng_Convergence_analysis_of_a_temporally_second_order_accurate_finite_element_scheme_for_the_Cahn_Hilliard_magnetohydrodynamics_system_of_equations}. Diegel et al. \cite{2017_Diegel_Convergence_analysis_and_error_estimates_for_a_second_order_accurate_finite_element_method_for_the_Cahn_Hilliard_Navier_Stokes_system} extend the analysis to a second-order in time convex-splitting finite element scheme for the CHNS equations, obtaining the same spatial convergence results as those in \cite{2015_Diegel_Analysis_of_a_mixed_finite_element_method_for_a_Cahn_Hilliard_Darcy_Stokes_system}. Diegel et al. \cite{2015_Diegel_Analysis_of_a_mixed_finite_element_method_for_a_Cahn_Hilliard_Darcy_Stokes_system,2017_Diegel_Convergence_analysis_and_error_estimates_for_a_second_order_accurate_finite_element_method_for_the_Cahn_Hilliard_Navier_Stokes_system} both analyze the optimal error estimate for the fully discrete scheme of the CHNS equations, but not the optimal error estimates.
Cai et al. \cite{2023_CaiWentao_Optimal_L2_error_estimates_of_unconditionally_stable_finite_element_schemes_for_the_Cahn_Hilliard_Navier_Stokes_system} proposed the convex-splitting finite element methods, by proposing the two newly proposed quasi-projection operators and utilizing the superconvergence of its negative norms estimates, and in the $P_r\times P_r\times \mathbf{P}_{r+1}\times P_r$ (standard Taylor-Hood elements), where $r$ is the degree of the local polynomials used, and $P_1\times P_1\times \mathbf{P}_{1b}\times P_1$ (MINI elements) finite element spaces,  first proved the optimal $L^2$ error estimate for the phase-field variable, chemical potential, and pressure in the case of $r \geq 1$, and for the velocity when $r \geq 2$. Subsequently, Yang et al. \cite{2024_ChenYaoyao_SAVFEM} used SAV technique, Euler semi-implicit and finite element method to discretize time and space to obtain a linear and unconditional energy stabilization scheme, and analyzed the optimal error estimate in MINI finite element spaces.
Meanwhile, Yang et al.\cite{2024_YiNianyu_Convergence_analysis_of_a_decoupled_pressure_correction_SAV_FEM_for_the_Cahn_Hilliard_Navier_Stokes_model} proposed a step-by-step decoupled, linear and unconditionally energy-stabilized scheme by SAV technique and analyzed the optimal convergence of the standard $P_r\times P_r\times \mathbf{P}_r\times P_{r-1}~(r\geq 2)$ finite element spaces.

In this manuscript, we propose a first-order in time, linear, unconditional energy stabilization, fully discrete scheme.  For time discretization, we use the two scalar auxiliary variables (SAVs) and the pressure-correction projection method and $P_r \times P_r \times \mathbf{P}_{r+1} \times P_r$ finite element spaces in the spatial discretization. We proved the optimal $L^2$ error estimate for the phase-field variable, chemical potential, and pressure in the case of $r \geq 1$, and for the velocity when $r \geq 2$, however, we does not employ the quasi-projection operator proposed in \cite{2023_CaiWentao_Optimal_L2_error_estimates_of_unconditionally_stable_finite_element_schemes_for_the_Cahn_Hilliard_Navier_Stokes_system} and the superconvergence property of  its $H^{-1}$-norm estimates. To the best of our konwledge, there is still a lack of the literature on third-order accuracy for the optimal error estimation of velocity in the $P_1 \times P_1\times \mathbf{P}_2 \times P_1$ finite element spaces. Nevertheless, experimentally, we find that the velocity still achieves the best third-order accuracy in the $\mathbf{P}_2$ element.

This manuscript is organized as follows: we introduce some notations and preliminaries in the 
Section \ref{section_notations_preliminaries_Fully_discrete}, and propose the SAVs fully discrete scheme for the CHNS system, and prove the energy stability. We establish the error analysis of the fully discrete scheme in the Section \ref{section_error_analysis}. Numerical tests are conducted in Section \ref{section_numerical_test} to validate the theoretical results of  our scheme. 
\section{Numerical scheme}\label{section_notations_preliminaries_Fully_discrete}
\subsection{Notations and preliminaries}\label{section_notations_preliminaries}	
In this section we introduce some notations and preliminaries. We denote by $\left(\cdot,\cdot\right)$ the inner products of 
both $L^2(\Omega)$ and $\mathbf{L}^2(\Omega)$, and for all integer $k\geq 0$ and $1\leq p\leq \infty$, set $W^{k,p}(\Omega)$ to be the Sobolev space of functions defined on $\Omega$, and denote 
$$
H^k=W^{k,2}(\Omega),
~L^2(\Omega)=W^{0,p}(\Omega),
$$
and 
$$
L_0^p(\Omega)=\{v\in L^p(\Omega): \int_{\Omega}v~d\mathbf{x}=0\}.
$$
We denote by
$W_0^{k,p}$ the closure of $C_0^{\infty}(\Omega)$ in $W^{k,p}$ space, and
$H_0^k(\Omega)=W_0^{k,p}$ and the relevant vector-valued Sobolev space with the norms $\|\cdot\|_{W^{k,p}}$, 
$\|\cdot\|_{H^k}$, and $\|\cdot\|_{L^p}$ for both scalar- and vector-valued functions define by
$$
\mathbf{W}^{k,p}=[W^{k,p}(\Omega)]^d,~\mathbf{L}^p(\Omega)=[L^p(\Omega)]^d,~
\mathbf{H}_0^1(\Omega)=[H_0^1(\Omega)]^d.
$$
Thus, we denote $\|\cdot\|_k$ as the norms on $H^k$ and $\|\cdot\|$ is equivalent to the norm on $L^2$ space.
We will frequently use the following discrete Gr\"{o}nwall lemma
\cite{1990_Heywood_Finite_element_approximation_of_the_nonstationary_Navier_Stokes_problem_IV_Error_analysis_for_second_order_time_discretization}:
\begin{Lemma}[ \cite{1990_Heywood_Finite_element_approximation_of_the_nonstationary_Navier_Stokes_problem_IV_Error_analysis_for_second_order_time_discretization}]
	\label{lemma_discrete_Gronwall_inequation}
	For all $0 \leq n \leq m$, let $a_n, b_n, c_n, d_n, \tau, C \geq 0$ such that
	$$
		a_m+\tau\sum_{n=0}^{m}b_n\leq \tau\sum_{n=0}^{m}d_na_n+\tau\sum_{n=0}^{m}c_n+C,
	$$
	suppose that $\tau d_n<1$, for all $n$, and set $\sigma_n=(1-\tau d_n)^{-1}$. Then, it holds
	$$
		a_{m}+\tau\sum_{n=0}^{m} b_n \leq \exp \left(\tau\sum_{n=0}^m \sigma_nd_n\right)\left\{a_0+\left(b_0+c_0\right) \tau+\tau\sum_{n=1}^{m} c_n \right\} .
	$$
\end{Lemma}
 Throughout the manuscript we use $C$ or $C$, with or without subscript, to denote a positive constant
independent of discretization parameters that could have uncertain values in different places.

\subsection{Fully discretization scheme}\label{section_fully_discretization_scheme}
We introduce in this section the SAV technique and the fully discrete scheme of the first-order of time. 
Let $\gamma>0$ be a positive constant, $F(\phi)=G(\phi)-\frac{\gamma}{2}\phi^2$, $E_1(\phi)=\int_{\Omega}F(\phi)d\textbf{x}$ and $E_2(\mathbf{u})=\frac{1}{2}\|\mathbf{u}\|^2$, where the term $\frac{\gamma}{2}\phi^2$ is to simply the analysis \cite{2018_Shenjie_Xujie_CAEAFTSAVSYGF}. We introduce two auxiliary variables as follows:
\begin{align}
	r(t)&=\sqrt{E_1(\phi)+C_1},~~  C_1>\gamma,\\
	\rho(t)&=\sqrt{E_2(\mathbf{u})+C_2},~~ C_2>0.
\end{align}
Then
\begin{equation}
	2\rho\frac{d \rho}{d t}=\int_{\Omega} \frac{\partial \mathbf{u}}{\partial t} \cdot \mathbf{u}~d\textbf{x}=\int_{\Omega}\left(\frac{\partial \mathbf{u}}{\partial t}+2\mathbf{u} \cdot \nabla \mathbf{u}\right) \cdot \mathbf{u}~d\textbf{x}-\int_{\partial \Omega}(\mathbf{n} \cdot \mathbf{u})\cdot \frac{1}{2}|\mathbf{u}|^2ds,
\end{equation}
where $\mathbf{n}$ is the outward-pointing unit vector normal to the boundary $\partial \Omega$, and we have used integration by part, the equation \eqref{eq_boundary_initial_conditions_equations}, and the divergence theorem. It should be emphasized that both $q(t)$ and $\rho(t)$ are scalar variables, not field functions. At $t=0$,
\begin{equation}
	\rho(0)=\left(\frac{1}{2}\int_{\Omega} \left|\mathbf{u}\right|^2~d\textbf{x}+C_2\right)^{\frac{1}{2}},
\end{equation}
and reformulate the equation \eqref{eqCHNS01} as
\begin{subequations}\label{eq_1st_dis}
	\begin{align}
		\label{eq_0305a}
		\frac{\partial \phi}{\partial t}+\frac{r}{\sqrt{E_1(\phi)+C_1}}(\mathbf{u}\cdot \nabla)\phi-M\Delta \mu =0\quad \text{in}\quad \Omega\times (0,T],&\\
		\label{eq_0305b}
		\mu+\lambda \Delta \phi-\lambda\gamma\phi- \frac{\lambda r}{\sqrt{E_1(\phi)+C_1}}F'(\phi)=0\quad \text{in}\quad\Omega\times (0,T],&\\
		\label{eq_0305c}
		\frac{dr}{dt}-\frac{1}{2\sqrt{E_1(\phi)+C_1}}\int_{\Omega}F'(\phi)\frac{\partial \phi}{\partial t}d\mathbf{x}=0\quad \text{in}\quad\Omega \times (0,T],&\\
		\label{eq_0305d}
		\frac{\partial \mathbf{u}}{\partial t}+\frac{\rho}{\sqrt{E_2(\mathbf{u})+C_2}}\mathbf{u}\cdot\nabla \mathbf{u}-\nu\Delta \mathbf{u}+\nabla p-\frac{r}{\sqrt{E_1(\phi)+C_1}}\mu\nabla\phi=0\quad \text{in}\quad \Omega\times (0,T],&\\
		\label{eq_0305e}
		\nabla\cdot \mathbf{u}=0\quad \text{in}\quad \Omega\times (0,T],&\\
		\label{eq_0305f}
		2\rho\frac{\partial \rho}{\partial t}-\mathbf{u}\frac{\partial \mathbf{u}}{\partial t} -\frac{2\rho}{\sqrt{E_2(\mathbf{u})+C_2}}\int_{\Omega}\mathbf{u}\cdot\nabla\mathbf{u}\cdot\mathbf{u}~d\mathbf{x}=0\quad\text{in}\quad \Omega\times (0,T].
	\end{align}
\end{subequations}
Thus, we have the dissipation law:
\begin{equation}
	\frac{d\tilde{E}(\phi,\mathbf{u},r,\rho)}{dt}=-M\|\nabla\mu\|^2-\nu\|\nabla \mathbf{u}\|^2,
\end{equation}
where $\tilde{E}(\phi,\mathbf{u},r,\rho)=\int_{\Omega}\left(\frac{1}{4}|\mathbf{u}|^2+\frac{\lambda\gamma}{2}|\phi|^2+\frac{\lambda}{2}|\nabla\phi|^2\right)d\mathbf{x}+\frac{1}{2}\rho^2+\lambda r^2$.
\subsection{Discrete scheme}
We consider the fully discrete mixed finite element method based on the SAVs scheme for solving the CHNS model, 
for $n=0,1,2\cdots,N-1$, and 
denote 
$$
\tau=\frac{T}{N}, ~t^n=n\tau,~\delta_{\tau}g^{n+1}=\frac{g^{n+1}-g^n}{\tau},~
$$
where $N$ is the number of time uniform divided and $\tau$ is the time step.
We denote $\mathfrak{T}_h$ as a uniform partition of $\bar{\Omega}$ into triangles $\mathcal{T}_j$, $(j=1,2,\cdots,M)$, in 
$R^2$ with mesh size $h=\max_{1\leq j\leq M}\{\dim \mathcal{T}_j\}$. For any integer $r\geq 1$, we set the phase field-velocity-pressure finite element spaces
\begin{equation}
	\begin{aligned}
		&S_h^r=\{v_h\in C(\Omega):v_h|_{\mathcal{T}_j}\in P_r(\mathcal{T}_j),\forall \mathcal{T}_j\in \mathfrak{T}_h\},\\
		&\mathring{S}_h^r=S_h^r\cap L_0^2(\Omega),\\
		&\mathbf{X}_h^{r+1}=\{\mathbf{v}_h\in \mathbf{H}_0^1(\Omega)^d:\mathbf{v}_h|_{\mathcal{T}_j}\in \mathbf{P}_{r+1}(\mathcal{T}_j)^d,\forall \mathcal{T}_j\in \mathfrak{T}_h\},
	\end{aligned}
\end{equation}
where $P_r(\mathcal{T}_j)$ is the space of polynomials of degree $r$ on $\mathcal{T}_j$.
In the mixed finite element spaces, we use elements that satisfy the inf-sup condition \cite{1986_Girault_Vivette_Finite_element_methods_for_Navier_Stokes_equations,1984_Brezzi_FortinA_stable_finite_element_for_the_Stokes_equations}:
for all $q_h\in \mathring{S}_h^r$, there exists $\mathbf{v}_h\in \mathbf{X}_h^{r+1}$ and $\mathbf{v}_h\neq 0$ such that
\begin{equation}
	\|q_h\|\leq C\sup \frac{\left(q_h,\nabla\cdot\mathbf{v}_h\right)}{\|\nabla\mathbf{v}_h\|},
\end{equation}
where $C>0$ is a constant independent of $h$ and $\tau$.
 For the simplicity of notations, we denote 
 \begin{equation}
 	\mathcal{X}_h^r:=S_h^r\times S_h^r\times \mathbf{X}_h^{r+1}\times \mathring{S}_h^r.
 \end{equation}
 According to the above notations, we give the following fully discrete finite element SAVs scheme for the CHNS system \eqref{eq_0305a}-\eqref{eq_0305f},
 which is to find $\left(\phi_h^{n+1},\mu_h^{n+1},\mathbf{u}_h^{n+1},p_h^{n+1}\right)\in \mathcal{X}_h^{r}$ such that
 \begin{flalign}
 	\label{eq_fully_discrete_scheme_phi}
 	\left(\delta_{\tau}\phi_h^{n+1},w_h\right)+\frac{r_h^{n+1}}{\sqrt{E_{1h}^{n}}}\left(\mathbf{u}_h^n\cdot\nabla\phi_h^n,w_h\right)
 		+M\left(\nabla\mu_h^{n+1},\nabla w_h\right)&=0,\\
 	\label{eq_fully_discrete_scheme_mu}
 	\left(\mu_h^{n+1},\varphi_h\right)-\lambda\left(\nabla\phi_h^{n+1},\nabla\varphi_h\right)-\lambda\gamma\left(\phi_h^{n+1},\varphi_h\right)
 		-\frac{\lambda r_h^{n+1}}{\sqrt{E_{1h}^n}}\left(F'(\phi_h^n),\varphi_h\right)&=0,\\
 	\label{eq_fully_discrete_scheme_r}
 	\delta_{\tau}r_h^{n+1}-\frac{1}{2\sqrt{E_{1h}^n}}\left(\left(F'(\phi_h^n),\frac{\phi_h^{n+1}-\phi_h^n}{\tau}\right)
 	+\frac{1}{\lambda}\left(\mu_h^{n+1},\mathbf{u}_h^n\cdot\nabla\phi_h^n\right)-\frac{1}{\lambda}\left(\tilde{\mathbf{u}}_h^{n+1},\mu_h^n\nabla\phi_h^n\right)\right)&=0,\\
 	\label{eq_fully_discrete_scheme_u}
 	\left(\frac{\tilde{\mathbf{u}}_h^{n+1}-\mathbf{u}_h^n}{\tau},\mathbf{v}_h\right)
 		+\frac{\rho_h^{n+1}}{\sqrt{E_{2h}^{n}}}\left(\mathbf{u}_h^n\cdot\nabla\mathbf{u}_h^n,\mathbf{v}_h\right)+\nu\left(\nabla \tilde{\mathbf{u}}_h^{n+1},\nabla\mathbf{v}_h\right)+\left(\nabla p_h^n,\mathbf{v}_h\right)\notag&\\
 		-\frac{r_h^{n+1}}{\sqrt{E_{1h}^{n}}}\left(\mu_h^n\nabla\phi_h^n,\mathbf{v}_h\right)&=0,\\
 	\label{eq_fully_discrete_scheme_u_incompressible}
 	\left(\nabla\cdot \mathbf{u}_h^{n+1},q_h\right)&=0,\\
 	\label{eq_fully_discrete_scheme_u_p_correction}
 	\frac{\mathbf{u}_h^{n+1}-\tilde{\mathbf{u}}_h^{n+1}}{\tau}+\nabla\left(p_h^{n+1}-p_h^n\right)&=0,\\
 	\label{eq_fully_discrete_scheme_q}
 	2\rho_h^{n+1}\frac{\rho_h^{n+1}-\rho_h^n}{\tau}-\left(	\frac{\tilde{\mathbf{u}}_h^{n+1}-\mathbf{u}_h^n}{\tau}+\frac{2\rho_h^{n+1}}{\sqrt{E_{2h}^{n}}}\mathbf{u}_h^n\cdot\nabla\mathbf{u}_h^n,\tilde{\mathbf{u}}_h^{n+1}\right)&=0,
 \end{flalign}
where $E_{1h}^n=E_1(\phi_h^n)+C_1$ and $E_{2h}^n=E_1(\mathbf{u}_h^n)+C_2$, hold for all $\left(w_h,\varphi_h,\mathbf{v}_h,q_h\right)\in \mathcal{X}_h^r$, and $n=0,1,2,\cdots,N-1$. 
 Note that we added the terms $$\frac{1}{\lambda}\left(\left(\mu_h^{n+1},\mathbf{u}_h^n\cdot\nabla\phi_h^n\right)-\left(\tilde{\mathbf{u}}_h^{n+1},\mu_h^n\cdot\nabla\phi_h^n\right)\right)
 $$ in equation \eqref{eq_fully_discrete_scheme_r}, which is a first-order approximation to $\left(\mu,~\mathbf{u}\nabla\phi\right)-\left(\mathbf{u},\mu\nabla\phi\right)=0$.
 For continuous case \eqref{eq_0305d}-\eqref{eq_0305f}, we use the first-order pressure-correction scheme \cite{2006_Guermond_ShenJie_An_overview_of_projection_methods_for_incompressible_flows} discretization to obtain \eqref{eq_fully_discrete_scheme_u_incompressible}-\eqref{eq_fully_discrete_scheme_q}. 
 
 We note that the \eqref{eq_fully_discrete_scheme_q} is a nonlinear quadratic equation for $\rho_h^{n+1}$, and therefore, the existence of its solution needs to be given.  We rewrite it in the following form
 	\begin{equation}
 		2(\rho_h^{n+1})^2-2\left(\rho_h^{n}+\frac{\tau\left(\mathbf{u}_h^n\cdot\nabla\mathbf{u}_h^n,\tilde{\mathbf{u}}_h^{n+1}\right)}{\sqrt{E_{2h}^n}}\right)\rho_h^{n+1}-\left(\tilde{\mathbf{u}}_h^{n+1}-\mathbf{u}_h^n,\tilde{\mathbf{u}}_h^{n+1}\right)=0.
 	\end{equation}
 	The above equation can be simplified as 
 	\begin{equation}
 		ax^2+bx+c=0,
 	\end{equation}
 	where the coefficients are
 	\begin{equation}
 		a=2,~b=-2\left(\rho_h^{n}+\frac{\tau\left(\mathbf{u}_h^n\cdot\nabla\mathbf{u}_h^n,\tilde{\mathbf{u}}_h^{n+1}\right)}{\sqrt{E_{2h}^n}}\right),~c=-\left(\tilde{\mathbf{u}}_h^{n+1}-\mathbf{u}_h^n,\tilde{\mathbf{u}}_h^{n+1}\right).
 	\end{equation}
 	Note that $\rho_h^{n+1}$ is an approximation to $\sqrt{E_2(\mathbf{u}^{n+1})+C_2}$, the  $\frac{\rho_h^{n+1}}{\sqrt{E_2(\mathbf{u}_h^{n+1})+C_2}}$ should have a value close to $1$ if the solution accuracy is good enough. Thus formally we obtain $c\rightarrow 0$ and  $b^2-4ac\geq 0$ as $\tau\rightarrow 0$. we can obtain
 	\begin{equation}
 		\begin{aligned}
 			b^2-4ac=&~4\bigg|\left(\rho_h^{n}+\frac{\tau\left(\mathbf{u}_h^n\cdot\nabla\mathbf{u}_h^n,\tilde{\mathbf{u}}_h^{n+1}\right)}{\sqrt{E_{2h}^n}}\right)\bigg|^2+8\left(\tilde{\mathbf{u}}_h^{n+1}-\mathbf{u}_h^n,\tilde{\mathbf{u}}_h^{n+1}\right)\\
 			\geq&~~4\bigg|\left(\rho_h^{n}+\frac{\tau\left(\mathbf{u}_h^n\cdot\nabla\mathbf{u}_h^n,\tilde{\mathbf{u}}_h^{n+1}\right)}{\sqrt{E_{2h}^n}}\right)\bigg|^2\geq 0.
 		\end{aligned}
 	\end{equation}
The nonlinear quadratic equation \eqref{eq_fully_discrete_scheme_q} has two solutions. 
 Therefore, we need to give how to choose a more desirable root. A similar approach has been given in \cite{2020_lixiaoli_New_SAV_MAC_NS,2019_LinLianlei_Numerical_approximation_of_incompressible_Navier_Stokes_equations_based_on_an_auxiliary_energy_variable,2021_LiMinghui_New_efficient_time_stepping_schemes_for_the_Navier_Stokes_Cahn_Hilliard_equations} and we omit the detailed process here.
 Since the exact solution of $\frac{\rho_h^{n+1}}{\sqrt{E_2(\mathbf{u}_h^{n+1})+C_2}}$ is $1$, we choose the root $\rho_h^{n+1}$ such that
 $\frac{\rho_h^{n+1}}{\sqrt{E_2(\mathbf{u}_h^{n+1})+C_2}}$ is closer to $1$.\\
 Next, we define the classic Ritz projection \cite{1973_Wheeler_Mary_Fanett_A_priori_L2_error_estimates_for_Galerkin_approximations_to_parabolic_partial_differential_equations}:
 $R_h:~H^1(\Omega)\rightarrow S_h^r$, for all $\varphi_h\in S_h^r$ with $\int_{\Omega}\left(\psi-R_h\psi\right)dx=0$,
 \begin{equation}
 	\left(\nabla\left(\psi-R_h\psi\right),\nabla\varphi_h\right)=0.
 \end{equation}
 According to the finite element theory \cite{2008_Brenner_Susanne_C_The_mathematical_theory_of_finite_element_methods},
 it holds that
 \begin{flalign}
 	\label{eq_boundness_Ritz_0001}
 	&\|\psi-R_h\psi\|+h\|\psi-R_h\psi\|_{H^{1}}\leq Ch^{r+1}\|\psi\|_{H^{r+1}},\\
 	\label{eq_boundness_Ritz_0002}
 	&\|\psi-R_h\psi\|_{H^{-1}} \leq C\mathcal{E}_h\|\psi\|_{H^{r+1}},\\
 	\label{eq_boundness_Ritz_0003}
 	&\|\delta_{\tau}\left(\psi^n-R_h\psi^n\right)\|+h\|\delta_{\tau}\left(\psi^n-R_h\psi^n\right)\|_{H^1}\leq Ch^{r+1}\|\delta_{\tau}\psi^n\|_{H^{r+1}},\\
 	\label{eq_boundness_Ritz_0004}
 	&\|\delta_{\tau}\left(\psi-R_h\psi\right)\|_{H^{-1}} \leq C\mathcal{E}_h\|\delta_{\tau}\psi\|_{H^{r+1}},
 \end{flalign}
 where $\mathcal{E}_h$ is defined by
 \begin{equation}
 	\label{eq_E_h}
 	\mathcal{E}_h=\begin{cases}
 		h^{r+1},\qquad for ~r=1,\\
 		h^{r+2},\qquad for ~r\geq 2.
 	\end{cases}
 \end{equation}
 We denote the $L^2$ projection operators $I_h:L^2(\Omega)\rightarrow S_h^r$ and $\mathbf{I}_h:\mathbf{L}^2(\Omega)\rightarrow \mathbf{X}_h^{r+1}$ 
 such that
 \begin{flalign}
 	\left(v-I_hv,w_h\right)=&~0, ~\forall w_h\in S_h^r,\\
 	\left(\mathbf{v}-\mathbf{I}_h\mathbf{v},\mathbf{w}_h\right)=&~0,~\forall \mathbf{w}_h\in \mathbf{X}_h^{r+1}.
 \end{flalign}
 It is well-known that the $L^2$ projection satisfies the following estimates:
 \begin{flalign}
 	\label{eq_L2_projection_boundness001}
 	\|v-I_hv\|+h\|\nabla\left(v-I_hv\right)\|\leq Ch^{r+1}\|v\|_{H^{r+1}},\\
 	\label{eq_L2_projection_boundness002}
 	\|\mathbf{v}-\mathbf{I}_h\mathbf{v}\|+h\|\nabla\left(\mathbf{v}-\mathbf{I}_h\mathbf{v}\right)\|\leq Ch^{r+2}\|\mathbf{v}\|_{H^{r+2}}.
 \end{flalign}
 For $\left(\mathbf{v}_h,q_h\right)\in \mathbf{X}_h^{r+1}\times \mathring{S}_h^r$, we recall the Stokes projection
 $\left(\mathbf{P}_h,P_h\right):\mathbf{H}_0^1(\Omega)\times L_0^2(\Omega)\rightarrow\mathbf{X}_h^{r+1}\times \mathring{S}_h^r$,
 \begin{flalign}
 	\label{eq_Stokes_projection_0001}
 	\left(\nabla\left(\mathbf{u}-\mathbf{P}_h(\mathbf{u},p)\right),\nabla\mathbf{v}_h\right)-\left(p-P_h(\mathbf{u},p),\nabla\cdot\mathbf{v}_h\right)=0,\\
 	\label{eq_Stokes_projection_0002}
 	\left(\nabla\cdot\left(\mathbf{u}-\mathbf{P}_h(\mathbf{u},p)\right),q_h\right)=0.
 \end{flalign}
 For brevity, we denote $\mathbf{P}_h(\mathbf{u},p)=\mathbf{P}_h\mathbf{u}$ and $P_h(\mathbf{u},p)=P_hp$.
 \begin{Lemma}[\cite{2015_HeYinnian_Unconditional_convergence_of_the_Euler_semi_implicit_scheme_for_the_three_dimensional_incompressible_MHD_equations}]\label{lemma_Stokes_projection_bound}
 	Suppose that $\left(\mathbf{u},p\right)\in \mathbf{H}^{r+1}(\Omega)\times H^r(\Omega)$, it holds that
 	\begin{flalign}
 		\label{eq_Stokes_operators_boundness_u}
 		\|\mathbf{u}-\mathbf{P}_h\mathbf{u}\|+h\|\nabla\left(\mathbf{u}-\mathbf{P}_h\mathbf{u}\right)\|\leq Ch^{r+2}\left(\|\mathbf{u}\|_{H^{r+2}}+\|p\|_{H^{r+1}}\right),\\
 		\label{eq_Stokes_operators_boundness_p}
 		\|\nabla\left(\mathbf{u}-\mathbf{P}_h\mathbf{u}\right)\|+\|p-P_hp\|\leq Ch^{r+1}\left(\|\mathbf{u}\|_{H^{r+2}}+\|p\|_{H^{r+1}}\right),
 	\end{flalign}
 	and for all $n=1,2,\cdots,N,$
 	\begin{equation}
 		\|\delta_{\tau}\left(\mathbf{u}^n-\mathbf{P}_h\mathbf{u}^n\right)\|\leq Ch^{r+2}\left(\|\delta_{\tau}\mathbf{u}^n\|_{H^{r+2}}+\|\delta_{\tau}p^n\|_{H^{r+1}}\right),
 	\end{equation}
 	where $C$ is positive constant that  does not depend on $\tau$ and $h$.
 \end{Lemma}
 \begin{Lemma}[\cite{2019_YangXiaofeng_Convergence_analysis_of_an_unconditionally_energy_stable_projection_scheme_for_magneto_hydrodynamic_equations,2024_YiNianyu_Convergence_analysis_of_a_decoupled_pressure_correction_SAV_FEM_for_the_Cahn_Hilliard_Navier_Stokes_model}]
 	\label{lemma_Stokes_projection_Qboundedness}For all $\mathbf{u}\in L^{\infty}(0,T;\mathbf{H}^{r+1}(\Omega))$ and $p\in L^{\infty}(0,T;H^{r}(\Omega))$,
 	the Stokes projection $P_hp$ is $H^1$ stable in the sense that
 	\begin{equation}
 		\|P_hp\|_1\leq C\left(\|\mathbf{u}\|_2+\|p\|_1\right),
 	\end{equation}
 	where $C$ is positive constant that  does not depend on $\tau$ and $h$.
 \end{Lemma}
 Then, giving the definition of the discrete Laplace operator $\Delta_h$ and the discrete Stokes operator $A_h=-\mathbf{P}_h\Delta_h$ \cite{1995_Temam_Roger_Navier_Stokes_equations_and_nonlinear_functional_analysis,1995_SheJie_On_error_estimates_of_the_penalty_method_for_unsteady_Navier_Stokes_equations}, it holds that
 \begin{flalign}
 	\label{eq_discrete_Laplace_operator_0001}
 	\left(\nabla v, \nabla\Delta_h^{-1} w_h\right)=-\left(v,w_h\right),~\forall v,w_h\in \mathring{S}_h^{r},\\
 	\label{eq_discrete_Laplace_operator_0002}
 	\left(-\Delta_hw_h,v\right)=\left(\nabla w_h,\nabla v\right),~\forall v,w_h\in \mathring{S}_h^{r},
 \end{flalign}
 where $\|\nabla\Delta_h^{-1} w_h\|=\|\Delta_h^{-1/2}w_h\|$.
 Thus, for all $\mathbf{v}_h\in \mathring{S}_h^r$, it is well-known that the discrete norms are defined by
 \begin{flalign}
 	\label{eq_discrete_laplace_operator_norms_differences_0001}
 	&\|\mathbf{v}_h\|_{H^2}=\|A_h\mathbf{v}_h\|,~\|\mathbf{v}_h\|_{H^{-1}}=\|A_h^{-1/2}\mathbf{v}_h\|, \\
 	\label{eq_discrete_laplace_operator_norms_differences_0002}
 	&\|A_h^{-1/2}\mathbf{v}_h\|=\|\nabla\mathbf{v}_h\|,~\|\nabla A_h^{-1/2}\mathbf{v}_h\|=\|\mathbf{v}_h\|.
 \end{flalign}
 \begin{Lemma}[\cite{2020_Chen_Hongtao_Optimal_error_estimates_for_the_scalar_auxiliary_variable_finite_element_schemes_for_gradient_flows,2024_YiNianyu_Convergence_analysis_of_a_decoupled_pressure_correction_SAV_FEM_for_the_Cahn_Hilliard_Navier_Stokes_model}]\label{lemma_Laplace_opreator}
	For all $\phi_h\in S_h^r$, it holds that, for the opreator $\Delta_h$, 	
	\begin{flalign}
		\|\Delta_h\phi_h\|_{-2}\leq C\|\phi_h\|,\\
		\|\Delta_hv\|_{-1}\leq C\|\nabla v\|.
	\end{flalign}
 \end{Lemma}
 We next give the following inequalities \cite{2007_HeYinnian_SunWeiwei_Stability_and_convergence_of_the_Crank_Nicolson_Adams_Bashforth_scheme_for_the_time_dependent_Navier_Stokes_equations},
 \begin{flalign}
 	\label{eq_boundness_basic_inequalities_0001}
 	&\|v\|_{L^l}\leq C\|\nabla v\|,~ (2\leq l\leq 6),~\|v\|_{L^4}\leq C\|v\|^{1/2}\|\nabla v\|^{1/2}, \quad \forall v\in H_0^1(\Omega),\\
 	\label{eq_boundness_basic_inequalities_0002}
 	&\|v\|_{L^{\infty}}\leq C\|v\|^{1/2}\|\Delta v\|^{1/2},~\|v\|\leq C\|\Delta v\|, \quad \forall v\in H^2(\Omega)\cap H_0^1(\Omega),\\
 	\label{eq_boundness_basic_inequalities_0003}
 	&\|\phi\|_{L^l}\leq C\|\phi\|_{H^1},~ (2\leq l\leq 6), ~\forall \phi\in H^1(\Omega),\\
 	\label{eq_boundness_basic_inequalities_0004}
 	&\|\phi\|_{L^{\infty}}\leq C\|\phi\|^{1/2}\left(\|\phi\|^2+\|\Delta\phi\|^2\right)^{1/4}, ~\forall \phi\in H^2(\Omega),~ \frac{\partial \phi}{\partial n}=0 ~ \text{on}~ \partial\Omega,
 \end{flalign}
 where $C$ is positive constant that depends on $\Omega$. 
 \subsection{ Energy stability}
 Next, we prove that the above scheme \eqref{eq_fully_discrete_scheme_phi}-\eqref{eq_fully_discrete_scheme_q} is unconditionally energy stable.
 \begin{Theorem}
 	The fully discrete scheme \eqref{eq_fully_discrete_scheme_phi}-\eqref{eq_fully_discrete_scheme_q} satisfies the unconditional stabilization of the modified energy in the sense that
 		\begin{equation}\label{eq_energy_stable_inequation0314}
 			\begin{aligned}
 				\tilde{E}(\phi_h^{n+1},\mathbf{u}_h^{n+1},r_h^{n+1},\rho_h^{n+1})&-\tilde{E}(\phi_h^n,\mathbf{u}_h^n,r_h^n,\rho_h^n)\leq-2M\tau\|\nabla\mu_h^{n+1}\|^2-2\tau\nu\|\nabla\tilde{\mathbf{u}}_h^{n+1}\|^2,
 			\end{aligned}
 	\end{equation}
 	where
 	\begin{equation}
 		\begin{aligned}
 			\tilde{E}(\phi_h^{n+1},\mathbf{u}_h^{n+1},r_h^{n+1},\rho_h^{n+1})=&~\lambda\|\nabla\phi_h^{n+1}\|^2+\lambda\gamma\|\phi_h^{n+1}\|^2+2\lambda|r_h^{n+1}|^2+\frac{1}{2}\|\mathbf{u}_h^{n+1}\|^2\\
 			&+	\tau^2\|\nabla p_h^{n+1}\|^2+|\rho_h^{n+1}|^2.
 		\end{aligned}
 	\end{equation}
 \end{Theorem}
 \begin{proof}
 	Taking $w_h=2\tau\mu_h^{n+1}$ in \eqref{eq_fully_discrete_scheme_phi}, $\varphi_h=2\left(\phi_h^{n+1}-\phi_h^n\right)$ in
 	\eqref{eq_fully_discrete_scheme_mu}, respectively,
 	the \eqref{eq_fully_discrete_scheme_r} is multiplied by $4\lambda\tau r_h^{n+1}$ and then we combine  these results to get the following equation:
 	\begin{equation}
 		\label{eq_above_three_terms_results_combined}
 		\begin{aligned}
 			\lambda\left(\|\nabla\phi_h^{n+1}\|^2-\|\nabla\phi_h^n\|^2+\|\nabla\phi_h^{n+1}-\nabla\phi_h^n\|^2\right)+&\\
 			\lambda\gamma\left(\|\phi_h^{n+1}\|^2-\|\phi_h^n\|^2+\|\phi_h^{n+1}-\phi_h^n\|^2\right)+&\\
 			2\lambda\left(|r_h^{n+1}|^2-|r_h^n|^2+|r_h^{n+1}-r_h^n|^2\right)+&\\
 			\frac{2\tau r_h^{n+1}}{\sqrt{E_{1h}^n}}\left(\tilde{\mathbf{u}}_h^{n+1},\mu_h^n\nabla\phi_h^n\right)&=-2M\tau\|\nabla \mu_h^{n+1}\|.
 		\end{aligned}
 	\end{equation}
 	Choosing $\mathbf{v}_h=2\tau\tilde{\mathbf{u}}_h^{n+1}$ in \eqref{eq_fully_discrete_scheme_u}  leads to
 	\begin{equation}
 		\label{eq_u2_0322}
 		\begin{aligned}
 			\|\tilde{\mathbf{u}}_h^{n+1}\|^2-\|\mathbf{{u}}_h^{n}\|^2&+\|\tilde{\mathbf{u}}_h^{n+1}-\mathbf{{u}}_h^{n}\|^2
 			+\frac{2\tau \rho_h^{n+1}}{\sqrt{E_{2h}^n}}\left(\mathbf{u}_h^n\cdot\nabla\mathbf{u}_h^n,\tilde{\mathbf{u}}_h^{n+1}\right)+2\tau\left(\nabla p_h^{n},\tilde{\mathbf{u}}_h^{n+1}\right)\\
 			=&\left(\frac{2\tau r_h^{n+1}}{\sqrt{E_{1h}^n}}\mu^n\nabla\phi_h^n,\tilde{\mathbf{u}}_h^{n+1}\right)-2\nu\tau\|\nabla \tilde{\mathbf{u}}_h^{n+1}\|^2.
 		\end{aligned}
 	\end{equation}
 	Recalling \eqref{eq_fully_discrete_scheme_q}, we have 
 	\begin{equation}\label{eq_q2_0324}
 		|\rho_h^{n+1}|^2-|\rho_h^{n}|^2+|\rho_h^{n+1}-\rho_h^n|^2=\left(\left(\tilde{\mathbf{u}}_h^{n+1}-\mathbf{u}_h^n\right)+\frac{2\tau \rho_h^{n+1}}{\sqrt{E_{2h}^n}}\mathbf{u}_h^n\cdot\nabla\mathbf{u}_h^n,\tilde{\mathbf{u}}_h^{n+1}\right).
 	\end{equation}
 	Combining the above two equations \eqref{eq_u2_0322} and \eqref{eq_q2_0324}, we have
 	\begin{equation}\label{eq_combining_eq_u2_0322_eq_q2_0324}
 		\begin{aligned}
 			&|\rho_h^{n+1}|^2-|\rho_h^{n}|^2+|\rho_h^{n+1}-\rho_h^n|^2+\frac{1}{2}\left(\|\tilde{\mathbf{u}}_h^{n+1}\|^2-\|\mathbf{{u}}_h^{n}\|^2+\|\tilde{\mathbf{u}}_h^{n+1}-\mathbf{{u}}_h^{n}\|^2\right)+2\tau\left(\nabla p_h^{n},\tilde{\mathbf{u}}_h^{n+1}\right)\\
 			=&\left(\frac{2\tau r_h^{n+1}}{\sqrt{E_{1h}^n}}\mu^n\nabla\phi_h^n,\tilde{\mathbf{u}}_h^{n+1}\right)-2\nu\tau\|\nabla \tilde{\mathbf{u}}_h^{n+1}\|^2.
 		\end{aligned}
 	\end{equation}
 	The equation \eqref{eq_fully_discrete_scheme_u_p_correction} can be recasted as
 	\begin{equation}\label{eq_u_dp_0325}
 		\mathbf{u}_h^{n+1}+\tau\nabla p_h^{n+1}=\tilde{\mathbf{u}}_h^{n+1}+\tau \nabla p_h^n.
 	\end{equation}
 	Taking the inner product of \eqref{eq_u_dp_0325} with itself on both sides and noticing that
 	$$\left(\nabla p_h^{n+1}, \mathbf{u}_h^{n+1}\right)=-\left(p_h^{n+1}, \nabla \cdot \mathbf{u}_h^{n+1}\right)=0,$$
 	we have
 	\begin{equation}\label{eq_u_p_0326}
 		\|\mathbf{u}_h^{n+1}\|^2+\tau^2\left\|\nabla p_h^{n+1}\right\|^2=\left\|\tilde{\mathbf{u}}_h^{n+1}\right\|^2+2 \tau\left(\nabla p_h^n, \tilde{\mathbf{u}}_h^{n+1}\right)+\tau^2\left\|\nabla p_h^n\right\|^2.
 	\end{equation}
 	According to \eqref{eq_u_dp_0325}, we can recast by
 	\begin{equation}
 		\tilde{\mathbf{u}}_h^{n+1}=\mathbf{u}_h^{n+1}+\tau \left(\nabla p_h^{n+1}-\nabla p_h^n\right).
 	\end{equation}
 	Thus, we have
 	\begin{equation}\label{eq_recasted_byeq_u_p_0326}
 		\|\tilde{\mathbf{u}}_h^{n+1}\|^2=\|\mathbf{u}_h^{n+1}\|^2+\tau^2\|\nabla \left(p_h^{n+1}- p_h^n\right)\|^2.
 	\end{equation}
 	Combining \eqref{eq_u2_0322} with \eqref{eq_combining_eq_u2_0322_eq_q2_0324}, \eqref{eq_u_p_0326} and \eqref{eq_recasted_byeq_u_p_0326} results in
 	 \begin{equation}\label{eq_E_all_0328}
 	 	\begin{aligned}
 	 		\frac{1}{2}&\left(\|\mathbf{u}_h^{n+1}\|^2-\|\mathbf{u}_h^n\|^2+\|\tilde{\mathbf{u}}_h^{n+1}-\mathbf{u}_h^n\|^2\right)+|\rho_h^{n+1}|^2-|\rho_h^{n}|^2+|\rho_h^{n+1}-\rho_h^n|^2\\
 	 		&+\tau^2\left(\|\nabla p_h^{n+1}\|^2-\|\nabla p_h^n\|^2+\|\nabla p_h^{n+1}-\nabla p_h^n\|^2\right)\\
 	 		=&\left(\frac{2\tau r_h^{n+1}}{\sqrt{E_{1h}^n}}\mu^n\nabla\phi_h^n,\tilde{\mathbf{u}}_h^{n+1}\right)-2\nu\tau\|\nabla \tilde{\mathbf{u}}_h^{n+1}\|^2.
 	 	\end{aligned}
 	 \end{equation}
 	Thus, we can obtain the desired result by combining \eqref{eq_above_three_terms_results_combined} with \eqref{eq_E_all_0328} as follows:
 	\begin{equation}
 		\label{eq_energy_stability_boundness}
 		\begin{aligned}
 			&\lambda\left(\|\nabla\phi_h^{n+1}\|^2-\|\nabla\phi_h^n\|^2+\|\nabla\phi_h^{n+1}-\nabla\phi_h^n\|^2\right)+
 			\lambda\gamma\left(\|\phi_h^{n+1}\|^2-\|\phi_h^n\|^2+\|\phi_h^{n+1}-\phi_h^n\|^2\right)\\
 			&+\frac{1}{2}\left(\|\mathbf{u}_h^{n+1}\|^2-\|\mathbf{{u}}_h^{n}\|^2+\|\tilde{\mathbf{u}}_h^{n+1}-\mathbf{u}_h^n\|^2\right)+
 			|\rho_h^{n+1}|^2-|\rho_h^{n}|^2+	|\rho_h^{n+1}-\rho_h^n|^2\\
 			&+\tau^2\left(\|\nabla p_h^{n+1}\|^2-\|\nabla p_h^n\|^2+\|\nabla p_h^{n+1}-\nabla p_h^n\|^2\right)+
 			2\lambda\left(|r_h^{n+1}|^2-|r_h^n|^2+|r_h^{n+1}-r_h^n|^2\right)\\
 			=&-2M\tau\|\nabla \mu_h^{n+1}\|^2-2\nu\tau\|\nabla \tilde{\mathbf{u}}_h^{n+1}\|^2\leq 0.
 		\end{aligned}
 	\end{equation}
 	This completes the proof.
 \end{proof}
 Next, denoting the modified energy $\tilde{E}_h^{n+1}=\tilde{E}(\phi_h^{n+1},\mathbf{u}_h^{n+1},r_h^{n+1},\rho_h^{n+1})$, we give the following stability results of the scheme \eqref{eq_fully_discrete_scheme_phi}-\eqref{eq_fully_discrete_scheme_q}. 
 
 \begin{Lemma}\label{lemma_boundedness_E_phi_Deltaphi}
 	For all $m\geq 0$, suppose that $\tilde{E}_h^0\leq C_0$, the following results hold
 	\begin{flalign}
 		\label{eq_boundness_E0}
 		|\tilde{E}_h^{m+1}|+2\tau\sum_{n=0}^{m}\left(\nu\|\nabla\tilde{\mathbf{u}}_h^{n+1}\|^2+M\|\nabla\mu_h^{n+1}\|^2\right)\leq C_0,&\\
 		\label{eq_boundness_phi_H1}
 		\|\phi_h^{m+1}\|_{H^1}^2+\tau\sum_{n=0}^{m}\|\mu_h^{n+1}\|^2\leq C\tau,
 	\end{flalign}
 	and
 	\begin{equation}
 		\label{eq_boundness_Laplace_phi}
 		\|\Delta_h\phi_h^{m+1}\|^2+\tau\sum_{n=0}^{m}\|\Delta_h\mu_h^{n+1}\|^2\leq C\tau,
 	\end{equation}
 	where $C$ is positive constant that  does not depend on $\tau$ and $h$.
 \end{Lemma}
 \begin{proof}
 	Summing from $n=0$ to $m$ in \eqref{eq_energy_stability_boundness}, we can derive that the \eqref{eq_boundness_E0} holds. 
 	Letting $w_h=2\tau\phi_h^{n+1}$ in \eqref{eq_fully_discrete_scheme_phi}, 
 	we get
 	\begin{equation}\label{eq_taking_inner_tau_phi}
 		\|\phi_h^{n+1}\|^2-\|\phi_h^n\|^2+\|\phi_h^{n+1}-\phi_h^n\|^2=-2\tau M\left(\nabla\mu_h^{n+1},\nabla\phi_h^{n+1}\right)-\frac{2\tau r_h^{n+1}}{\sqrt{E_{1h}^{n}}}\left(\mathbf{u}_h^n\cdot\nabla\phi_h^n,\phi_h^{n+1}\right).
 	\end{equation}
 	Using the Cauchy-Schwarz inequality, Young inequality, \eqref{eq_boundness_basic_inequalities_0001}, \eqref{eq_boundness_basic_inequalities_0003}
 	and \eqref{eq_boundness_E0}, we estimate the right-hand side of \eqref{eq_taking_inner_tau_phi},
 	\begin{equation}\label{eq_boundness_first_term}
 		\big|-2\tau M\left(\nabla\mu_h^{n+1},\nabla\phi_h^{n+1}\right)\big|\leq C\tau\left(\|\nabla\mu_h^{n+1}\|^2+\|\nabla\phi_h^{n+1}\|^2\right),
 	\end{equation}
 	\begin{equation}\label{eq_boundness_second_term}
 		\begin{aligned}
 			\big|-\frac{2\tau r_h^{n+1}}{\sqrt{E_{1h}^{n}}}\left(\mathbf{u}_h^n\cdot\nabla\phi_h^n,\phi_h^{n+1}\right)\big|
 			\leq &~\frac{C\tau|r_h^{n+1}|}{\sqrt{E_{1h}^{n}}}\|\mathbf{u}_h^n\|_{L^4}\|\nabla\phi_h^n\|\|\phi_h^{n+1}\|_{L^4}\\
 			\leq &~C\tau\|\nabla\phi_h^{n+1}\|^2+C\tau\|\nabla\mathbf{u}_h^n\|^2\left(\|\phi_h^n\|^2+\|\nabla\phi_h^n\|^2\right).
 		\end{aligned}
 	\end{equation}
 	Combining the two inequalities \eqref{eq_boundness_first_term}-\eqref{eq_boundness_second_term} to the \eqref{eq_taking_inner_tau_phi}, and then summing the obtained results from $n=0$ to $m$, we obtain
 	\begin{equation}
 		\begin{aligned}
 			\|\phi_h^{m+1}\|^2\leq \|\phi_h^0\|^2+C\tau\sum_{n=0}^{m}\left(\|\nabla\mu_h^{n+1}\|^2+\|\nabla\phi_h^{n+1}\|^2+\|\nabla\mathbf{u}_h^n\|^2\|\phi_h^n\|^2+\|\nabla\mathbf{u}_h^n\|^2\|\nabla\phi_h^n\|^2\right).
 		\end{aligned}
 	\end{equation}
 	Using the Lemma \ref{lemma_discrete_Gronwall_inequation} and \eqref{eq_boundness_E0}, we have
 	\begin{equation}
 		\label{eq_boundness_phi_L2}
 		\|\phi_h^{m+1}\|^2\leq C.
 	\end{equation}
 	Thus, we have 
 	\begin{equation}
 		\label{eq_boundness_phi_H1001}
 		\|\phi_h^{m+1}\|_{H^1}^2\leq \|\phi_h^{m+1}\|^2+\|\nabla\phi_h^{m+1}\|^2\leq C.
 	\end{equation}
 	Taking 
 	$\varphi_h=2\tau\mu_h^{n+1}$ in
 	 \eqref{eq_fully_discrete_scheme_mu}, we get
 	\begin{equation}
 		\label{eq_taking_inner_product_2tau_mu}
 		\begin{aligned}
 			2\tau\|\mu_h^{n+1}\|^2=&~2\tau\lambda\left(\nabla\phi_h^{n+1},\nabla\mu_h^{n+1}\right)+2\tau\lambda\gamma\left(\phi_h^{n+1},\mu_h^{n+1}\right)
 			+\frac{2\tau\lambda r_h^{n+1}}{\sqrt{E_{1h}^n}}\left(F'(\phi_h^n),\mu_h^{n+1}\right)\\
 			\leq &~C\tau\left(\|\nabla\phi_h^{n+1}\|^2+\|\nabla\mu_h^{n+1}\|^2+\|\phi_h^{n+1}\|^2\right)+\tau\|\mu_h^{n+1}\|^2+|\frac{2C\tau\lambda r_h^{n+1}}{\sqrt{E_{1h}^n}}|^2\|F'(\phi_h^n)\|^2.
 		\end{aligned}
 	\end{equation}
 	Recalling the (2.37) in \cite{2015_Diegel_Analysis_of_a_mixed_finite_element_method_for_a_Cahn_Hilliard_Darcy_Stokes_system} and \eqref{eq_boundness_basic_inequalities_0003}, we can obtain
 	\begin{equation}
 		\label{eq_boundness_phi_nonlinear_term}
 		\begin{aligned}
 			\|F'(\phi_h^n)\|^2=&~\frac{1}{\varepsilon^4}\|(\phi_h^n)^3-\phi_h^n\|^2\leq C\left(\|\phi_h^n\|_{L^6}^6+\|\phi_h^n\|^2\right)\\
 				\leq &~C\left(\|\phi_h^n\|_{H^1}^6+\|\phi_h^n\|^2\right)\leq C.
 		\end{aligned}
 	\end{equation}
 	Thus, we can derive 
 	\begin{equation}
 		\label{eq_combine_above_inequalities}
 		\begin{aligned}
 			\tau\|\mu_h^{n+1}\|^2\leq &~C\tau\left(\|\nabla\phi_h^{n+1}\|^2+\|\nabla\mu_h^{n+1}\|^2+\|\phi_h^{n+1}\|^2\right)+|\frac{2C\tau\lambda r_h^{n+1}}{\sqrt{E_{1h}^n}}|^2\|F'(\phi_h^n)\|^2\\
 			\leq &~C\tau\left(\|\nabla\phi_h^{n+1}\|^2+\|\nabla\mu_h^{n+1}\|^2+\|\phi_h^{n+1}\|^2\right)+C\tau.
 		\end{aligned}
 	\end{equation}
 	Summing the \eqref{eq_combine_above_inequalities} from $n=0$ to $m$, and	using the Lemma \ref{lemma_discrete_Gronwall_inequation}, we get
 	\begin{equation}\label{eq_tau_mu_sum_boundedness}
 		\tau\sum_{n=0}^{m}\|\mu_h^{n+1}\|^2\leq C\left(1+2\tau\right)+C\tau\left(\left(C_1+2C_2\right)^3+C_1\right)\leq C\tau.
 	\end{equation}
 	Combining the \eqref{eq_boundness_phi_H1001} and \eqref{eq_tau_mu_sum_boundedness}, we have the result of \eqref{eq_boundness_phi_H1}.\\ 	
 	We then prove that the \eqref{eq_boundness_Laplace_phi} holds. Letting $\left(w_h,\varphi_h\right)=2\tau\left(\Delta_h\mu_h^{n+1},\Delta_h\delta_{\tau}\phi_h^{n+1}\right)$ in \eqref{eq_fully_discrete_scheme_phi}, \eqref{eq_fully_discrete_scheme_mu}, 
 	and then combining the two results,
 	we can obtain
 	\begin{equation}
 		\label{eq_Delta_phi_boundness}
 		\begin{aligned}
 			&~\lambda\tau\left(\|\Delta_h\phi_h^{n+1}\|^2-\|\Delta_h\phi_h^{n}\|^2+\|\Delta_h\phi_h^{n+1}-\Delta_h\phi_h^{n}\|^2\right)+2\tau M\|\Delta_h\mu_h^{n+1}\|^2\\
 			=&~\tau\lambda\gamma\left(\|\nabla\phi_h^{n+1}\|^2-\|\nabla\phi_h^{n}\|^2+\|\nabla\phi_h^{n+1}-\nabla\phi_h^n\|^2\right)\\
 			&~+\frac{2\tau r_h^{n+1}}{\sqrt{E_{1h}^{n}}}\left(\mathbf{u}_h^n\cdot\nabla\phi_h^n,\Delta_h\mu_h^{n+1}\right)+\frac{2\tau\lambda r_h^{n+1}}{\sqrt{E_{1h}^n}}\left(F'(\phi_h^n),\Delta_h\delta_{\tau}\phi_h^{n+1}\right).
 		\end{aligned}
 	\end{equation}
 	According to the Lemma 2.14 in \cite{2015_Diegel_Analysis_of_a_mixed_finite_element_method_for_a_Cahn_Hilliard_Darcy_Stokes_system}, using the 
 	\eqref{eq_boundness_E0}, we have
 	\begin{equation}
 		\label{eq_the_second_term_of_taking_inner_product_two_comb001}
 		\begin{aligned}
 			\bigg|\frac{2\tau r_h^{n+1}}{\sqrt{E_{1h}^{n}}}\left(\mathbf{u}_h^n\cdot\nabla\phi_h^n,\Delta_h\mu_h^{n+1}\right)\bigg|
 			\leq &~2\tau\frac{|r_h^{n+1}|}{\sqrt{E_{1h}^{n}}}\|\mathbf{u}_h^n\|_{L^4}\|\nabla\phi_h^n\|_{L^4}\|\Delta_h\mu_h^{n+1}\|\\
 			\leq &~\frac{C\tau|r_h^{n+1}|}{\sqrt{E_{1h}^{n}}}\|\nabla\mathbf{u}_h^n\|\left(\|\nabla\phi_h^n\|+\|\Delta_h\phi_h^n\|\right)\|\Delta_h\mu_h^{n+1}\|\\
 			\leq &~\tau M\|\Delta_h\mu_h^{n+1}\|^2+\frac{C\tau}{C_0}\|\nabla\mathbf{u}_h^n\|^2\left(\|\nabla\phi_h^n\|^2+\|\Delta_h\phi_h^n\|^2\right),
 		\end{aligned}
 	\end{equation}
 	and
 	\begin{equation}
 		\label{eq_the_second_term_of_taking_inner_product_two_comb002}
 		\begin{aligned}
 			\bigg|\frac{2\tau\lambda r_h^{n+1}}{\sqrt{E_{1h}^n}}\left(F'(\phi_h^n),\Delta_h\delta_{\tau}\phi_h^{n+1}\right)\bigg|
			\leq &~\frac{2\tau\lambda|r_h^{n+1}|}{\sqrt{E_{1h}^n}}\|F'(\phi_h^n)\|\|\Delta_h\delta_{\tau}\phi_h^{n+1}\|\\
			\leq &~\frac{C\tau}{\sqrt{C_0}}\|F'(\phi_h^n)\|^2+\tau^2\|\Delta_h\delta_{\tau}\phi_h^{n+1}\|^2\\
			\leq &~\frac{C\tau}{\sqrt{C_0}}+\tau^2\|\Delta_h\delta_{\tau}\phi_h^{n+1}\|^2.
 		\end{aligned}
 	\end{equation}
 	Combining the \eqref{eq_the_second_term_of_taking_inner_product_two_comb001}-\eqref{eq_the_second_term_of_taking_inner_product_two_comb002} 
 	with \eqref{eq_Delta_phi_boundness}, we have
 	\begin{equation}
 		\label{eq_combining_the_two}
 		\begin{aligned}
 			&~\tau\lambda\left(\|\Delta_h\phi_h^{n+1}\|^2-\|\Delta_h\phi_h^{n}\|^2+\|\Delta_h\phi_h^{n+1}-\Delta_h\phi_h^{n}\|^2\right)+\tau M\|\Delta_h\mu_h^{n+1}\|^2\\
 			\leq &~\tau\lambda\gamma\left(\|\nabla\phi_h^{n+1}\|^2-\|\nabla\phi_h^{n}\|^2+\|\nabla\phi_h^{n+1}-\nabla\phi_h^n\|^2\right)\\
 			&~+\frac{C\tau}{\sqrt{C_0}}\|\nabla\mathbf{u}_h^n\|^2\left(\|\nabla\phi_h^n\|^2+\|\Delta_h\phi_h^n\|^2\right)+\frac{C\tau}{\sqrt{C_0}}+\tau^2\|\Delta_h\delta_{\tau}\phi_h^{n+1}\|^2.
 		\end{aligned}
 	\end{equation}
 	Summing the \eqref{eq_combining_the_two} from $n=0$ to $m$ and using the Lemma \ref{lemma_discrete_Gronwall_inequation}, we have 
 	\begin{equation}
 		\label{eq_Delta_phi_boundness_result}
 		\|\Delta_h\phi_h^{m+1}\|^2+\tau\sum_{n=0}^{m}\|\Delta_h\mu_h^{n+1}\|^2\leq C\tau.
 	\end{equation}
 	we finish the proof of this lemma.
 \end{proof}
\section{Error analysis}\label{section_error_analysis}
We establish the error analysis of the fully discrete scheme \eqref{eq_fully_discrete_scheme_phi}-\eqref{eq_fully_discrete_scheme_q} in the section. 
We assume that the variables of CHNS model satisfy the following regularity:
\begin{equation}
	\label{eq_varibles_satisfied_regularities}
	\begin{aligned}
		&\phi\in L^\infty\left(0,T;H^{r+1}(\Omega)\right),~
		\phi_t\in L^\infty\left(0,T;H^1(\Omega)\right)\cap L^2\left(0,T;H^1(\Omega)\right),
		\phi_{tt}\in L^2\left(0,T;L^2(\Omega)\right),\\
		&\mu\in L^\infty\left(0,T;H^{r+1}(\Omega)\right),~\mu_t\in L^{\infty}\left(0,T;H^1(\Omega)\right),
		~p\in L^\infty\left(0,T;H^{r}(\Omega)\right),~p_t\in L^2\left(0,T;H^r(\Omega)\right),\\
		&\mathbf{u}\in L^\infty\left(0,T;\mathbf{H}^{r+1}(\Omega)\right),~
		\mathbf{u}_t\in  L^\infty\left(0,T;\mathbf{H}^1(\Omega)\right)\cap L^2\left(0,T;\mathbf{H}^{r+1}(\Omega)\right), ~
		\mathbf{u}_{tt}\in L^2\left(0,T;\mathbf{H}^{-1}(\Omega)\right).
	\end{aligned}
\end{equation}
We give the following notations for the error analysis of the fully discrete scheme,
\begin{equation}
	\begin{aligned}
		&e_{\phi}^{n+1}=\phi^{n+1}-\phi_h^{n+1},~\Phi_{\phi}^{n+1}=\phi^{n+1}-R_h\phi^{n+1},~\Theta_{\phi}^{n+1}=R_h\phi^{n+1}-\phi_h^{n+1},\\
		&e_{\mu}^{n+1}=\mu^{n+1}-\mu_h^{n+1},~\Phi_{\mu}^{n+1}=\mu^{n+1}-R_h\mu^{n+1},~\Theta_{\mu}^{n+1}=R_h\mu^{n+1}-\mu_h^{n+1},\\
		&e_{\tilde{\mathbf{u}}}^{n+1}=\mathbf{u}^{n+1}-\tilde{\mathbf{u}}_h^{n+1},~\Phi_{\mathbf{u}}^{n+1}=\mathbf{u}^{n+1}-\mathbf{P}_h\mathbf{u}^{n+1},~\Theta_{\tilde{\mathbf{u}}}^{n+1}=\mathbf{P}_h\mathbf{u}^{n+1}-\tilde{\mathbf{u}}_h^{n+1},\\
		&e_{\mathbf{u}}^{n+1}=\mathbf{u}^{n+1}-\mathbf{u}_h^{n+1},~\Phi_{\mathbf{u}}^{n+1}=\mathbf{u}^{n+1}-\mathbf{P}_h\mathbf{u}^{n+1},~\Theta_{\mathbf{u}}^{n+1}=\mathbf{P}_h\mathbf{u}^{n+1}-\mathbf{u}_h^{n+1},\\
		&e_p^{n+1}=p^{n+1}-p_h^{n+1},~\Phi_p^{n+1}=p^{n+1}-P_hp^{n+1},~\Theta_p^{n+1}=P_hp^{n+1}-p_h^{n+1},\\
		&e_r^{n+1}=r^{n+1}-r_h^{n+1},~e_{\rho}^{n+1}=\rho^{n+1}-\rho_h^{n+1}.
	\end{aligned}
\end{equation}
Denoting $E_1^{n+1}=E_1[\phi(t^{n+1})]+C_1\leq C_0$, $E_2^{n+1}=E_2[\mathbf{u}(t^{n+1})]+C_2\leq C_0$ and $g^{n+1}=g(t^{n+1})$, we  obtain the variational form of the model \eqref{eq_0305a}-\eqref{eq_0305f} at moment $t^{n+1}$ as follows 
\begin{flalign}
	\label{eq_truncation_error_equations_phi}
	\left(\delta_{\tau}\phi^{n+1},w\right)+\frac{r^{n+1}}{\sqrt{E_1^{n+1}}}\left(\mathbf{u}^{n+1}\cdot\nabla\phi^{n+1},w\right)+M\left(\nabla\mu^{n+1},\nabla w\right)-\left(R_{\phi}^n,w\right)&=0,\\
	\label{eq_truncation_error_equations_mu}
	\left(\mu^{n+1},\varphi\right)-\lambda\left(\nabla\phi^{n+1},\nabla \varphi\right)-\lambda\gamma\left(\phi^{n+1},\varphi\right)
		-\frac{\lambda r^{n+1}}{\sqrt{E_1^{n+1}}}\left(F'(\phi^{n+1}),\varphi\right)&=0,\\
	\label{eq_truncation_error_equations_r}
	\delta_{\tau}r^{n+1}-\frac{1}{2\sqrt{E_1^{n+1}}}\left(F'(\phi^{n+1}),\delta_{\tau}\phi^{n+1}\right)-R_r^n+T_1^n&=0,\\
	\label{eq_truncation_error_equations_u}
	\left(\delta_{\tau}\mathbf{u}^{n+1},\mathbf{v}\right)+\frac{\rho^{n+1}}{\sqrt{E_2^{n+1}}}\left(\mathbf{u}^{n+1}\cdot\nabla\mathbf{u}^{n+1},\mathbf{v}\right)+\nu\left(\nabla\mathbf{u}^{n+1},\nabla\mathbf{v}\right)+\left(\nabla p^{n},\mathbf{v}\right)\notag&\\
	-\frac{r^{n+1}}{\sqrt{E_1^{n+1}}}\left(\mu^{n+1}\nabla\phi^{n+1},\mathbf{v}\right)-\left(R_{\mathbf{u}}^{n},\mathbf{v}\right)&=0,\\
	\label{eq_truncation_error_equations_u_incompressible}
	\left(\nabla\cdot\mathbf{u}^{n+1},q\right)&=0,\\
	\label{eq_truncation_error_equations_u_p_decouple}
	\frac{\mathbf{u}^{n+1}-\mathbf{u}^{n+1}}{\tau}+\nabla\left(p^{n+1}-p^n\right)-R_p^n&=0,\\
	\label{eq_truncation_error_equations_rho}
	\delta_{\tau}\rho^{n+1}-\frac{1}{2\rho^{n+1}}\left(\mathbf{u}^{n+1},\delta_{\tau}\mathbf{u}^{n+1}\right)-\frac{1}{\sqrt{E_2^{n}}}\left(\mathbf{u}^n\cdot\nabla\mathbf{u}^n,\mathbf{u}^n\right)-R_{\rho}^n+T_2^n&=0,
\end{flalign}
where the truncation errors are defined by
\begin{equation}
	\label{eq_error_truncation_R_T}
	\begin{aligned}
		&R_{\phi}^{n}=\frac{1}{\tau}\int_{t^n}^{t^{n+1}}\left(t^{n+1}-t\right)\phi_{tt}(t)dt,~R_{\mathbf{u}}^n=\frac{1}{\tau}\int_{t^n}^{t^{n+1}}\left(t^{n+1}-t\right)\mathbf{u}_{tt}(t)dt,~R_{p}^n=\int_{t^n}^{t^{n+1}}\nabla p_tdt,\\
		&R_r^n=\frac{1}{\tau}\int_{t^n}^{t^{n+1}}\left(t^{n+1}-t\right)r_{tt}dt,~T_1^n=\frac{1}{2\tau r^n}\left(F'(\phi^n),\int_{t^n}^{t^{n+1}}\left(t^{n+1}-t\right)\phi_{tt}dt\right),\\
		&R_{\rho}^n=\frac{1}{\tau}\int_{t^n}^{t^{n+1}}\left(t^{n+1}-t\right)\rho_{tt}dt,~T_2^n=\frac{1}{2\tau \rho^n}\left(\mathbf{u}^n,\int_{t^n}^{t^{n+1}}\left(t^{n+1}-t\right)\mathbf{u}_{tt}dt\right).
	\end{aligned}
\end{equation}
By Taylor expansion, we can obtain the following lemma for truncation errors \eqref{eq_error_truncation_R_T}.
\begin{Lemma}\label{lemma_truncation_boundedness}
	Under the assumption of regularity \eqref{eq_varibles_satisfied_regularities}, for all
	$m\geq 0$, we obtain the boundedness of the truncation errors as follows:
	$$
		\tau\sum_{n=0}^{m}\left(\|R_{\phi}^n\|_{{-1}}^2+\|R_{\mathbf{u}}^n\|_{{-1}}^2+\|R_{p}^n\|_{{-1}}^2+\|R_p^n\|^2+|R_r^n|^2+|R_{\rho}^n|^2+\|T_1^n\|^2+\|T_2^n\|^2\right)\leq C\tau^2,
	$$
	where $C$ is positive constant that  does not depend on $\tau$ and $h$.
\end{Lemma}
\begin{proof}
	By the definitions of truncation errors, we have
	\begin{equation}
		\label{eq_R_phi_boundness_H_neg}
		\begin{aligned}
			\|R_{\phi}^n\|_{{-1}}^2\leq &~\frac{1}{\tau^2}\int_{t^n}^{t^{n+1}}\left(t^{n+1}-t\right)dt\int_{t^n}^{t^{n+1}}\|\phi_{tt}\|_{{-1}}^2dt\\
			\leq &~C\tau\int_{t^n}^{t^{n+1}}\|\phi_{tt}\|_{{-1}}^2dt.
		\end{aligned}
	\end{equation}
	Similarly to \eqref{eq_R_phi_boundness_H_neg}, we derive that
	\begin{flalign}
		\label{eq_R_u_boundness_H_neg}
		&\|R_{\mathbf{u}}^n\|_{{-1}}^2\leq C\tau\int_{t^n}^{t^{n+1}}\|\mathbf{u}_{tt}\|_{{-1}}^2dt,\\
		\label{eq_R_p_boundness_H_neg}
		&\|R_{p}^n\|_{{-1}}^2\leq C\tau\int_{t^n}^{t^{n+1}}\|\nabla p_{t}\|_{{-1}}^2dt,\\
		\label{eq_R_p_boundness_L2}
		&\|R_{p}^n\|^2\leq C\tau\int_{t^n}^{t^{n+1}}\|\nabla p_{t}\|^2dt,
	\end{flalign}
	and
	\begin{flalign}
		\label{eq_T1_L2_boundness}
		&\|T_1^n\|^2\leq C\tau\|F'(\phi^n)\|^2\int_{t^n}^{t^{n+1}}\|\phi_{tt}\|^2dt\leq C\tau\int_{t^n}^{t^{n+1}}\|\phi_{tt}\|^2dt,\\
		\label{eq_T2_L2_boundness}
		&\|T_2^n\|^2\leq C\tau\frac{1}{|\rho^n|}\|\mathbf{u}^n\|^2\int_{t^n}^{t^{n+1}}\|\mathbf{u}_{tt}\|^2dt\leq C\tau\int_{t^n}^{t^{n+1}}\|\mathbf{u}_{tt}\|^2dt.
	\end{flalign}
	By the boundedness of $E_1$, $E_2$, $F'(\phi)$, and $F''(\phi)$, using the \eqref{eq_boundness_basic_inequalities_0003}, we obtain
	\begin{equation}
		\label{eq_boundedness_r_tt}
		\begin{aligned}
			|r_{tt}|^2=&~-\frac{1}{16(E_1(\phi))^3}\left(\int_{\Omega}F'(\phi)\phi_tdx\right)^4
			+\frac{1}{4E_1(\phi)}\left(\int_{\Omega}\left(F''(\phi)\phi_t^2+F'(\phi)\phi_{tt}\right)dx\right)^2\\
			\leq&~C\left(\|F'(\phi)\|_{L^4}^4\|\phi_t\|_{L^4}^4+\|F''(\phi)\|^2\|\phi_t\|_{L^4}^4+\|F'(\phi)\|^2\|\phi_{tt}\|^2\right)\\
			\leq &~C\left(\|\phi_t\|_1^4+\|\phi_{tt}\|^2\right).
		\end{aligned}
	\end{equation}
	And similar to \eqref{eq_boundedness_r_tt}, we have
	\begin{equation}
		\label{eq_boundedness_u_tt}
		\begin{aligned}
			|\rho_{tt}|^2=&~-\frac{1}{(E_2(\mathbf{u}))^3}\left(\int_{\Omega}\mathbf{u}_t\mathbf{u}dx\right)^4
				+\frac{1}{E_2(\mathbf{u})}\left(\int_{\Omega}\left(\mathbf{u}_{tt}\mathbf{u}+\mathbf{u}_t^2\mathbf{u}\right)dx\right)^2
				\leq C\left(\|\mathbf{u}_t\|_1^4+\|\mathbf{u}_{tt}\|^2\right).
		\end{aligned}
	\end{equation}
	Thus, according to the \eqref{eq_boundedness_r_tt} and \eqref{eq_boundedness_u_tt}, we can derive that
	\begin{flalign}
		|R_r^n|^2\leq C\tau\int_{t^n}^{t^{n+1}}|r_{tt}|^2dt\leq C\tau\left(\|\phi_t\|_1^4+\|\phi_{tt}\|^2\right),\\
		|R_{\rho}^n|^2\leq C\tau\int_{t^n}^{t^{n+1}}|\rho_{tt}|^2dt\leq C\tau\left(\|\mathbf{u}_t\|_1^4+\|\mathbf{u}_{tt}\|^2\right).
	\end{flalign}
	So, we can conclude based on the regularity assumption \eqref{eq_varibles_satisfied_regularities}. This proof is complete.
\end{proof}
\subsection{Convergence analysis}
Next, we will derive the error estimation of the parameters based on the above results. Firstly, we need to give the error equations.
Subtracting \eqref{eq_fully_discrete_scheme_phi}-\eqref{eq_fully_discrete_scheme_q} from 
\eqref{eq_truncation_error_equations_phi}-\eqref{eq_truncation_error_equations_rho}, the following error equations are yielded by
\begin{flalign}
	\label{eq_error_equations_phi}
	\left(\delta_{\tau}e_{\phi}^{n+1},w_h\right)+M\left(\nabla e_{\mu}^{n+1},\nabla w_h\right)+\left(\frac{r^{n+1}}{\sqrt{E_1^{n+1}}}\mathbf{u}^{n+1}\cdot\nabla\phi^{n+1}-\frac{r_h^{n+1}}{\sqrt{E_{1h}^{n}}}\mathbf{u}_h^n\cdot\nabla\phi_h^n,w_h\right)&\notag\\
	-\left(R_{\phi}^n,w_h\right)&=0,\\
	\label{eq_error_equations_mu}
	\left(e_{\mu}^{n+1},\varphi_h\right)-\lambda\left(\nabla e_{\phi}^{n+1},\nabla\varphi_h\right)
	-\left(\frac{\lambda r^{n+1}}{\sqrt{E_1^{n+1}}}F'(\phi^{n+1})-\frac{\lambda r_h^{n+1}}{\sqrt{E_{1h}^{n}}}F'(\phi_h^{n}),\varphi_h\right)&\notag\\
	-\lambda\gamma\left(e_{\phi}^{n+1},\varphi_h\right)&=0,\\
	\label{eq_error_equations_r}
	\delta_{\tau}e_r^{n+1}+\frac{1}{2\sqrt{E_{1h}^{n}}}\left(F'(\phi_h^{n}),\delta_{\tau}\phi_h^{n+1}\right)-\frac{1}{2\sqrt{E_1^{n+1}}}\left(F'(\phi^{n+1}),\delta_{\tau}\phi^{n+1}\right)-R_r^n+T_1^n&\notag\\
	+\frac{1}{2\lambda\sqrt{E_{1h}^n}}\left(\left(\mu_h^{n+1},\mathbf{u}_h^n\cdot\nabla\phi_h^n\right)-\left(\tilde{\mathbf{u}}_h^{n+1},\mu_h^n\cdot\nabla\phi_h^n\right)\right)&=0,\\
	\label{eq_error_equations_u}
	\left(\frac{e_{\tilde{\mathbf{u}}}^{n+1}-e_{\mathbf{u}}^n}{\tau},\mathbf{v}_h\right)+\nu\left(\nabla e_{\tilde{\mathbf{u}}}^{n+1},\nabla\mathbf{v}_h\right)-\left(e_p^n,\nabla\cdot\mathbf{v}_h\right)+\left(R_{\mathbf{u}}^n,\mathbf{v}_h\right)&\notag\\
	+\left(\frac{\rho^{n+1}}{\sqrt{E_2^{n+1}}}\mathbf{u}^{n+1}\cdot\nabla\mathbf{u}^{n+1}-\frac{\rho_h^{n+1}}{\sqrt{E_{2h}^n}}\mathbf{u}_h^n\cdot\nabla\mathbf{u}_h^n,\mathbf{v}_h\right)&\notag\\
	-\left(\frac{r^{n+1}}{\sqrt{E_1^{n+1}}}\mu^{n+1}\nabla\phi^{n+1}-\frac{r_h^{n+1}}{\sqrt{E_{1h}^n}}\mu_h^n\phi_h^n,\mathbf{v}_h\right)&=0,\\
	\label{eq_error_equations_incompressible_condition}
	\left(\nabla\cdot e_{\mathbf{u}}^{n+1},q_h\right)&=0,\\
	\label{eq_error_equations_u_p_splitting}
	\frac{e_{\mathbf{u}}^{n+1}-e_{\tilde{\mathbf{u}}}^{n+1}}{\tau}+\nabla\left(e_p^{n+1}-e_p^n\right)-R_p^n&=0,\\
	\label{eq_error_equations_rho}
	\delta_{\tau}e_{\rho}^{n+1}+\frac{1}{\rho_h^{n+1}}\left(\frac{\tilde{\mathbf{u}}_h^{n+1}-\mathbf{u}_h^n}{\tau},\tilde{\mathbf{u}}_h^{n+1}\right)-\frac{1}{2\rho^{n+1}}\left(\delta_{\tau}\mathbf{u}^{n+1},\mathbf{u}^{n+1}\right)-R_{\rho}^n+T_2^n&\notag\\
	+\frac{1}{\sqrt{E_{2h}^n}}\left(\mathbf{u}_h^n\cdot\nabla\mathbf{u}_h^n,\tilde{\mathbf{u}}_h^{n+1}\right)-\frac{1}{\sqrt{E_2^n}}\left(\mathbf{u}^n\cdot\nabla\mathbf{u}^n,\mathbf{u}^n\right)&=0.
\end{flalign}
Based on the above error equations, we can derive the error estimates for the phase field function $\phi$, the chemical potential $\mu$, the velocity $\mathbf{u}$ and pressure $p$. Before we get the optimal $L^2$ error estimates, we need to prove the following lemmas. To simplify the analysis later, we set the parameters $M,\lambda,\gamma,\nu$ to 1.
\begin{Lemma}\label{lemma_Theta_phi_H1}
	Suppose that the system \eqref{eq_1st_dis} has
	a unique solution $\left(\phi,\mu,\mathbf{u},p\right)$ and satisfies the regularity \eqref{eq_varibles_satisfied_regularities}, for all
	$m\geq 0$, we have
	$$
		\begin{aligned}
			\|\Theta_{\phi}^{m+1}\|_1^2&~+\tau\sum_{n=0}^{m}\|e_{\mu}^{n+1}\|_1^2\leq C\left(\tau^2+h^{2r}\right)\\
			&~+C\tau\sum_{n=0}^{m}\left(\left(\|\phi^n\|_2^2+\|\nabla\mathbf{u}_h^n\|^2+\|\nabla\mathbf{u}_h^n\|^2\|\phi_h^n\|_1^2\right)\left(\|\Theta_{\mathbf{u}}^n\|^2+|e_r^{n+1}|^2\right)\right),
		\end{aligned}
	$$
	where $C$ is positive constant that  does not depend on $\tau$ and $h$.
\end{Lemma}
\begin{proof}
	By taking $w_h=2\tau\Theta_{\mu}^{n+1}+2\tau\Theta_{\phi}^{n+1}$ in \eqref{eq_error_equations_phi}, $\varphi_h=2\tau\left(\Theta_{\mu}^{n+1}-\delta_{\tau}\Theta_{\phi}^{n+1}\right)$ in \eqref{eq_error_equations_mu}, respectively,  adding the results obtained, we have
	\begin{equation}
		\label{eq_adding_results}
		\begin{aligned}
			&\|\Theta_{\phi}^{n+1}\|_1^2-\|\Theta_{\phi}^n\|_1^2+\|\Theta_{\phi}^{n+1}-\Theta_{\phi}^{n}\|_1^2+\tau\left(\|\Theta_{\mu}^{n+1}\|_1^2+\|e_{\mu}^{n+1}\|_1^2-\|\Phi_{\mu}^{n+1}\|_1^2\right)\\
			=&~2\tau\left(R_{\phi}^n,\Theta_{\mu}^{n+1}+\Theta_{\phi}^{n+1}\right)-2\tau\left(\frac{r^{n+1}}{\sqrt{E_1^{n+1}}}\mathbf{u}^{n+1}\cdot\nabla\phi^{n+1}-\frac{r_h^{n+1}}{\sqrt{E_{1h}^{n}}}\mathbf{u}_h^n\cdot\nabla\phi_h^n,\Theta_{\mu}^{n+1}+\Theta_{\phi}^{n+1}\right)\\
			&~+2\tau\left(\frac{ r^{n+1}}{\sqrt{E_1^{n+1}}}F'(\phi^{n+1})-\frac{ r_h^{n+1}}{\sqrt{E_{1h}^{n}}}F'(\phi_h^{n}),\Theta_{\mu}^{n+1}-\delta_{\tau}\Theta_{\phi}^{n+1}\right)
			-2\tau\left(e_{\phi}^{n+1},\Theta_{\mu}^{n+1}-\delta_{\tau}\Theta_{\phi}^{n+1}\right)\\
			&~-2\tau\left(\delta_{\tau}\Phi_{\phi}^{n+1},\Theta_{\mu}^{n+1}+\Theta_{\phi}^{n+1}\right)+2\tau\left(\Phi_{\mu}^{n+1},\delta_{\tau}\Theta_{\phi}^{n+1}\right)
			=\sum_{i=1}^{6}Tb_i.
		\end{aligned}
	\end{equation}
	Now, we estimate the right-hand side terms of \eqref{eq_adding_results},
	\begin{equation}
		\label{eq_Tb1_boundedness_term}
		\begin{aligned}
			\big|Tb_1\big|=\big|2\tau\left(R_{\phi}^n,\Theta_{\mu}^{n+1}+\Theta_{\phi}^{n+1}\right)\big|\leq C\tau\|R_{\phi}^n\|_{-1}^2+\frac{\tau}{10}\left(\|\Theta_{\mu}^{n+1}\|_1^2+\|\Theta_{\phi}^{n+1}\|_1^2\right).
		\end{aligned}
	\end{equation}
	According to $r^{n+1}=\sqrt{E_1^{n+1}}$, we have
	\begin{equation}
		\label{eq_Tb2_boundedness_term}
		\begin{aligned}
			\big|Tb_2\big|=&~\bigg|-2\tau\left(\frac{r^{n+1}}{\sqrt{E_1^{n+1}}}\mathbf{u}^{n+1}\cdot\nabla\phi^{n+1}-\frac{r_h^{n+1}}{\sqrt{E_{1h}^{n}}}\mathbf{u}_h^n\cdot\nabla\phi_h^n,\Theta_{\mu}^{n+1}+\Theta_{\phi}^{n+1}\right)\bigg|\\
			=&~2\tau\bigg|\left(\mathbf{u}^{n+1}\cdot\nabla\phi^{n+1}-\mathbf{u}^{n}\cdot\nabla\phi^{n},\Theta_{\mu}^{n+1}+\Theta_{\phi}^{n+1}\right)\\
			&~+\left(\mathbf{u}^{n}\cdot\nabla\phi^{n}-\mathbf{u}_h^{n}\cdot\nabla\phi_h^{n},\Theta_{\mu}^{n+1}+\Theta_{\phi}^{n+1}\right)\\
			&~+\left(\frac{r_h^{n+1}}{\sqrt{E_1^{n+1}}}-\frac{r_h^{n+1}}{\sqrt{E_{1h}^n}}+\frac{e_{r}^{n+1}}{\sqrt{E_{1}^{n+1}}}\right)\left(\mathbf{u}_h^{n}\cdot\nabla\phi_h^{n},\Theta_{\mu}^{n+1}+\Theta_{\phi}^{n+1}\right)\bigg|.
		\end{aligned}
	\end{equation}
	Using the inequalities \eqref{eq_boundness_basic_inequalities_0001}-\eqref{eq_boundness_basic_inequalities_0004}, we can estimate 
	the each term of $Tb_2$,
	\begin{equation}
		\label{eq_boundedness_Tb2_00001}
		\begin{aligned}
			&~2\tau\bigg|\left(\mathbf{u}^{n+1}\cdot\nabla\phi^{n+1}-\mathbf{u}^{n}\cdot\nabla\phi^{n},\Theta_{\mu}^{n+1}+\Theta_{\phi}^{n+1}\right)\bigg|\\
			=&~2\tau\bigg|\left(\left(\mathbf{u}^{n+1}-\mathbf{u}^{n}\right)\cdot\nabla\phi^{n+1}+\mathbf{u}^n\cdot\nabla\left(\phi^{n+1}-\phi^n\right),\Theta_{\mu}^{n+1}+\Theta_{\phi}^{n+1}\right)\bigg|\\
			\leq &~\frac{\tau}{10}\left(\|\Theta_{\mu}^{n+1}\|_1^2+\|\Theta_{\phi}^{n+1}\|_1^2\right)+C\tau^3\left(\|\phi^{n+1}\|_2^2\|\mathbf{u}_t\|_{L^{\infty}(0,T;L^2(\Omega))}^2+\|A\mathbf{u}^n\|^2\|\phi_t\|_{L^{\infty}(0,T;H^1(\Omega))}^2\right),
		\end{aligned}
	\end{equation}
	and
		\begin{equation}
		\label{eq_boundedness_Tb2_00002}
		\begin{aligned}	
			&~2\tau\bigg|\left(\mathbf{u}^{n}\cdot\nabla\phi^{n}-\mathbf{u}_h^{n}\cdot\nabla\phi_h^{n},\Theta_{\mu}^{n+1}+\Theta_{\phi}^{n+1}\right)\bigg|\\
			=&~2\tau\bigg|\left(\left(\mathbf{u}^{n}-\mathbf{u}_h^{n}\right)\cdot\nabla\phi^{n}+\mathbf{u}_h^n\cdot\nabla\left(\phi^n-\phi_h^n\right),\Theta_{\mu}^{n+1}+\Theta_{\phi}^{n+1}\right)\bigg|\\
			\leq&~\frac{\tau}{10}\left(\|\Theta_{\mu}^{n+1}\|_1^2+\|\Theta_{\phi}^{n+1}\|_1^2\right)+C\tau\|\phi^n\|_2^2\left(\|\Phi_{\mathbf{u}}^{n+1}\|^2+\|\Theta_{\mathbf{u}}^{n+1}\|^2\right)\\
			&~+\|\nabla\mathbf{u}_h^n\|^2\left(\|\nabla\Phi_{\phi}^{n}\|^2+\|\nabla\Theta_{\phi}^n\|^2\right).
		\end{aligned}
		\end{equation}
	And then, we have
	\begin{equation}
		\label{eq_boundedness_Tb2_00003_E1_boundedness_term}
		\begin{aligned}
			\frac{1}{\sqrt{E_1^{n+1}}}-\frac{1}{\sqrt{E_{1h}^n}}
			=&~\frac{|E_{1h}^n-E_1^{n+1}|}{\sqrt{E_1^{n+1}E_{1h}^n}\left(\sqrt{E_1^{n+1}}+\sqrt{E_{1h}^n}\right)}\\
			=&~\frac{|E_{1h}^n-E_1^n+E_1^n-E_1^{n+1}|}{\sqrt{E_1^{n+1}E_{1h}^n}\left(\sqrt{E_1^{n+1}}+\sqrt{E_{1h}^n}\right)}
			\leq C\left(\|e_{\phi}^n\| +\tau\|\phi_t\|_{L^{\infty}(0,T;H^1(\Omega))}\right).
		\end{aligned}
	\end{equation}
	Thus, we obtain
	\begin{equation}
		\label{eq_boundedness_Tb2_00003}
		\begin{aligned}
			&2\tau\bigg|\left(\frac{r_h^{n+1}}{\sqrt{E_1^{n+1}}}-\frac{r_h^{n+1}}{\sqrt{E_{1h}^n}}+\frac{e_{r}^{n+1}}{\sqrt{E_{1}^{n+1}}}\right)\left(\mathbf{u}_h^{n}\cdot\nabla\phi_h^{n},\Theta_{\mu}^{n+1}+\Theta_{\phi}^{n+1}\right)\bigg|\\
			\leq ~&\left(C\tau\|e_{\phi}^n\|+C\tau+\frac{2\tau|e_r^{n+1}|}{\sqrt{C_0}}\right)\|\nabla\mathbf{u}_h^n\|\|\nabla\phi_h^n\|\left(\|\Theta_{\mu}^{n+1}\|_1+\|\Theta_{\phi}^{n+1}\|_1\right)\\
			\leq ~&\frac{\tau}{10}\left(\|\Theta_{\mu}^{n+1}\|_1^2+\|\Theta_{\phi}^{n+1}\|_1^2\right)+C\tau\|\nabla\mathbf{u}_h^n\|^2\|\nabla\phi_h^n\|^2\left(\|\Theta_{\phi}^n\|^2+\Phi_{\phi}^n\|^2\right)\\
			&+C\tau\|\nabla\mathbf{u}_h^n\|^2\|\nabla\phi_h^n\|^2|e_r^{n+1}|^2.
		\end{aligned}
	\end{equation}
	According to \cite{2020_Chen_Hongtao_Optimal_error_estimates_for_the_scalar_auxiliary_variable_finite_element_schemes_for_gradient_flows},
	we konw that the following inequalities hold,
	\begin{flalign}
		\label{eq_nonlinear_term_f_boundedness0001}
		\|F'(\phi^{n+1})-F'(\phi_h^{n+1})\|\leq C\|\phi^{n+1}-\phi_h^{n+1}\|=\|e_{\phi}^{n+1}\|,\\
		\label{eq_nonlinear_term_f_boundedness0002}
		\|\nabla F'(\phi^{n+1})-\nabla F'(\phi_h^{n+1})\|\leq C\left(\|e_{\phi}^{n+1}\|+\|\nabla e_{\phi}^{n+1}\|\right).
	\end{flalign}
	Recalling the \eqref{eq_error_equations_phi}, we recast
	\begin{equation}
		\label{eq_delta_tau_Theta_phi_negative_norm_boundedness}
		\begin{aligned}
			\delta_{\tau}\Theta_{\phi}^{n+1}=&~\Delta_he_{\mu}^{n+1}-\delta_{\tau}\Phi_{\phi}^{n+1}-\frac{r^{n+1}}{\sqrt{E_1^{n+1}}}\mathbf{u}^{n+1}\cdot\nabla\phi^{n+1}+\frac{r_h^{n+1}}{\sqrt{E_{1h}^{n}}}\mathbf{u}_h^n\cdot\nabla\phi_h^n+
			R_{\phi}^n.
		\end{aligned}
	\end{equation}
	Thus, we obtain the negative norm of \eqref{eq_delta_tau_Theta_phi_negative_norm_boundedness} as follows,
	\begin{equation}
		\label{eq_delta_tau_Theta_phi_Hneg_boundedness}
		\begin{aligned}
			\|\delta_{\tau}\Theta_{\phi}^{n+1}\|_{-1}\leq &~C\|\nabla e_{\mu}^{n+1}\|+\|R_{\phi}^n\|_{-1}+Ch^{r+1}\|\phi_t\|_{H^r}\\
			&~+C\tau\left(\|\phi^{n+1}\|_2\|\mathbf{u}_{t}\|_{L^{\infty}(0,T;L^2(\Omega))}+\|A\mathbf{u}^n\|\|\phi_t\|_{L^{\infty}(0,T;H^1(\Omega))}\right)\\
			&~+C\tau\left(\|\phi^n\|_2+\|\nabla\mathbf{u}_h^n\|+\|\nabla\mathbf{u}_h^n\|\|\phi_h^n\|_1\right)
			\left(\|\Theta_{\mathbf{u}}^n\|+\|\Phi_{\mathbf{u}}^n\|+\|\Phi_{\phi}^n\|_1+\|\Theta_{\phi}^n\|_1+|e_r^{n+1}|\right).
		\end{aligned}
	\end{equation}
	Using the \eqref{eq_boundness_Ritz_0004}, \eqref{eq_boundedness_Tb2_00003_E1_boundedness_term}, \eqref{eq_nonlinear_term_f_boundedness0001}-\eqref{eq_nonlinear_term_f_boundedness0002} and Lemma \ref{lemma_Laplace_opreator}, we have
	\begin{flalign}
		\big|Tb_3\big|=&~\bigg|2\tau\left(\frac{ r^{n+1}}{\sqrt{E_1^{n+1}}}F'(\phi^{n+1})-\frac{ r_h^{n+1}}{\sqrt{E_{1h}^{n}}}F'(\phi_h^{n}),\Theta_{\mu}^{n+1}-\delta_{\tau}\Theta_{\phi}^{n+1}\right)\bigg|\notag\\
		=&~2\tau\bigg|\left(F'(\phi^{n+1})-F'(\phi^n)+F'(\phi^n)-\frac{ r_h^{n+1}}{\sqrt{E_{1h}^{n}}}F'(\phi_h^{n}),\Theta_{\mu}^{n+1}-\delta_{\tau}\Theta_{\phi}^{n+1}\right)\bigg|\notag\\
		=&~2\tau\bigg|\left(F'(\phi^{n+1})-F'(\phi^n)+F'(\phi^n)-F'(\phi_h^n),\Theta_{\mu}^{n+1}-\delta_{\tau}\Theta_{\phi}^{n+1}\right)\notag\\
		&~+\left(\frac{e_r^{n+1}}{\sqrt{E_1^{n+1}}}+\frac{r_h^{n+1}}{\sqrt{E_1^{n+1}}}-\frac{r_h^{n+1}}{\sqrt{E_{1h}^n}}\right)\left(F'(\phi_h^n),\Theta_{\mu}^{n+1}-\delta_{\tau}\Theta_{\phi}^{n+1}\right)\bigg|\notag\\
		\leq &~C\tau\left(\|F'(\phi^{n+1})-F'(\phi^n)\|+\|F'(\phi^n)-F'(\phi_h^n)\|\right)\|\Theta_{\mu}^{n+1}\|\notag\\
		\label{eq_boundedness_Tb3_terms}
		&~+C\tau\left(\|\nabla\left(F'(\phi^{n+1})-F'(\phi^n)\right)\|+\|\nabla\left(F'(\phi^n)-F'(\phi_h^n)\right)\|\right)\|\delta_{\tau}\Theta_{\phi}^{n+1}\|_{-1}\\
		&~+\left(\frac{2\tau|e_r^{n+1}|}{\sqrt{C_0}}+C\tau\|e_{\phi}^n\|+C\tau\right)\left(\|F'(\phi_h^n)\|\|\Theta_{\mu}^{n+1}\|+\|\nabla F'(\phi_h^n)\|\|\delta_{\tau}\Theta_{\phi}^{n+1}\|_{-1}\right)\notag\\
		\leq &~C\tau\left(\tau\|\phi_t\|_{L^{\infty}(0,T;L^2(\Omega))}+\|e_{\phi}^{n+1}\|\right)\|\Theta_{\mu}^{n+1}\|\notag\\
		&~+C\tau\left(\tau\|\phi_t\|_{L^{\infty}(0,T;L^2(\Omega))}+\|e_{\phi}^{n+1}\|_1\right)\|\delta_{\tau}\Theta_{\phi}^{n+1}\|_{-1}\notag\\
		&~+\left(\frac{2\tau|e_r^{n+1}|}{\sqrt{C_0}}+C\tau\|e_{\phi}^n\|+C\tau\right)\left(\|F'(\phi_h^n)\|\|\Theta_{\mu}^{n+1}\|+\|\nabla F'(\phi_h^n)\|\|\delta_{\tau}\Theta_{\phi}^{n+1}\|_{-1}\right)\notag\\
		\leq &~\frac{\tau}{2}\|\Theta_{\mu}^{n+1}\|^2+\frac{\tau}{2}\|e_{\mu}^{n+1}\|^2+C\tau\|R_{\phi}^n\|_{-1}^2+Ch^{2(r+1)}\|\phi_t\|_{H^r}^2\notag\\
		&~+C\tau^3\left(\|\mathbf{u}_t\|_{L^{\infty}(0,T;L^2(\Omega))}^2\|\phi^{n+1}\|_2^2+\|\phi_t\|_{L^{\infty}(0,T;H^1(\Omega))}^2+\|A\mathbf{u}^n\|^2\|\phi_t\|_{L^{\infty}(0,T;H^1(\Omega))}^2\right)\notag\\
		&~+C\tau\left(\|\phi^n\|_2^2+\|\nabla\mathbf{u}_h^n\|^2+\|\nabla\mathbf{u}_h^n\|^2\|\phi_h^n\|_1^2\right)\left(\|\Theta_{\mathbf{u}}^n\|^2+\|\Phi_{\mathbf{u}}^n\|^2+\|\Phi_{\phi}^n\|_1^2+\|\Theta_{\phi}^n\|_1^2+|e_r^{n+1}|^2\right).\notag
	\end{flalign}
	We next have error estimate of $Tb_4$ as follows
	\begin{equation}
		\label{eq_boundedness_Tb4_terms}
		\begin{aligned}
			\big|Tb_4\big|=&~\bigg|-2\tau\left(e_{\phi}^{n+1},\Theta_{\mu}^{n+1}-\delta_{\tau}\Theta_{\mu}^{n+1}\right)\bigg|\\
			\leq &~C\tau\left(\|\Phi_{\phi}^{n+1}\|+\|\Theta_{\phi}^{n+1}\|\right)\|\Theta_{\mu}^{n+1}\|+C\tau\left(\|\Phi_{\phi}^{n+1}\|_1+\|\Theta_{\phi}^{n+1}\|_1\right)\|\delta_{\tau}\Theta_{\mu}^{n+1}\|_{-1}\\
			\leq &~\frac{\tau}{2}\|\Theta_{\mu}^{n+1}\|^2+C\tau\left(\|\Phi_{\phi}^{n+1}\|^2+\|\Theta_{\phi}^{n+1}\|^2+\|\Phi_{\phi}^{n+1}\|_1^2
			+\|\Theta_{\phi}^{n+1}\|_1^2+\|\delta_{\tau}\Theta_{\mu}^{n+1}\|_{-1}^2\right).
		\end{aligned}
	\end{equation}
According to the \eqref{eq_boundness_Ritz_0002} and Lemma \ref{lemma_Laplace_opreator}, we have error estimate of $Tb_5$ as follows
\begin{equation}
	\label{eq_boundedness_Tb5_terms}
	\begin{aligned}
		\big|Tb_5\big|=&~\big|-2\tau\left(\delta_{\tau}\Phi_{\phi}^{n+1},\Theta_{\mu}^{n+1}+\Theta_{\phi}^{n+1}\right)\big|
		\leq C\tau\left(\|\Theta_{\mu}^{n+1}\|_1+\|\Theta_{\phi}^{n+1}\|_1\right)\|\delta_{\tau}\Phi_{\phi}^{n+1}\|_{-1}\\
		\leq&~\frac{\tau}{10}\left(\|\Theta_{\mu}^{n+1}\|_1^2+\|\Theta_{\phi}^{n+1}\|_1^2\right)+C\tau\|\delta_{\tau}\Phi_{\phi}^{n+1}\|_{-1}^2,
	\end{aligned}
\end{equation}
and
\begin{equation}
	\label{eq_boundedness_Tb6_terms}
	\begin{aligned}
		\big|Tb_6\big|=&~\big|2\tau\left(\Phi_{\mu}^{n+1},\delta_{\tau}\Theta_{\phi}^{n+1}\right)\big|
		\leq C\tau\|\Phi_{\mu}^{n+1}\|_1\|\delta_{\tau}\Theta_{\phi}^{n+1}\|_{-1}\\
		\leq &~\frac{\tau}{2}\|\nabla e_{\mu}^{n+1}\|^2+C\tau\left(\|\Phi_{\mu}^{n+1}\|_1^2+\|R_{\phi}^n\|_{-1}^2\right)+Ch^{2(r+1)}\|\phi_t\|_{H^r}^2\\
		&~+C\tau^3\left(\|\mathbf{u}_t\|_{L^{\infty}(0,T;L^2(\Omega))}^2\|\phi^{n+1}\|_2^2+\|\phi_t\|_{L^{\infty}(0,T;H^1(\Omega))}^2+\|A\mathbf{u}^n\|^2\|\phi_t\|_{L^{\infty}(0,T;H^1(\Omega))}^2\right)\\
		&~+C\tau\left(\|\phi^n\|_2^2+\|\nabla\mathbf{u}_h^n\|^2+\|\nabla\mathbf{u}_h^n\|^2\|\phi_h^n\|_1^2\right)\left(\|\Theta_{\mathbf{u}}^n\|^2+\|\Phi_{\mathbf{u}}^n\|^2+\|\Phi_{\phi}^n\|_1^2+\|\Theta_{\phi}^n\|_1^2+|e_r^{n+1}|^2\right).
	\end{aligned}
\end{equation}
Substituting the \eqref{eq_Tb1_boundedness_term}-\eqref{eq_boundedness_Tb6_terms} into \eqref{eq_adding_results}, summing from $n= 0$ to $m$, we have
\begin{equation}
	\label{eq_combining_Tb1_Tb6_resluts_summing}
	\begin{aligned}
		&\|\Theta_{\phi}^{m+1}\|_1^2+\frac{\tau}{2}\sum_{n=0}^{m}\|e_{\mu}^{n+1}\|_1^2+\frac{\tau}{2}\sum_{n=0}^{m}\|\Theta_{\mu}^{n+1}\|_1^2\leq \|\Theta_{\phi}^0\|_1^2+Ch^{2(r+1)}\|\phi_t\|_{H^r}^2\\
		&
		+C\tau\sum_{n=0}^{m}\left(\|\Theta_{\mu}^{n+1}\|^2+\|\Theta_{\phi}^{n+1}\|_1^2+\|R_{\phi}^n\|_{-1}^2\right)\\
		&+C\tau^3\sum_{n=0}^{m}\left(\|\mathbf{u}_t\|_{L^{\infty}(0,T;L^2(\Omega))}^2\|\phi^{n+1}\|_2^2+\|\phi_t\|_{L^{\infty}(0,T;H^1(\Omega))}^2+\|A\mathbf{u}^n\|^2\|\phi_t\|_{L^{\infty}(0,T;H^1(\Omega))}^2\right)\\
		&+C\tau\sum_{n=0}^{m}\left(\|\phi^n\|_2^2+\|\nabla\mathbf{u}_h^n\|^2+\|\nabla\mathbf{u}_h^n\|^2\|\phi_h^n\|_1^2\right)\left(\|\Theta_{\mathbf{u}}^n\|^2+\|\Phi_{\mathbf{u}}^n\|^2+\|\Phi_{\phi}^n\|_1^2+\|\Theta_{\phi}^n\|_1^2+|e_r^{n+1}|^2\right).
	\end{aligned}
\end{equation}
Using the inequalities \eqref{eq_boundness_Ritz_0001}, \eqref{eq_boundness_Ritz_0003}, Lemmas \ref{lemma_boundedness_E_phi_Deltaphi}, \ref{lemma_truncation_boundedness} and the  discrete Gr\"{o}nwall Lemma \ref{lemma_discrete_Gronwall_inequation}, we obtain the result of the Lemma \ref{lemma_Theta_phi_H1}.
\end{proof}
\begin{Lemma}
	\label{lemma_Theta_u_p_L2}
	Under the assumption of \eqref{eq_varibles_satisfied_regularities}, for all $m\geq 0$, and $e_{\mathbf{u}}^0=e_p^0=0$, we have
	\begin{flalign}
		&\|\Theta_{\mathbf{u}}^{m+1}\|^2+\tau^2\|\nabla\Theta_p^{m+1}\|^2+\tau\sum_{n=0}^{m}\|\nabla\Theta_{\tilde{\mathbf{u}}}^{n+1}\|^2\notag\\
		\leq &~C\left(\tau^2+h^{2r}+\tau\sum_{n=0}^{m}\|e_{\mu}^n\|_1^2\right)+C\tau\sum_{n=0}^{m}\left(\|A\mathbf{u}^{n+1}\|^2+\|A\mathbf{u}^n\|^2+\|\nabla\phi^n\|^2+\|\mu^n\|_1^2\right)\notag\\
		&+C\tau\sum_{n=0}^{m}\left(\|A\mathbf{u}^n\|^2+\|\nabla\mathbf{u}^n\|^2+\|\nabla\mathbf{u}_h^n\|^2\right)\left(\|\Phi_{\mathbf{u}}^{n}\|^2+\|\Theta_{\mathbf{u}}^{n}\|^2\right)\notag\\
		&+C\tau\sum_{n=0}^{m}\left(\|\mathbf{u}_h^n\|^2\|\nabla\mathbf{u}_h^n\|^2\right)\left(|e_{\rho}^{n+1}|^2+\|\Phi_{\mathbf{u}}^n\|^2+\|\Theta_{\mathbf{u}}^n\|^2\right)\notag\\
		&~+C\tau\sum_{n=0}^{m}\left(\|\phi^n\|_2^2+\|\mu_h^n\|_1^2\left(\|\nabla \Phi_{\phi}^n\|^2+\|\nabla\Theta_{\phi}^n\|^2\right)\right)\notag\\
		&~+C\tau\sum_{n=0}^{m}\|\mu_h^n\|_1^2\|\nabla\phi_h^n\|^2\left(|e_{r}^{n+1}|^2+\|\Phi_{\phi}^n\|^2+\|\Theta_{\phi}^n\|^2\right),\notag
	\end{flalign}
	where $C$ is positive constant that  does not depend on $\tau$ and $h$.
\end{Lemma}
\begin{proof}
	We recast the \eqref{eq_error_equations_u_p_splitting} to get
	\begin{equation}
		\label{eq_recast_eq_error_equations_u_p_splitting}
		\frac{\Theta_{\mathbf{u}}^{n+1}}{\tau}+\nabla\Theta_p^{n+1}=\frac{\Theta_{\tilde{\mathbf{u}}}^{n+1}}{\tau}+\nabla\Theta_p^n+R_p^n-\tau\nabla\delta_{\tau}\Phi_p^{n+1}.
	\end{equation}
	Taking the inner product of \eqref{eq_error_equations_u_p_splitting} with itself on both sides, we have
	\begin{equation}
		\label{eq_taking_inner_product_itself}
		\begin{aligned}
			\|\Theta_{\mathbf{u}}^{n+1}\|^2+\tau^2\|\nabla\Theta_p^{n+1}\|^2=&~\|\Theta_{\tilde{\mathbf{u}}}^{n+1}\|^2+\tau^2\|\nabla\Theta_p^{n}\|^2
			+\tau^2\|R_p^n-\tau\nabla\delta_{\tau}\Phi_p^{n+1}\|^2+2\tau\left(\Theta_{\tilde{\mathbf{u}}}^{n+1},\nabla\Theta_p^n\right)\\
			&~+2\tau\left(\Theta_{\tilde{\mathbf{u}}}^{n+1},R_p^n-\tau\nabla\delta_{\tau}\Phi_p^{n+1}\right)+2\tau^2\left(\nabla\Theta_p^n,R_p^n-\tau\nabla\delta_{\tau}\Phi_p^{n+1}\right).
		\end{aligned}
	\end{equation}
	Letting $\mathbf{v}_h=2\tau\Theta_{\tilde{\mathbf{u}}}^{n+1}$ in \eqref{eq_error_equations_u}, 
	combining the \eqref{eq_taking_inner_product_itself}, and using the \eqref{eq_Stokes_projection_0001}-\eqref{eq_Stokes_projection_0002},
	we have
	\begin{equation}
		\label{eq_combining_u_tildeu_result}
		\begin{aligned}
			&\|\Theta_{\mathbf{u}}^{n+1}\|^2-\|\Theta_{\mathbf{u}}^n\|^2+\|\Theta_{\tilde{\mathbf{u}}}^{n+1}-\Theta_{\mathbf{u}}^n\|^2+
			\tau^2\left(\|\nabla\Theta_p^{n+1}\|^2-\|\nabla\Theta_p^n\|^2\right)+2\tau\|\nabla\Theta_{\tilde{\mathbf{u}}}^{n+1}\|^2\\
			=&~\tau^2\|R_p^n-\tau\nabla\delta_{\tau}\Phi_p^{n+1}\|^2+2\tau\left(\Theta_{\tilde{\mathbf{u}}}^{n+1},\delta_{\tau}\Phi_{\mathbf{u}}^{n+1}\right)
				+2\tau\left(R_{\mathbf{u}}^n,\Theta_{\tilde{\mathbf{u}}}^{n+1}\right)-2\tau\left(\Phi_p^{n+1}-\Phi_p^n,\nabla\cdot\Theta_{\tilde{\mathbf{u}}}^{n+1}\right)\\
			&~+2\tau\left(\Theta_{\tilde{\mathbf{u}}}^{n+1},R_p^n-\tau\nabla\delta_{\tau}\Phi_p^{n+1}\right)+2\tau^2\left(\nabla\Theta_p^n,R_p^n-\tau\nabla\delta_{\tau}\Phi_p^{n+1}\right)\\
			&~-2\tau\left(\frac{\rho^{n+1}}{\sqrt{E_2^{n+1}}}\mathbf{u}^{n+1}\cdot\nabla\mathbf{u}^{n+1}-\frac{\rho_h^{n+1}}{\sqrt{E_{2h}^n}}\mathbf{u}_h^n\cdot\nabla\mathbf{u}_h^n,\Theta_{\tilde{\mathbf{u}}}^{n+1}\right)\\
			&~+2\tau\left(\frac{\rho^{n+1}}{E_1^{n+1}}\mu^{n+1}\nabla\phi^{n+1}-\frac{r_h^{n+1}}{\sqrt{E_{1h}^{n}}}\mu_h^n\nabla\phi_h^n,\Theta_{\tilde{\mathbf{u}}}^{n+1}\right)=\sum_{i=1}^{8}T_{ui}.
		\end{aligned}
	\end{equation}
	We now estimate the terms $T_{ui}$ on the right-hand side of the equation \eqref{eq_combining_u_tildeu_result} in turn.
	\begin{flalign}
		\big|T_{u1}\big|=&~\big|\tau^2\|R_p^n-\tau\nabla\delta_{\tau}\Phi_p^{n+1}\|^2\big|\leq \tau^2\|R_p^n\|^2+\tau^4\|\nabla\delta_{\tau}\Theta_p^{n+1}\|^2,\\
		\big|T_{u2}\big|=&~\big|2\tau\left(\Theta_{\tilde{\mathbf{u}}}^{n+1},\delta_{\tau}\Phi_{\mathbf{u}}^{n+1}\right)\big|\leq \frac{\tau}{9}\|\nabla\Theta_{\tilde{\mathbf{u}}}^{n+1}\|^2+C\tau\|\delta_{\tau}\Phi_{\mathbf{u}}^{n+1}\|^2,\\		
		\big|T_{u3}\big|=&~\big|2\tau\left(R_{\mathbf{u}}^n,\Theta_{\tilde{\mathbf{u}}}^{n+1}\right)\big|\leq \frac{\tau}{9}\|\nabla\Theta_{\tilde{\mathbf{u}}}^{n+1}\|^2+C\tau\|R_{\mathbf{u}}^n\|_{-1}^2,\\		
		\big|T_{u4}\big|=&~\big|-2\tau\left(\Phi_p^{n+1}-\Phi_p^n,\nabla\cdot\Theta_{\tilde{\mathbf{u}}}^{n+1}\right)\big|=\big|2\tau^2\left(\delta_{\tau}\Phi_p^{n+1},\nabla\cdot\Theta_{\tilde{\mathbf{u}}}^{n+1}\right)\big|\notag\\
		\leq &~\frac{\tau}{9}\|\nabla\Theta_{\tilde{\mathbf{u}}}^{n+1}\|^2+C\tau^3\|\delta_{\tau}\Phi_p^{n+1}\|^2,
	\end{flalign}
	and
	\begin{equation}
		\begin{aligned}
			\big|T_{u5}+T_{u6}\big|=&~\big|2\tau\left(\Theta_{\tilde{\mathbf{u}}}^{n+1},R_p^n-\tau\nabla\delta_{\tau}\Phi_p^{n+1}\right)+2\tau^2\left(\nabla\Theta_p^n,R_p^n-\tau\nabla\delta_{\tau}\Phi_p^{n+1}\right)\big|\\
			\leq&~\frac{\tau}{9}\|\nabla\Theta_{\tilde{\mathbf{u}}}^{n+1}\|^2+C\tau^3\|\nabla\Theta_p^{n}\|^2+C\tau\|R_p^n-\tau\nabla\delta_{\tau}\Phi_p^{n+1}\|^2,\\
			\leq &~\frac{\tau}{9}\|\nabla\Theta_{\tilde{\mathbf{u}}}^{n+1}\|^2+C\tau^3\|\nabla\Theta_p^{n}\|^2+C\tau\left(\|R_p^n\|^2+\tau^2\|\nabla\delta_{\tau}\Phi_p^{n+1}\|^2\right),
		\end{aligned}
	\end{equation}
	and using the equation $\rho^{n+1}=\sqrt{E_2^{n+1}}$, we have
	\begin{equation}
		\label{eq_boundedness_Tr7}
		\begin{aligned}
			\big|T_{u7}\big|=&~\bigg|-2\tau\left(\frac{\rho^{n+1}}{\sqrt{E_2^{n+1}}}\mathbf{u}^{n+1}\cdot\nabla\mathbf{u}^{n+1}-\frac{\rho_h^{n+1}}{\sqrt{E_{2h}^n}}\mathbf{u}_h^n\cdot\nabla\mathbf{u}_h^n,\Theta_{\tilde{\mathbf{u}}}^{n+1}\right)\bigg|\\
			=&~2\tau\bigg|\left(\mathbf{u}^{n+1}\cdot\nabla\mathbf{u}^{n+1}-\mathbf{u}^n\cdot\nabla\mathbf{u}^n+\mathbf{u}^n\cdot\nabla\mathbf{u}^n
			-\mathbf{u}_h^n\cdot\nabla\mathbf{u}_h^n,\Theta_{\tilde{\mathbf{u}}}^{n+1}\right)\\
			&~+\left(\frac{e_{\rho}^{n+1}}{\sqrt{E_{2}^{n+1}}}+\frac{\rho_h^{n+1}}{\sqrt{E_2^{n+1}}}-\frac{\rho_h^{n+1}}{\sqrt{E_{2h}^n}}\right)\left(\mathbf{u}_h^n\cdot\nabla\mathbf{u}_h^n,\Theta_{\tilde{\mathbf{u}}}^{n+1}\right)\bigg|.
		\end{aligned}
	\end{equation}
	We next give the boundedness of each terms on the right-hand side of \eqref{eq_boundedness_Tr7},
	\begin{equation}
		\label{eq_boundedness_Tr7_sub0001}
		\begin{aligned}
			&~2\tau\big|\left(\mathbf{u}^{n+1}\cdot\nabla\mathbf{u}^{n+1}-\mathbf{u}^n\cdot\nabla\mathbf{u}^n,\Theta_{\tilde{\mathbf{u}}}^{n+1}\right)\big|\\
			=&~2\tau\big|\left(\left(\mathbf{u}^{n+1}-\mathbf{u}^n\right)\cdot\nabla\mathbf{u}^{n+1}+\mathbf{u}^n\cdot\nabla\left(\mathbf{u}^{n+1}-\mathbf{u}^n\right),\Theta_{\tilde{\mathbf{u}}}^{n+1}\right)\big|\\
			\leq &~\frac{\tau}{9}\|\nabla\Theta_{\tilde{\mathbf{u}}}^{n+1}\|^2+C\tau^3\|\mathbf{u}_t\|_{L^{\infty}(0,T;L^2(\Omega))}^2\left(\|A\mathbf{u}^{n+1}\|^2+\|A\mathbf{u}^n\|^2\right),
		\end{aligned}
	\end{equation}
	and
	\begin{equation}
		\label{eq_boundedness_Tr7_sub0002}
		\begin{aligned}
			&2\tau\big|\left(\mathbf{u}^n\cdot\nabla\mathbf{u}^n
			-\mathbf{u}_h^n\cdot\nabla\mathbf{u}_h^n,\Theta_{\tilde{\mathbf{u}}}^{n+1}\right)\big|\\
			=&~2\tau\big|\left(\left(\mathbf{u}^n-\mathbf{u}_h^n\right)\cdot\nabla\mathbf{u}^n+\mathbf{u}_h^n\cdot\nabla\left(\mathbf{u}^n-\mathbf{u}_h^n\right),\Theta_{\tilde{\mathbf{u}}}^{n+1}\right)\big|\\
			=&~2\tau\big|\left(e_{\mathbf{u}}^n\cdot\nabla\mathbf{u}^n+\mathbf{u}_h^n\cdot\nabla e_{\mathbf{u}}^n,\Theta_{\tilde{\mathbf{u}}}^{n+1}\right)\big|\\
			\leq &~\frac{\tau}{9}\|\nabla\Theta_{\tilde{\mathbf{u}}}^{n+1}\|^2+C\tau\left(\|A\mathbf{u}^n\|^2+\|\nabla\mathbf{u}^n\|^2+\|\nabla\mathbf{u}_h^n\|^2\right)\left(\|\Phi_{\mathbf{u}}^{n}\|^2+\|\Theta_{\mathbf{u}}^{n}\|^2\right).
		\end{aligned}
	\end{equation}
	According to the \eqref{eq_boundedness_Tb2_00003_E1_boundedness_term}, and using a similar approach, we can obtain the following results:
	\begin{equation}
		\label{eq_boundedness_Tr7_00003_E2_boundedness_term}
		\begin{aligned}
			\bigg|\frac{1}{\sqrt{E_2^{n+1}}}-\frac{1}{\sqrt{E_{2h}^n}}\bigg|
			=&~\frac{|E_{2h}^n-E_2^{n+1}|}{\sqrt{E_2^{n+1}}\sqrt{E_{2h}^n}\left(\sqrt{E_2^{n+1}}+\sqrt{E_{2h}^n}\right)}\\
			=&~\frac{|E_{2h}^n-E_2^n+E_2^n-E_2^{n+1}|}{\sqrt{E_2^{n+1}}\sqrt{E_{2h}^n}\left(\sqrt{E_2^{n+1}}+\sqrt{E_{2h}^n}\right)}
			\leq C\left(\|e_{\mathbf{u}}^n\|+\tau\|\mathbf{u}_t\|_{L^{\infty}(0,T;L^2(\Omega))}\right).
		\end{aligned}
	\end{equation}
	We then obtain
	\begin{equation}
		\label{eq_boundedness_Tr7_sub0003}
		\begin{aligned}
			&2\tau\bigg|\left(\frac{e_{\rho}^{n+1}}{\sqrt{E_{2}^{n+1}}}+\frac{\rho_h^{n+1}}{\sqrt{E_2^{n+1}}}-\frac{\rho_h^{n+1}}{\sqrt{E_{2h}^n}}\right)\left(\mathbf{u}_h^n\cdot\nabla\mathbf{u}_h^n,\Theta_{\tilde{\mathbf{u}}}^{n+1}\right)\bigg|\\
			\leq &~C\tau\left(\frac{|e_{\rho}^{n+1}|}{\sqrt{C_0}}+\|e_{\mathbf{u}}^n\|+\tau\|\mathbf{u}_t\|_{L^{\infty}(0,T;L^2(\Omega))}\right)\|\mathbf{u}_h^n\|\|\nabla\mathbf{u}_h^n\|\|\nabla\Theta_{\tilde{\mathbf{u}}}^{n+1}\|\\
			\leq &~\frac{\tau}{9}\|\nabla\Theta_{\tilde{\mathbf{u}}}^{n+1}\|^2+C\tau\|\mathbf{u}_h^n\|^2\|\nabla\mathbf{u}_h^n\|^2\left(|e_{\rho}^{n+1}|^2+\|\Phi_{\mathbf{u}}^n\|^2+\|\Theta_{\mathbf{u}}^n\|^2+\tau\|\mathbf{u}_t\|_{L^{\infty}(0,T;L^2(\Omega))}^2\right).
		\end{aligned}
	\end{equation}
	Thus, combining the \eqref{eq_boundedness_Tr7_sub0001}-\eqref{eq_boundedness_Tr7_sub0003} with \eqref{eq_boundedness_Tr7}, we obtain the boundedness of term $T_{u7}$ as follows,
	\begin{equation}
		\begin{aligned}
			\big|T_{u7}\big|\leq&~ \frac{\tau}{9}\|\nabla\Theta_{\tilde{\mathbf{u}}}^{n+1}\|^2+C\tau^3\|\mathbf{u}_t\|_{L^{\infty}(0,T;L^2(\Omega))}^2\left(\|A\mathbf{u}^{n+1}\|^2+\|A\mathbf{u}^n\|^2\right)\\
			&+C\tau\left(\|A\mathbf{u}^n\|^2+\|\nabla\mathbf{u}^n\|^2+\|\nabla\mathbf{u}_h^n\|^2\right)\left(\|\Phi_{\mathbf{u}}^{n}\|^2+\|\Theta_{\mathbf{u}}^{n}\|^2\right)\\
			&+C\tau\left(\|\mathbf{u}_h^n\|^2\|\nabla\mathbf{u}_h^n\|^2\right)\left(|e_{\rho}^{n+1}|^2+\|\Phi_{\mathbf{u}}^n\|^2+\|\Theta_{\mathbf{u}}^n\|^2+\tau^2\|\mathbf{u}_t\|_{L^{\infty}(0,T;L^2(\Omega))}^2\right).
		\end{aligned}
	\end{equation}
	Using the \eqref{eq_boundedness_Tb2_00003_E1_boundedness_term},	We next give the boundedness of the term $T_{u8}$,
	\begin{flalign}\big|T_{u8}\big|=&~\bigg|2\tau\left(\frac{\rho^{n+1}}{E_1^{n+1}}\mu^{n+1}\nabla\phi^{n+1}-\frac{r_h^{n+1}}{\sqrt{E_{1h}^{n}}}\mu_h^n\nabla\phi_h^n,\Theta_{\tilde{\mathbf{u}}}^{n+1}\right)\bigg|\notag\\
		= &~2\tau\bigg|\left(\mu^{n+1}\nabla\phi^{n+1}-\mu^n\nabla\phi^n+\mu^n\nabla\phi^n-\mu_h^n\nabla\phi_h^n,\Theta_{\tilde{\mathbf{u}}}^{n+1}\right)\notag\\
		&~+\left(\frac{e_r^{n+1}}{\sqrt{E_1^{n+1}}}+\frac{r_h^{n+1}}{\sqrt{E_1^{n+1}}}-\frac{r_h^{n+1}}{\sqrt{E_{1h}^n}}\right)\left(\mu_h^n\nabla\phi_h^n,\Theta_{\tilde{\mathbf{u}}}^{n+1}\right)\bigg|\notag\\
		= &~2\tau\big|\left(\left(\mu^{n+1}-\mu^n\right)\nabla\phi^{n+1}+\mu^n\nabla\left(\phi^{n+1}-\phi^n\right)+e_{\mu}^n\nabla\phi^n+\mu_h^n\nabla e_{\phi}^n,\Theta_{\tilde{\mathbf{u}}}^{n+1}\right)\notag\\
		\label{eq_boundedness_Tr8}
		&~+\left(\frac{e_r^{n+1}}{\sqrt{E_1^{n+1}}}+\frac{r_h^{n+1}}{\sqrt{E_1^{n+1}}}-\frac{r_h^{n+1}}{\sqrt{E_{1h}^n}}\right)\left(\mu_h^n\nabla\phi_h^n,\Theta_{\tilde{\mathbf{u}}}^{n+1}\right)\bigg|\\
		\leq &~C\tau\left(\tau^2\|\mu_t\|_{L^{\infty}(0,T;H^1(\Omega))}\|\nabla\phi^{n+1}\|+\tau^2\|\mu^{n}\|_1\|\phi_t\|_{L^{\infty}(0,T;H^1(\Omega))}\right)\|\nabla\Theta_{\tilde{\mathbf{u}}}^{n+1}\|\notag\\
		&~+C\tau\left(\|e_{\mu}^n\|_{L^4}\|\nabla\phi^n\|_{L^4}\|\nabla\Theta_{\tilde{\mathbf{u}}}^{n+1}\|+\|\mu_h^n\|_{L^4}\|\nabla e_{\phi}^n\|\|\Theta_{\tilde{\mathbf{u}}}^{n+1}\|_{L^\infty}\right)\notag\\
		&~+\left(\frac{C\tau|e_r^{n+1}|}{\sqrt{C_0}}+\|e_{\phi}^n\|+\tau\|\phi_t\|_{L^{\infty}(0,T;H^1(\Omega))}\right)\|\mu_h^n\|_1\|\nabla\phi_h^n\|\|\nabla\Theta_{\tilde{\mathbf{u}}}^{n+1}\|\notag\\
		\leq &~\frac{\tau}{9}\|\nabla\Theta_{\tilde{\mathbf{u}}}^{n+1}\|^2+C\tau^3\left(\|\mu_t\|_{L^{\infty}(0,T;H^1(\Omega))}^2\|\nabla\phi^{n+1}\|^2+\|\mu^{n}\|_1^2\|\phi_t\|_{L^{\infty}(0,T;H^1(\Omega))}^2\right)\notag\\
		&~+C\tau\|e_{\mu}^n\|_1^2\|\phi^n\|_2^2+C\tau\|\mu_h^n\|_1^2\left(\|\nabla \Phi_{\phi}^n\|^2+\|\nabla\Theta_{\phi}^n\|^2\right)\notag\\
		&~+C\tau\|\mu_h^n\|_1^2\|\nabla\phi_h^n\|^2\left(|e_{r}^{n+1}|^2+\|\Phi_{\phi}^n\|^2+\|\Theta_{\phi}^n\|^2+\tau^2\|\phi_t\|_{L^{\infty}(0,T;H^1(\Omega))}^2\right).\notag
	\end{flalign}
	Combining the boundedness of terms $T_{u1}\sim T_{u8}$ with \eqref{eq_combining_u_tildeu_result}, summing up $n$ from $0$ to $m$, we have
	\begin{flalign}
		&\|\Theta_{\mathbf{u}}^{m+1}\|^2+\tau^2\|\nabla\Theta_p^{m+1}\|^2+\sum_{n=0}^{m}\|\Theta_{\tilde{\mathbf{u}}}^{n+1}-\Theta_{\mathbf{u}}^n\|^2+\tau\sum_{n=0}^{m}\|\nabla\Theta_{\tilde{\mathbf{u}}}^{n+1}\|^2\notag\\
		\leq &~\|\Theta_{\mathbf{u}}^{0}\|^2+\tau^2\|\nabla\Theta_p^{0}\|^2+C\tau\sum_{n=0}^{m}\left(\|R_p^n\|^2+\|R_{\mathbf{u}}^n\|_{-1}^2+\|\delta_{\tau}\Phi_{\mathbf{u}}^{n+1}\|^2\right)\notag\\
		&~+C\tau^3\sum_{n=0}^{m}\|\delta_{\tau}\Phi_p^{n+1}\|_1^2+C\tau^3\sum_{n=0}^{m}\|\mathbf{u}_t\|_{L^{\infty}(0,T;L^2(\Omega))}^2\left(\|A\mathbf{u}^{n+1}\|^2+\|A\mathbf{u}^n\|^2\right)\notag\\
		&+C\tau\sum_{n=0}^{m}\left(\|A\mathbf{u}^n\|^2+\|\nabla\mathbf{u}^n\|^2+\|\nabla\mathbf{u}_h^n\|^2\right)\left(\|\Phi_{\mathbf{u}}^{n}\|^2+\|\Theta_{\mathbf{u}}^{n}\|^2\right)\notag\\
		&+C\tau\sum_{n=0}^{m}\left(\|\mathbf{u}_h^n\|^2\|\nabla\mathbf{u}_h^n\|^2\right)\left(|e_{\rho}^{n+1}|^2+\|\Phi_{\mathbf{u}}^n\|^2+\|\Theta_{\mathbf{u}}^n\|^2+\tau^2\|\mathbf{u}_t\|_{L^{\infty}(0,T;L^2(\Omega))}^2\right)\notag\\
		&~+C\tau^3\sum_{n=0}^{m}\left(\|\mu_t\|_{L^{\infty}(0,T;H^1(\Omega))}^2\|\nabla\phi^{n+1}\|^2+\|\mu^{n}\|_1^2\|\phi_t\|_{L^{\infty}(0,T;H^1(\Omega))}^2\right)\notag\\
		&~+C\tau\sum_{n=0}^{m}\left(\|e_{\mu}^n\|_1^2\|\phi^n\|_2^2+\|\mu_h^n\|_1^2\left(\|\nabla \Phi_{\phi}^n\|^2+\|\nabla\Theta_{\phi}^n\|^2\right)\right)\notag\\
		&~+C\tau\sum_{n=0}^{m}\|\mu_h^n\|_1^2\|\nabla\phi_h^n\|^2\left(|e_{r}^{n+1}|^2+\|\Phi_{\phi}^n\|^2+\|\Theta_{\phi}^n\|^2+\tau^2\|\phi_t\|_{L^{\infty}(0,T;H^1(\Omega))}^2\right).\notag
	\end{flalign}
	According to Lemma \ref{lemma_Stokes_projection_bound} and \ref{lemma_Stokes_projection_Qboundedness},
	we have 
	\begin{equation}
		\tau^3\sum_{n=0}^{m}\|\delta_{\tau}\Phi_p^{n+1}\|_1^2\leq C\tau^2\int_{0}^{T}\left(\|\mathbf{u}\|_2^2+\|p_t\|_1^2\right)dt,
	\end{equation}
	and
	\begin{equation}
		\tau\sum_{n=0}^{m}\|\delta_{\tau}\Phi_{\mathbf{u}}^{n+1}\|^2\leq Ch^{2(r+1)}\int_{0}^{T}\left(\|\mathbf{u}_t\|_{r+1}^2+\|p_t\|_r^2\right).
	\end{equation}
	Thus, we complete the proof of this lemma based on the above inequalities, Lemma \ref{lemma_discrete_Gronwall_inequation}, \ref{lemma_Stokes_projection_bound}, \ref{lemma_Stokes_projection_Qboundedness}, \ref{lemma_boundedness_E_phi_Deltaphi} and
	\ref{lemma_truncation_boundedness}.
\end{proof}
\begin{Lemma}
	\label{lemma_boundedness_er_rho}
	Under the assumption of \eqref{eq_varibles_satisfied_regularities}, for all $m\geq 0$, and $e_{r}^0=e_{\rho}^0=0$, we have
	\begin{equation}\label{eq_r_boundedness_error_estimate}
		|e_r^{n+1}|^2\leq C\left(\tau^2+h^{2(r+1)}\right),
	\end{equation}
	and
	\begin{equation}\label{eq_rho_boundedness_error_estimate}
		|e_{\rho}^{n+1}|^2\leq C\left(\tau^2+h^{2(r+1)}\right),
	\end{equation}
	where $C$ is positive constant that  does not depend on $\tau$ and $h$.
\end{Lemma}
\begin{proof}
Multiplying \eqref{eq_error_equations_r} by $2\tau e_r^{n+1}$, we obtain
\begin{equation}
	\label{eq_multiply_2tau_ern1}
	\begin{aligned}
		&|e_{r}^{n+1}|^2-|e_r^n|^2+|e_r^{n+1}-e_r^n|^2\\
		=&~\frac{\tau e_r^{n+1}}{\sqrt{E_1^{n+1}}}\left(F'(\phi^{n+1}),\delta_{\tau}\phi^{n+1}\right)-\frac{\tau e_r^{n+1}}{\sqrt{E_{1h}^{n}}}\left(F'(\phi_h^{n}),\delta_{\tau}\phi_h^{n+1}\right)+2\tau e_r^{n+1}\left(R_r^n-T_1^n\right)\\
		&~-\frac{\tau e_r^{n+1}}{\sqrt{E_{1h}^n}}\left(\left(\mu_h^{n+1},\mathbf{u}_h^n\cdot\nabla\phi_h^n\right)-\left(\tilde{\mathbf{u}}_h^{n+1},\mu_h^n\cdot\nabla\phi_h^n\right)\right)=\sum_{i=1}^{3}T_{ri}.
	\end{aligned}
\end{equation}
We can first derive the following result by \eqref{eq_boundedness_Tb2_00003_E1_boundedness_term} and
\eqref{eq_delta_tau_Theta_phi_Hneg_boundedness},
\begin{equation}
	\label{eq_delta_tau_e_phi_n1_negative_norm_boundedness}
	\begin{aligned}
		\|\delta_{\tau} e_{\phi}^{n+1}\|_{-1}\leq &~\|\nabla e_{\mu}^{n+1}\|+\|R_{\phi}^n\|_{-1}+C\tau\left(\|\phi^{n+1}\|_2\|\mathbf{u}_{t}\|_{L^{\infty}(0,T;L^2(\Omega))}+\|A\mathbf{u}^n\|\|\phi_t\|_{L^{\infty}(0,T;H^1(\Omega))}\right)\\
		&~+C\tau\left(\|\phi^n\|_2+\|\nabla\mathbf{u}_h^n\|+\|\nabla\mathbf{u}_h^n\|\|\phi_h^n\|_1\right)
		\left(\|\Theta_{\mathbf{u}}^n\|+\|\Phi_{\mathbf{u}}^n\|+\|\Phi_{\phi}^n\|_1+\|\Theta_{\phi}^n\|_1+|e_r^{n+1}|\right).
	\end{aligned}
\end{equation}
And then the term $T_{r1}$ is bounded by
\begin{equation}
	\label{eq_Tr1_boundedness_esitmation}
	\begin{aligned}
		\big|T_{r1}\big|=&~\bigg|\frac{\tau e_r^{n+1}}{\sqrt{E_1^{n+1}}}\left(F'(\phi^{n+1}),\delta_{\tau}\phi^{n+1}\right)-\frac{\tau e_r^{n+1}}{\sqrt{E_{1h}^{n}}}\left(F'(\phi_h^{n}),\delta_{\tau}\phi_h^{n+1}\right)\bigg|\\
		=&~\bigg|\frac{\tau e_r^{n+1}}{\sqrt{E_1^{n+1}}}\left(F'(\phi^{n+1}),\delta_{\tau}\phi^{n+1}\right)-\frac{\tau e_{r}^{n+1}}{\sqrt{E_{1h}^{n}}}\left(F'(\phi_h^n),\delta_{\tau}\phi^{n+1}\right)+\frac{\tau e_{r}^{n+1}}{\sqrt{E_{1h}^{n}}}\left(F'(\phi_h^n),\delta_{\tau} e_{\phi}^{n+1}\right)\bigg|\\
		\leq &~\tau|e_r^{n+1}|\bigg|\left(\frac{1}{\sqrt{E_1^{n+1}}}-\frac{1}{\sqrt{E_{1h}^{n}}}\right)\left(F'(\phi^{n+1}),\delta_{\tau}\phi^{n+1}\right)+\frac{1}{\sqrt{E_{1h}^{n}}}\left(F'(\phi^{n+1})-F'(\phi_h^n),\delta_{\tau}\phi^{n+1}\right)\\
		&~+\frac{1}{\sqrt{E_{1h}^{n}}}\left(F'(\phi_h^n),\delta_{\tau}e_{\phi}^{n+1}\right)\bigg|\\
		\leq &~C\tau|e_{r}^{n+1}|\left(\|e_{\phi}^{n}\|+\tau\|\phi_t\|_{L^{\infty}(0,T;H^1(\Omega))}+\|e_{\phi}^{n+1}\|\right)\|\phi_t\|_{L^{\infty}(0,T;H^1(\Omega))}\\
		&~+C\tau|e_r^{n+1}|\|\nabla F'(\phi_h^n)\|\|\delta_{\tau}e_{\phi}^{n+1}\|_{-1}\\
		\leq &~C\tau\left(|e_r^{n+1}|^2\|\phi_t\|_{L^{\infty}(0,T;H^1(\Omega))}^2+\|e_{\phi}^n\|^2+\|e_{\phi}^{n+1}\|^2+\tau^2\|\phi_t\|_{L^{\infty}(0,T;H^1(\Omega))}^2\right)+\frac{\tau}{3}\| e_{\mu}^{n+1}\|_1^2\\
		&~+C\tau\|R_{\phi}^n\|_{-1}^2+C\tau^3\left(\|\phi^{n+1}\|_2^2\|\mathbf{u}_{t}\|_{L^{\infty}(0,T;L^2(\Omega))}^2+\|A\mathbf{u}^n\|^2\|\phi_t\|_{L^{\infty}(0,T;H^1(\Omega))}^2\right)\\
		&~+C\tau\left(\|\phi^n\|_2^2+\|\nabla\mathbf{u}_h^n\|^2+\|\nabla\mathbf{u}_h^n\|^2\|\phi_h^n\|_1^2\right)
		\left(\|\Theta_{\mathbf{u}}^n\|^2+\|\Phi_{\mathbf{u}}^n\|^2+\|\Phi_{\phi}^n\|_1^2+\|\Theta_{\phi}^n\|_1^2+|e_r^{n+1}|^2\right).
	\end{aligned}
\end{equation}
Nextly, we can derive the boundedness of the term $T_{r2}$,
\begin{equation}
	\label{eq_Tr2_boundedness_esitmation}
	\begin{aligned}
		|T_{r2}|=&~\big|2\tau e_r^{n+1}\left(R_r^n-T_1^n\right)\big|\leq \tau|e_r^{n+1}|^2+\tau\left(\|R_r^n\|^2+\|T_1^n\|^2\right).
	\end{aligned}
\end{equation}
By using the $\left(\mathbf{u}\cdot\nabla\phi,\mu\right)-\left(\mu\nabla\phi,\mathbf{u}\right)=0$, we make the following transformation
for the term $T_{r3}$,
\begin{equation}
	\label{eq_Tr3_transformation_equations}
	\begin{aligned}
		\big|T_{r3}\big|=&~\bigg|-\frac{\tau e_r^{n+1}}{\sqrt{E_{1h}^n}}\left(\left(\mu_h^{n+1},\mathbf{u}_h^n\cdot\nabla\phi_h^n\right)-\left(\tilde{\mathbf{u}}_h^{n+1},\mu_h^n\cdot\nabla\phi_h^n\right)\right)\bigg|\\
		=&~\tau|e_r^{n+1}|\bigg|\frac{1}{\sqrt{E_1^n}}\left(\mathbf{u}^n\cdot\nabla\phi^n,\mu^n\right)-\frac{1}{\sqrt{E_{1h}^n}}\left(\mathbf{u}_h^n\cdot\nabla\phi_h^n,\mu_h^{n+1}\right)\\
		&~+\frac{1}{\sqrt{E_1^n}}\left(\mu^n\nabla\phi^n,\mathbf{u}^n\right)-\frac{1}{\sqrt{E_{1h}^n}}\left(\mu_h^n\nabla\phi_h^n,\tilde{\mathbf{u}}_h^{n+1}\right)\bigg|.
	\end{aligned}
\end{equation}
Thus, according to the \eqref{eq_boundedness_Tb2_00003_E1_boundedness_term}, we have
\begin{equation}
	\label{eq_boundedness_Tb2_00003_E1_boundedness_term_re_0001}
	\begin{aligned}
		\frac{1}{\sqrt{E_1^{n}}}-\frac{1}{\sqrt{E_{1h}^n}}
		=&~\frac{|E_{1h}^n-E_1^{n}|}{\sqrt{E_1^{n}E_{1h}^n}\left(\sqrt{E_1^{n}}+\sqrt{E_{1h}^n}\right)}\\
		=&~\frac{|E_{1h}^n-E_1^n|}{\sqrt{E_1^{n}E_{1h}^n}\left(\sqrt{E_1^{n}}+\sqrt{E_{1h}^n}\right)}
		\leq C\|e_{\phi}^n\|.
	\end{aligned}
\end{equation}
Using the \eqref{eq_boundedness_Tb2_00003_E1_boundedness_term_re_0001},
the first term of \eqref{eq_Tr3_transformation_equations} is bounded by
\begin{equation}
	\label{eq_Tr3_transformation_equations_sub001}
	\begin{aligned}
		&\tau|e_r^{n+1}|\bigg|\frac{1}{\sqrt{E_1^n}}\left(\mathbf{u}^n\cdot\nabla\phi^n,\mu^n\right)-\frac{1}{\sqrt{E_{1h}^n}}\left(\mathbf{u}_h^n\cdot\nabla\phi_h^n,\mu_h^{n+1}\right)\bigg|\\
		\leq&~\frac{\tau|e_r^{n+1}|}{\sqrt{E_1^n}}\bigg|\left(e_{\mathbf{u}}^n\cdot\nabla\phi^n+\mathbf{u}_h^n\cdot\nabla e_{\phi}^n,\mu^n\right)
		+\left(\mathbf{u}_h^n\cdot\nabla\phi_h^n,\mu^n-\mu^{n+1}+e_{\mu}^{n+1}\right)\bigg|\\
		&~+\bigg|\left(\frac{\tau|e_r^{n+1}|}{\sqrt{E_1^n}}-\frac{\tau|e_r^{n+1}|}{\sqrt{E_{1h}^n}}\right)\left(\mathbf{u}_h^n\cdot\nabla\phi_h^n,\mu_h^{n+1}\right)\bigg|\\
		\leq &~\frac{C\tau|e_r^{n+1}|}{\sqrt{C_0}}\left(\left(\|e_{\mathbf{u}}^n\|\|\phi^n\|_2+\|\nabla\mathbf{u}_h^n\|\|\nabla e_{\phi}^n\|\right)\|\mu^n\|_1+\left(\|e_{\mu}^{n+1}\|_1+\tau\|\mu_t\|_{L^{\infty}(0,T;H^1(\Omega))}\right)\|\nabla\phi_h^n\|\|\nabla\mathbf{u}_h^n\|\right)\\
		&~+C\tau|e_r^{n+1}|\|e_{\phi}^n\|\|\nabla\mathbf{u}_h^n\|\|\nabla\phi_h^n\|\|\mu_h^{n+1}\|_{1}\\
		\leq &~\frac{\tau}{3}\|e_{\mu}^{n+1}\|_1^2+C\tau^3\|\mu_t\|_{L^{\infty}(0,T;H^1(\Omega))}^2\|\nabla\phi_h^n\|^2\|\nabla\mathbf{u}_h^n\|^2+C\tau\left(\|\nabla\mathbf{u}_h^n\|^2\|\nabla\phi_h^n\|^2+\|\mu_h^{n+1}\|_1^2\right)|e_r^{n+1}|^2\\
		&~
		+C\tau\left(\|e_{\mathbf{u}}^n\|^2+\| e_{\phi}^n\|_1^2\right)\left(\|\phi^n\|_2^2\|\mu^n\|_1^2+\|\nabla\mathbf{u}_h^n\|^2\|\mu^n\|_1^2+\|\nabla\mathbf{u}_h^n\|^2\|\nabla\phi_h^n\|^2\right),
	\end{aligned}
\end{equation}
and similar to \eqref{eq_Tr3_transformation_equations_sub001}, the second term of  \eqref{eq_Tr3_transformation_equations} is estimated by
\begin{equation}
	\label{eq_Tr3_transformation_equations_sub002}
	\begin{aligned}
		&\tau|e_r^{n+1}|\bigg|\frac{1}{\sqrt{E_1^n}}\left(\mu^n\nabla\phi^n,\mathbf{u}^n\right)-\frac{1}{\sqrt{E_{1h}^n}}\left(\mu_h^n\nabla\phi_h^n,\tilde{\mathbf{u}}_h^{n+1}\right)\bigg|\\
		\leq&~\frac{\tau|e_r^{n+1}}{\sqrt{E_1^n}}|\bigg|\left(e_{\mu}^n\nabla\phi^n+\mu_h^n\nabla e_{\phi}^n,\mathbf{u}^n\right)+\left(\mu_h^n\nabla\phi_h^n,\mathbf{u}^n-\mathbf{u}^{n+1}+e_{\tilde{\mathbf{u}}}^{n+1}\right)\bigg|\\
		&~+\bigg|\left(\frac{\tau|e_r^{n+1}|}{\sqrt{E_1^n}}-\frac{\tau |e_r^{n+1}|}{\sqrt{E_{1h}^n}}\right)\left(\mu_h^n\nabla\phi_h^n,\tilde{\mathbf{u}}_h^{n+1}\right)\bigg|\\
		\leq &~\frac{C\tau|e_r^{n+1}|}{\sqrt{C_0}}\left(\left(\|e_{\mu}^n\|_1\|\nabla\phi^n\|+\|\mu_h^n\|_1\|\nabla e_{\phi}^n\|\right)\|\nabla\mathbf{u}^n\|\right.\\
		&~\left.
		+\left(\|\nabla e_{\tilde{\mathbf{u}}}^{n+1}\|+\tau\|\mathbf{u}_t\|_{L^{\infty}(0,T;L^2(\Omega))}\right)\|\nabla\phi_h^n\|\|\mu_h^n\|_1\right)+C\tau|e_r^{n+1}|\|e_{\phi}^n\|\|\nabla\tilde{\mathbf{u}}_h^n\|\|\nabla\phi_h^n\|\|\mu_h^n\|_1\\
		\leq &~\frac{\tau}{2}\left(\|\nabla\Phi_{\mathbf{u}}^{n+1}\|^2+\|\nabla\Theta_{\tilde{\mathbf{u}}}^{n+1}\|^2\right)+\frac{\tau}{3}\|e_{\mu}^n\|_1^2+C\tau^3\|\mathbf{u}_t\|_{L^{\infty}(0,T;L^2(\Omega))}^2\|\nabla\phi_h^n\|^2\|\mu_h^n\|_1^2\\
		&~+C\tau\left(|e_r^{n+1}|^2+\|e_{\phi}^n\|_1^2\right)\left(\|\nabla\phi^n\|^2\|\nabla\mathbf{u}^n\|^2+\|\mu_h^n\|_1^2\|\nabla\mathbf{u}^n\|^2+\|\nabla\phi_h^n\|^2\|\mu_h^n\|_1^2+\|\nabla\tilde{\mathbf{u}}_h^{n+1}\|^2+C\right).
	\end{aligned}
\end{equation}
Combining the above inequalities \eqref{eq_Tr1_boundedness_esitmation}-\eqref{eq_Tr3_transformation_equations_sub002} with \eqref{eq_multiply_2tau_ern1}, and summing up from $n=0$ to $m$, we have
\begin{flalign}
	|e_r^{m+1}|^2&~+\sum_{n=0}^{m}|e_r^{n+1}-e_r^n|^2\leq |e_r^0|^2+\tau\sum_{n=0}^{m}\|e_{\mu}^{n+1}\|_1^2+C\tau\sum_{n=0}^{m}\left(\|R_{\phi}^n\|_{-1}^2+\|R_r^n\|^2+\|T_1^n\|^2\right)\notag\\
	&~+C\tau\sum_{n=0}^{m}\left(\|\phi_t\|_{L^{\infty}(0,T;H^1(\Omega))}^2+\|e_{\phi}^n\|^2+\|e_{\phi}^{n+1}\|^2+\tau^2\|\phi_t\|_{L^{\infty}(0,T;H^1(\Omega))}^2\right)\notag\\
	&~
	+C\tau^3\sum_{n=0}^{m}\left(\|\phi^{n+1}\|_2^2\|\mathbf{u}_{t}\|_{L^{\infty}(0,T;L^2(\Omega))}^2+\|A\mathbf{u}^n\|^2\|\phi_t\|_{L^{\infty}(0,T;H^1(\Omega))}^2\right)\notag\\
	&~+C\tau\sum_{n=0}^{m}\left(\|\phi^n\|_2^2+\|\nabla\mathbf{u}_h^n\|^2+\|\nabla\mathbf{u}_h^n\|^2\|\phi_h^n\|_1^2\right)
	\left(\|\Theta_{\mathbf{u}}^n\|^2+\|\Phi_{\mathbf{u}}^n\|^2+\|\Phi_{\phi}^n\|_1^2\right.\notag\\
	&~\left.\qquad+\|\Theta_{\phi}^n\|_1^2+|e_r^{n+1}|^2\right)+\frac{\tau}{2}\sum_{n=0}^{m}\left(\|\nabla\Phi_{\mathbf{u}}^{n+1}\|^2+\|\nabla\Theta_{\tilde{\mathbf{u}}}^{n+1}\|^2\right)\\
	&~+C\tau^3\sum_{n=0}^{m}\|\mu_t\|_{L^{\infty}(0,T;H^1(\Omega))}^2\|\nabla\phi_h^n\|^2\|\nabla\mathbf{u}_h^n\|^2+C\tau\sum_{n=0}^{m}\left(\|\nabla\mathbf{u}_h^n\|^2\|\nabla\phi_h^n\|^2+\|\mu_h^{n+1}\|_1^2\right)\notag\\
	&~
	+C\tau\sum_{n=0}^{m}\left(\|e_{\mathbf{u}}^n\|^2+\| e_{\phi}^n\|_1^2\right)\left(\|\phi^n\|_2^2\|\mu^n\|_1^2+\|\nabla\mathbf{u}_h^n\|^2\|\mu^n\|_1^2+\|\nabla\mathbf{u}_h^n\|^2\|\nabla\phi_h^n\|^2\right)\notag\\
	&~+C\tau^3\sum_{n=0}^{m}\|\mathbf{u}_t\|_{L^{\infty}(0,T;L^2(\Omega))}^2\|\nabla\phi_h^n\|^2\|\mu_h^n\|_1^2+C\tau\sum_{n=0}^{m}\left(|e_r^{n+1}|^2+\|e_{\phi}^n\|_1^2\right)\left(\|\nabla\phi^n\|^2\|\nabla\mathbf{u}^n\|^2\right.\notag\\
	&~\left.\qquad+\|\mu_h^n\|_1^2\|\nabla\mathbf{u}^n\|^2+\|\nabla\phi_h^n\|^2\|\mu_h^n\|_1^2+\|\nabla\tilde{\mathbf{u}}_h^{n+1}\|^2+C\right).\notag
\end{flalign}
Using the above Lemmas, we complete the proof of \eqref{eq_r_boundedness_error_estimate}.
We next give the error estimate of the scalar auxiliary variable $\rho$.
Multiplying \eqref{eq_error_equations_rho} by $2\tau e_{\rho}^{n+1}$, we have
\begin{equation}
	\label{eq_multiplying_tau_e_rho_n1_results}
	\begin{aligned}
		|e_{\rho}^{n+1}|^2&~-|e_{\rho}^n|^2+|e_{\rho}^{n+1}-e_{\rho}^n|^2=\frac{\tau e_{\rho}^{n+1}}{\rho^{n+1}}\left(\delta_{\tau}\mathbf{u}^{n+1},\mathbf{u}^{n+1}\right)-\frac{\tau e_{\rho}^{n+1}}{\rho_h^{n+1}}\left(\frac{\tilde{\mathbf{u}}_h^{n+1}-\mathbf{u}_h^n}{\tau},\tilde{\mathbf{u}}_h^{n+1}\right)\\
		&~+\frac{2\tau e_{\rho}^{n+1}}{\sqrt{E_2^n}}\left(\mathbf{u}^n\cdot\nabla\mathbf{u}^n,\mathbf{u}^n\right)-\frac{2\tau e_{\rho}^{n+1}}{\sqrt{E_{2h}^n}}\left(\mathbf{u}_h^n\cdot\nabla\mathbf{u}_h^n,\tilde{\mathbf{u}}_h^{n+1}\right)+
		2\tau e_{\rho}^{n+1}\left(R_{\rho}^n-T_2^n\right)=\sum_{i=1}^{3}T_{\rho i}.
	\end{aligned}
\end{equation}
We can first derive the following result by  $\rho^{n+1}=\sqrt{E_2^{n+1}}$ and \eqref{eq_boundedness_Tb2_00003_E1_boundedness_term},
\begin{equation}
	\begin{aligned}
		\bigg|\frac{1}{\rho^{n+1}}-\frac{1}{\rho_h^{n+1}}\bigg|=&~\bigg|\frac{E_2^{n+1}-E_{2h}^{n+1}}{\rho^{n+1}\rho_h^{n+1}\left(\rho^{n+1}+\rho_h^{n+1}\right)}\bigg|
		\leq C\|e_{\mathbf{u}}^{n+1}\|.
	\end{aligned}
\end{equation}
The term $T_{\rho 1}$ is bounded by
\begin{equation}
	\label{eq_T_rho1_boundedness}
	\begin{aligned}
		\big|T_{\rho 1}\big|=&~\bigg|\frac{\tau e_{\rho}^{n+1}}{\rho^{n+1}}\left(\delta_{\tau}\mathbf{u}^{n+1},\mathbf{u}^{n+1}\right)-\frac{\tau e_{\rho}^{n+1}}{\rho_h^{n+1}}\left(\frac{\tilde{\mathbf{u}}_h^{n+1}-\mathbf{u}_h^n}{\tau},\tilde{\mathbf{u}}_h^{n+1}\right)\bigg|\\
		\leq&~C\tau|e_{\rho}^{n+1}|\|\mathbf{u}^{n+1}\|\|\mathbf{u}_t\|_{L^{\infty}(0,T;L^2(\Omega))}\\
		&~+C|e_{\rho}^{n+1}|\|\nabla\tilde{\mathbf{u}}_h^{n+1}\|\left(\|\nabla e_{\tilde{\mathbf{u}}}^{n+1}\|+\|\nabla e_{\mathbf{u}}^n\|+\tau\|\mathbf{u}_t\|_{L^{\infty}(0,T;\mathbf{H}^1(\Omega))}\right)\\
		\leq &~\frac{\tau}{4}|e_{\rho}^{n+1}|^2++C\tau\left(\|\mathbf{u}^{n+1}\|^2+\|\nabla\tilde{\mathbf{u}}_h^{n+1}\|^2+\|\nabla\Phi_{\mathbf{u}}^{n+1}\|^2+\|\nabla\Theta_{\tilde{\mathbf{u}}}^{n+1}\|^2+\|\nabla\Phi_{\mathbf{u}}^n\|^2+\|\nabla \Theta_{\mathbf{u}}^{n}\|^2\right)\\
		&~+C\tau^3\|\mathbf{u}_t\|_{L^{\infty}(0,T;\mathbf{H}^1(\Omega))}^2.
	\end{aligned}
\end{equation}
Similar to the \eqref{eq_boundedness_Tr7_00003_E2_boundedness_term}, we have
\begin{equation}
	\begin{aligned}
			\bigg|\frac{1}{\sqrt{E_2^n}}-\frac{1}{\sqrt{E_{2h}^n}}\bigg|=&~\frac{|E_{2h}^n-E_2^{n}|}{\sqrt{E_2^{n}}\sqrt{E_{2h}^{n}}\left(\sqrt{E_2^{n}}+\sqrt{E_{2h}^n}\right)}
		\leq C\|e_{\mathbf{u}}^n\|.
	\end{aligned}
\end{equation}
The term $T_{\rho2}$ is estimated by
\begin{equation}
	\label{eq_T_rho2_boundedness}
	\begin{aligned}
		\big|T_{\rho2}\big|=&~\bigg|\frac{2\tau e_{\rho}^{n+1}}{\sqrt{E_2^n}}\left(\mathbf{u}^n\cdot\nabla\mathbf{u}^n,\mathbf{u}^n\right)-\frac{2\tau e_{\rho}^{n+1}}{\sqrt{E_{2h}^n}}\left(\mathbf{u}_h^n\cdot\nabla\mathbf{u}_h^n,\tilde{\mathbf{u}}_h^{n+1}\right)\bigg|\\
		\leq&~\frac{C\tau|e_{\rho}^{n+1}|}{\sqrt{E_2^n}}\bigg|\left(e_{\mathbf{u}}^n\cdot\nabla\mathbf{u}^{n}+\mathbf{u}_h^n\cdot\nabla e_{\mathbf{u}}^n,\mathbf{u}^n\right)+\left(\mathbf{u}_h^n\cdot\nabla\mathbf{u}_h^n,\mathbf{u}^n-\mathbf{u}^{n+1}+e_{\tilde{\mathbf{u}}}^{n+1}\right)\bigg|\\
		&~+\left(\frac{\tau|e_{\rho}^{n+1}|}{\sqrt{E_2^n}}-\frac{\tau|e_{\rho}^{n+1}|}{\sqrt{E_{2h}^n}}\right)\bigg|\left(\mathbf{u}_h^n\cdot\nabla\mathbf{u}_h^n,\tilde{\mathbf{u}}_h^{n+1}\right)\bigg|\\
		\leq &~C\tau|e_{\rho}^{n+1}|\left(\|e_{\mathbf{u}}^n\|\|A\mathbf{u}^n\|\left(\|\nabla\mathbf{u}^n\|+\|\nabla\mathbf{u}_h^n\|\right)+\|\mathbf{u}_h^n\|\|\nabla\mathbf{u}_h^n\|\left(\tau\|\mathbf{u}_t\|_{L^{\infty}(0,T;\mathbf{H}^1(\Omega))}+\|\nabla e_{\tilde{\mathbf{u}}}^{n+1}\|\right)\right)\\
		&~+C\tau|e_{\rho}^{n+1}|\|e_{\mathbf{u}}^n\|\|\mathbf{u}_h^n\|\|\nabla\mathbf{u}_h^n\|\|\nabla\tilde{\mathbf{u}}_h^{n+1}\|\\
		\leq &~\frac{\tau}{2}\left(\|\nabla\Phi_{\mathbf{u}}^{n+1}\|^2+\|\nabla \Theta_{\tilde{\mathbf{u}}}^{n+1}\|^2\right)+C\tau^3\|\mathbf{u}_t\|_{L^{\infty}(0,T;\mathbf{H}^1(\Omega))}^2\|\mathbf{u}_h^n\|^2\|\nabla\mathbf{u}_h^n\|^2\\
		&~+C\tau\left(\|A\mathbf{u}^n\|^2+\|\mathbf{u}_h^n\|^2\|\nabla\mathbf{u}_h^n\|^2+\|\nabla\tilde{\mathbf{u}}_h^{n+1}\|^2+1\right)|e_{\rho}^{n+1}|^2\\
		&~+C\tau\left(\|\Phi_{\mathbf{u}}^n\|^2+\|\Theta_{{\mathbf{u}}}\|^2\right)\left(\|\nabla\mathbf{u}^n\|^2+\|\mathbf{u}_h^n\|^2\|\nabla\mathbf{u}_h^n\|^2\right),
	\end{aligned}
\end{equation}
and  the term $T_{\rho 3}$ is bounded by
\begin{equation}
	\label{eq_T_rho3_boundedness}
	\begin{aligned}
		\big|T_{\rho 3}\big|=&~\bigg|2\tau e_{\rho}^{n+1}\left(R_{\rho}^n-T_2^n\right)\bigg|\leq C\tau| e_{\rho}^{n+1}|\left(\|R_{\rho}^n\|+\|T_2^n\|\right)\\
		\leq &~\frac{\tau}{4}| e_{\rho}^{n+1}|^2+C\tau\left(\|R_{\rho}^n\|^2+\|T_2^n\|^2\right).
	\end{aligned}
\end{equation}
Combining the above inequalities \eqref{eq_T_rho1_boundedness}-\eqref{eq_T_rho3_boundedness} with
\eqref{eq_multiplying_tau_e_rho_n1_results}, and summing up $n$ from $0$ to $m$, we have
\begin{equation}
	\label{eq_combining_eq_T_rho1_eq_T_rho3_boundedness}
	\begin{aligned}
		|e_{\rho}^{m+1}|^2\leq&~ |e_{\rho}^{0}|^2+C\tau\sum_{n=0}^{n}\left(\|\mathbf{u}^{n+1}\|^2+\|\nabla\Phi_{\mathbf{u}}^{n+1}\|^2+\|\nabla\Theta_{\tilde{\mathbf{u}}}^{n+1}\|^2+\|\nabla\Phi_{\mathbf{u}}^n\|^2+\|\nabla \Theta_{\mathbf{u}}^{n}\|^2\right)\\
		&~+C\tau^3\sum_{n=0}^{n}\|\mathbf{u}_t\|_{L^{\infty}(0,T;\mathbf{H}^1(\Omega))}^2+\frac{\tau}{2}\sum_{n=0}^{n}\left(|e_{\rho}^{n+1}|^2+\|\nabla\Phi_{\mathbf{u}}^{n+1}\|^2+\|\nabla \Theta_{\tilde{\mathbf{u}}}^{n+1}\|^2\right)\\
		&~+C\tau^3\sum_{n=0}^{n}\|\mathbf{u}_t\|_{L^{\infty}(0,T;\mathbf{H}^1(\Omega))}^2\|\mathbf{u}_h^n\|^2\|\nabla\mathbf{u}_h^n\|^2+C\tau\sum_{n=0}^{n}\left(\|R_{\rho}^n\|^2+\|T_2^n\|^2\right)\\
		&~+C\tau\sum_{n=0}^{n}\left(\|A\mathbf{u}^n\|^2+\|\mathbf{u}_h^n\|^2\|\nabla\mathbf{u}_h^n\|^2+\|\nabla\tilde{\mathbf{u}}_h^{n+1}\|^2+1\right)|e_{\rho}^{n+1}|^2\\
		&~+C\tau\sum_{n=0}^{n}\left(\|\Phi_{\mathbf{u}}^n\|^2+\|\Theta_{{\mathbf{u}}}\|^2\right)\left(\|\nabla\mathbf{u}^n\|^2+\|\mathbf{u}_h^n\|^2\|\nabla\mathbf{u}_h^n\|^2\right).
	\end{aligned}
\end{equation}
Using the above Lemmas, we complete the proof of \eqref{eq_rho_boundedness_error_estimate}.
\end{proof}
Thus, using the Lemmas \ref{lemma_Theta_phi_H1}-\ref{lemma_boundedness_er_rho} and 
\ref{lemma_discrete_Gronwall_inequation}, we obtain the following theorem.
\begin{Theorem}\label{theorem_e_phi_H1_boundedness}
	Under the assumption of \eqref{eq_varibles_satisfied_regularities}, for all $m\geq 0$, we have
	$$
		\begin{aligned}
			\|e_{\phi}^{m+1}\|_1^2+\|e_{\mathbf{u}}^{m+1}\|^2+|e_r^{m+1}|^2+|e_{\rho}^{m+1}|^2+\tau^2\|\nabla e_p^{n+1}\|^2&\\
			+\tau\sum_{n=0}^{m}\left(\|e_{\mu}^{n+1}\|_1^2+\|\nabla e_{\tilde{\mathbf{u}}}^{n+1}\|^2\right)&\leq C\left(\tau^2+h^{2r}\right),
		\end{aligned}
	$$
	where $C$ is positive constant that does not depend on $\tau$ and $h$.
\end{Theorem}
To obtain the optimal $L^2$ error estimates of $\phi$, $\mu$, $\mathbf{u}$ and $p$, we also need the following lemmas.
\begin{Lemma}\label{lemma_delta_tilde_u_L2_boundedness}
	Under the assumption of regularity \eqref{eq_varibles_satisfied_regularities} and $\tau=O(h)$, for all $m\geq 0$, we have
	\begin{equation}
		\|\nabla\tilde{\mathbf{u}}_h^{m+1}\|^2+\tau\sum_{n=0}^{m}\|A_h\tilde{\mathbf{u}}_h^{n+1}\|^2\leq C\tau,
	\end{equation}
	where $C$ is positive constant that  does not depend on $\tau$ and $h$. 
\end{Lemma}
\begin{proof}
	Taking inner product of \eqref{eq_fully_discrete_scheme_u} with $2\tau A_h\tilde{\mathbf{u}}_h^{n+1}$,  we obtain
	\begin{equation}
		\label{eq_2tau_Ah_tilde_u_taking_product_fully_discrete_scheme_u}
		\begin{aligned}
			&\|\nabla\tilde{\mathbf{u}}_h^{n+1}\|^2-\|\nabla\mathbf{u}_h^n\|^2+\|\nabla\tilde{\mathbf{u}}_h^{n+1}-\nabla\mathbf{u}_h^n\|^2+2\tau\|A_h\tilde{\mathbf{u}}_h^{n+1}\|^2\\
			=&~-2\tau\left(\nabla p_h^n,A_h\tilde{\mathbf{u}}_h^{n+1}\right)+\frac{\tau r_h^{n+1}}{\sqrt{E_{1h}^n}}\left(\mu_h^n\nabla\phi_h^n,A_h\tilde{\mathbf{u}}_h^{n+1}\right)-\frac{2\tau\rho_h^{n+1}}{\sqrt{E_{2h}^n}}\left(\mathbf{u}_h^n\cdot\nabla\mathbf{u}_h^n,A_h\tilde{\mathbf{u}}_h^{n+1}\right).
		\end{aligned}
	\end{equation}
	Thus, we have the following boundedness of the right-hand side terms of \eqref{eq_2tau_Ah_tilde_u_taking_product_fully_discrete_scheme_u}.
	Using the Young inequality, Cauchy-Schwarz inequality and Theorem \ref{theorem_e_phi_H1_boundedness}, we obtain
	\begin{equation}
		\label{eq_boundedness_2tau_Ah_tilde_u_taking_product_fully_discrete_scheme_u001}
		\begin{aligned}
			\big|-2\tau\left(p_h^n,A_h\tilde{\mathbf{u}}_h^{n+1}\right)\big|
			\leq &~C\tau\big|\left(\nabla \left(e_p^n+p^n\right), A_h\tilde{\mathbf{u}}_h^{n+1}\right)\big|\\
			\leq&~C\tau\left(\|\nabla e_p^n\|^2+\|\nabla p^n\|^2\right)+\frac{\tau}{3}\|A_h\tilde{\mathbf{u}}_h^{n+1}\|^2.
		\end{aligned}
	\end{equation}
	And using the \eqref{eq_boundness_basic_inequalities_0001}-\eqref{eq_boundness_basic_inequalities_0001} and Lemma \ref{lemma_boundedness_E_phi_Deltaphi},   we obtain
	\begin{equation}
		\label{eq_boundedness_2tau_Ah_tilde_u_taking_product_fully_discrete_scheme_u002}
		\begin{aligned}
			\bigg|\frac{\tau r_h^{n+1}}{\sqrt{E_{1h}^n}}\left(\mu_h^n\nabla\phi_h^n,A_h\tilde{\mathbf{u}}_h^{n+1}\right)\bigg|\leq &~\frac{C\tau|r_h^{n+1}|}{\sqrt{E_{1h}^n}}\|\mu_h^n\|_{L^4}\|\nabla\phi_h^n\|_{L^4}\|A_h\tilde{\mathbf{u}}_h^{n+1}\|\\
			\leq &~\frac{\tau}{3}\|A_h\tilde{\mathbf{u}}_h^{n+1}\|^2+\frac{C\tau}{C_0}\|\mu_h^n\|_1^2,
		\end{aligned}
	\end{equation}
	and 
	\begin{equation}
		\label{eq_boundedness_2tau_Ah_tilde_u_taking_product_fully_discrete_scheme_u003}
		\begin{aligned}
			\bigg|-\frac{2\tau\rho_h^{n+1}}{\sqrt{E_{2h}^n}}\left(\mathbf{u}_h^n\cdot\nabla\mathbf{u}_h^n,A_h\tilde{\mathbf{u}}_h^{n+1}\right)\bigg|
			\leq &~\frac{C\tau|\rho_h^{n+1}|}{\sqrt{E_{2h}^n}}\|\mathbf{u}_h^n\|_{L^{\infty}}\|\nabla\mathbf{u}_h^n\|\|A_h\tilde{\mathbf{u}}_h^{n+1}\|\\
			\leq &~\frac{C\tau|\rho_h^{n+1}|}{\sqrt{E_{2h}^n}}\|A_h\mathbf{u}_h^{n}\|^{\frac{1}{2}}\|\mathbf{u}_h^n\|^{\frac{1}{2}}\|\mathbf{u}_h^n\|_1\|A_h\tilde{\mathbf{u}}_h^{n+1}\|\\
			\leq &~\frac{\tau}{3}\|A_h\tilde{\mathbf{u}}_h^{n+1}\|^2+\frac{C\tau}{C_0}\|\mathbf{u}_h^n\|\|\nabla\mathbf{u}_h^n\|^2\|A_h\mathbf{u}_h^n\|\\
			\leq &~\frac{\tau}{3}\|A_h\tilde{\mathbf{u}}_h^{n+1}\|^2+\tau\|A_h\mathbf{u}_h^n\|^2+\frac{C\tau}{C_0}\|\mathbf{u}_h^n\|^2\|\nabla\mathbf{u}_h^n\|^4.
		\end{aligned}
	\end{equation}
	Combining the above inequalities \eqref{eq_boundedness_2tau_Ah_tilde_u_taking_product_fully_discrete_scheme_u001}-\eqref{eq_boundedness_2tau_Ah_tilde_u_taking_product_fully_discrete_scheme_u003} with \eqref{eq_2tau_Ah_tilde_u_taking_product_fully_discrete_scheme_u}, using the $\tau=O(h)$ and Theorem \ref{theorem_e_phi_H1_boundedness}, then summing up $n$ from $0$ to $m$, we have
	\begin{equation}
		\begin{aligned}
			\|\nabla\tilde{\mathbf{u}}_h^{m+1}\|^2&~+\sum_{n=0}^{m}\|\nabla\tilde{\mathbf{u}}_h^{n+1}-\nabla\mathbf{u}_h^n\|^2+\tau\sum_{n=0}^{m}\|A_h\tilde{\mathbf{u}}_h^{n+1}\|^2\\
			\leq &~\|\nabla\tilde{\mathbf{u}}_h^0\|^2+\tau\sum_{n=0}^{m}\|A_h\mathbf{u}_h^n\|^2+\frac{C\tau}{C_0}\left(\|\mu_h^n\|^2_1+\|\mathbf{u}_h^n\|^2\|\nabla\mathbf{u}_h^n\|^4\right)+C.
		\end{aligned}
	\end{equation}
	Using the discrete Gr\"{o}nwall's Lemma \ref{lemma_discrete_Gronwall_inequation} and  Lemma \ref{lemma_boundedness_E_phi_Deltaphi}, we obtain the desired result.
\end{proof}
\begin{Theorem}\label{theorem_e_phi_u_L2_error_estimates}
	Suppose that the system \eqref{eqCHNS01} has a unique solution $\left(\phi,\mu,\mathbf{u},p\right)$ satisfying \eqref{eq_varibles_satisfied_regularities}. 
	Then the fully discrete scheme \eqref{eq_fully_discrete_scheme_phi}-\eqref{eq_fully_discrete_scheme_q} have a unique solution $\left(\phi_h^{n+1},\mu_h^{n+1},\mathbf{u}_h^{n+1},p_h^{n+1}\right) \in \mathcal{X}_h^r$ and the initial error is $0$ such that
	\begin{flalign}
		\label{eq_e_phi_mu_L2_error_estimate}
		\|e_\phi^{m+1}\|^2+\tau\sum_{n=0}^{m}\|e_{\mu}^{n+1}\|^2\leq C\left(\tau^2+h^{2(r+1)}\right),\\
		\label{eq_e_u_L2_error_estimate}
		\|e_{\mathbf{u}}^{n+1}\| \leq C\left(\tau+\mathcal{E}_h\right),
	\end{flalign}
	for some $m\geq 0$  and $C$ is a positive constant independent of $\tau$ and $h$.
\end{Theorem}
\begin{proof}
	Testing $w_h = 2\tau\Delta_h^{-1}\Theta_{\mu}^{n+1}$ in \eqref{eq_error_equations_phi}, $\varphi_h=2\tau\delta_{\tau}\Delta_h^{-1}\Theta_{\phi}^{n+1}$ in \eqref{eq_error_equations_mu}, respectively, using the \eqref{eq_discrete_Laplace_operator_0001} and
	\eqref{eq_discrete_Laplace_operator_0002},  we have
	\begin{equation}
		\label{eq_error_estimation_L2_phi}
		\begin{aligned}
			&\|\Theta_{\phi}^{n+1}\|^2-\|\Theta_{\phi}^n\|^2+\|\Theta_{\phi}^{n+1}-\Theta_{\phi}^n\|^2+\tau\left(\|e_{\mu}^{n+1}\|^2+\|\Theta_{\mu}^{n+1}\|^2\right)\\
			=&~\tau\|\Phi_{\mu}^{n+1}\|^2-2\tau\left(\left(R_{\phi}^n,\Delta_h^{-1}\Theta_{\mu}^{n+1}\right)+\left(\Phi_{\phi}^{n+1},\delta_{\tau}\Theta_{\phi}^{n+1}\right)\right)+2\tau\left(e_{\phi}^{n+1},\delta_{\tau}\Delta_h^{-1}\Theta_{\phi}^{n+1}\right)\\
			&~
			+\left(\frac{2\tau r^{n+1}}{\sqrt{E_1^{n+1}}}\mathbf{u}^{n+1}\cdot\nabla\phi^{n+1}-\frac{2\tau r_h^{n+1}}{\sqrt{E_{1h}^n}}\mathbf{u}_h^n\cdot\nabla\phi_h^n,\Delta_h^{-1}\Theta_{\mu}^{n+1}\right)\\
			&~+\left(\frac{2\tau r^{n+1}}{\sqrt{E_1^{n+1}}}F'(\phi^{n+1})-\frac{2\tau r_h^{n+1}}{\sqrt{E_{1h}^n}}F'(\phi_h^n),\delta_{\tau}\Delta_h^{-1}\Theta_{\phi}^{n+1}\right)\\
			=&~\tau\|\Phi_{\mu}^{n+1}\|^2+\sum_{i=1}^{4}T_{\phi i}.
		\end{aligned}
	\end{equation}
	We next estimate the terms $T_{\phi 1}\sim T_{\phi 4}$ on the right-hand side of the equation \eqref{eq_error_estimation_L2_phi} separately. Using \eqref{eq_boundness_basic_inequalities_0001}-\eqref{eq_boundness_basic_inequalities_0004}, we obtain
	\begin{equation}
		\begin{aligned}
			\big|T_{\phi 1}\big|=&~\bigg|-2\tau\left(\left(R_{\phi}^n,\Delta_h^{-1}\Theta_{\mu}^{n+1}\right)+\left(\Phi_{\phi}^{n+1},\delta_{\tau}\Theta_{\phi}^{n+1}\right)\right)\bigg|\\
			\leq &~\frac{\tau}{8}\|\Theta_{\mu}^{n+1}\|^2+\tau^2\|\delta_{\tau}\Theta_{\phi}^{n+1}\|^2+\|\Phi_{\phi}^{n+1}\|^2+C\tau\|R_{\phi}^n\|_{-1}^2.
		\end{aligned}
	\end{equation}
	By using the $r^{n+1}=\sqrt{E_1^{n+1}}$, the term $T_{\phi 2}$ is estimated as follows
	\begin{equation}
		\begin{aligned}
			\big|T_{\phi 2}\big|=&~\bigg|\left(\frac{2\tau r^{n+1}}{\sqrt{E_1^{n+1}}}\mathbf{u}^{n+1}\cdot\nabla\phi^{n+1}-\frac{2\tau r_h^{n+1}}{\sqrt{E_{1h}^n}}\mathbf{u}_h^n\cdot\nabla\phi_h^n,\Delta_h^{-1}\Theta_{\mu}^{n+1}\right)\bigg|\\
			= &~2\tau\bigg|\left(\mathbf{u}^{n+1}\cdot\nabla\phi^{n+1}-\mathbf{u}^n\cdot\nabla\phi^n+\mathbf{u}^n\cdot\nabla\phi^n-\mathbf{u}_h^n\cdot\nabla\phi_h^n,\Delta_h^{-1}\Theta_{\mu}^{n+1}\right)\\
			&~+\left(\frac{e_r^{n+1}}{\sqrt{E_1^{n+1}}}+\frac{r_h^{n+1}}{\sqrt{E_1^{n+1}}}-\frac{r_h^{n+1}}{\sqrt{E_{1h}^{n}}}\right)\left(\mathbf{u}_h^n\cdot\nabla\phi_h^n,\Delta_h^{-1}\Theta_{\mu}^{n+1}\right)\bigg|.
		\end{aligned}
	\end{equation}
	Using the above inequality \eqref{eq_boundedness_Tb2_00003_E1_boundedness_term},
	we have
	\begin{equation}
		\begin{aligned}
			&2\tau\big|\left(\mathbf{u}^{n+1}\cdot\nabla\phi^{n+1}-\mathbf{u}^n\cdot\nabla\phi^n,\Delta_h^{-1}\Theta_{\mu}^{n+1}\right)\big|\\
			=&~ 2\tau\big|\left(\left(\mathbf{u}^{n+1}-\mathbf{u}^n\right)\cdot\nabla\phi^{n+1}+\mathbf{u}^n\cdot\nabla\left(\phi^{n+1}-\phi^n\right),\Delta_h^{-1}\Theta_{\mu}^{n+1}\right)\big|\\
			\leq &~\frac{\tau}{8}\|\Theta_{\mu}^{n+1}\|^2+C\tau^3\left(\|\mathbf{u}_t\|_{L^{\infty}(0,T;H^1(\Omega))}^2\|\phi^{n+1}\|_1^2+\|\nabla\mathbf{u}^n\|^2\|\phi_{tt}\|_{L^{\infty}(0,T;L^2(\Omega))}^2\right),
		\end{aligned}
	\end{equation}
	and
	\begin{equation}
		\begin{aligned}
			&2\tau\big|\left(\mathbf{u}^n\cdot\nabla\phi^n-\mathbf{u}_h^n\cdot\nabla\phi_h^n,\Delta_h^{-1}\Theta_{\mu}^{n+1}\right)\big|\\
			=&~2\tau\big|\left(e_{\mathbf{u}}^n\cdot\nabla\phi^n+\mathbf{u}_h^n\cdot\nabla e_{\phi}^n,\Delta_h^{-1}\Theta_{\mu}^{n+1}\right)\big|\\
			\leq &~\frac{\tau}{8}\|\Theta_{\mu}^{n+1}\|^2+C\tau\left(\|\phi^n\|_1^2+\|\nabla\mathbf{u}_h^n\|^2\right)\left(\|\Phi_{\phi}^n\|^2+\|\Theta_{\phi}^n\|^2+\|\Phi_{\mathbf{u}}^n\|^2+\|\Theta_{\mathbf{u}}^n\|^2\right),
		\end{aligned}
	\end{equation}
	and
	\begin{equation}
		\begin{aligned}
			&2\tau\bigg|\left(\frac{e_r^{n+1}}{\sqrt{E_1^{n+1}}}+\frac{r_h^{n+1}}{\sqrt{E_1^{n+1}}}-\frac{r_h^{n+1}}{\sqrt{E_{1h}^{n}}}\right)\left(\mathbf{u}_h^n\cdot\nabla\phi_h^n,\Delta_h^{-1}\Theta_{\mu}^{n+1}\right)\bigg|\\
			\leq &~C\tau\left(\frac{|e_r^{n+1}|}{\sqrt{C_0}}+\|e_{\phi}^n\|+\tau\|\phi_t\|_{L^{\infty}(0,T;H^1(\Omega))}\right)\|\mathbf{u}_h^n\|\|\nabla\phi_h^n\|\|\Theta_{\mu}^{n+1}\|\\
			\leq &~\frac{\tau}{8}\|\Theta_{\mu}^{n+1}\|^2+C\tau\left(|e_r^{n+1}|^2+\|\Phi_{\phi}^n\|^2+\|\Theta_{\phi}^n\|^2+\tau^2\|\phi_t\|_{L^{\infty}(0,T;H^1(\Omega))}^2\right)\|\mathbf{u}_h^n\|^2\|\nabla\phi_h^n\|^2.
		\end{aligned}
	\end{equation}
	Similar to the estimate of \eqref{eq_boundedness_Tb3_terms}, we have 
	\begin{equation}
		\label{eq_T_phi4_boundedness_estimate}
		\begin{aligned}
			\big|T_{\phi 4}\big|=&~\bigg|\left(\frac{2\tau r^{n+1}}{\sqrt{E_1^{n+1}}}F'(\phi^{n+1})-\frac{2\tau r_h^{n+1}}{\sqrt{E_{1h}^n}}F'(\phi_h^n),\delta_{\tau}\Delta_h^{-1}\Theta_{\phi}^{n+1}\right)\bigg|\\
			=&~2\tau\bigg|\left(F'(\phi^{n+1})-F'(\phi^n)+F'(\phi^n)-F'(\phi_h^n),\delta_{\tau}\Delta_h^{-1}\Theta_{\phi}^{n+1}\right)\\
			&~+\left(\frac{e_r^{n+1}}{\sqrt{E_1^{n+1}}}+\frac{r_h^{n+1}}{\sqrt{E_1^{n+1}}}-\frac{r_h^{n+1}}{\sqrt{E_{1h}^{n}}}\right)\left(F'(\phi_h^n),\delta_{\tau}\Delta_h^{-1}\Theta_{\phi}^{n+1}\right)\bigg|.
		\end{aligned}
	\end{equation}
	From \eqref{eq_boundedness_Tb2_00003_E1_boundedness_term} and \eqref{eq_nonlinear_term_f_boundedness0001}, we have
	\begin{equation}
		\begin{aligned}
			&2\tau\big|\left(F'(\phi^{n+1})-F'(\phi^n),\delta_{\tau}\Delta_h^{-1}\Theta_{\phi}^{n+1}\right)\big|
			\leq\|F'(\phi^{n+1})-F'(\phi^n)\|\|\delta_{\tau}\Theta_{\phi}^{n+1}\|_{-2}.
		\end{aligned}
	\end{equation}
	According to \eqref{eq_delta_tau_Theta_phi_negative_norm_boundedness},
	\eqref{eq_delta_tau_Theta_phi_Hneg_boundedness} and the fact that $\|\cdot\|_{-2}\leq \|\cdot\|_{-1}$, we have
	\begin{flalign}
		\|\delta_{\tau}\Theta_{\phi}^{n+1}\|_{-2}\leq &~
		\|\Delta_h e_{\mu}^{n+1}\|_{-2}+\|\delta_{\tau}\Phi_{\phi}^{n+1}\|_{-1}+\|R_{\phi}^n\|_{-1}\notag\\
		&~+\bigg\|\frac{r^{n+1}}{\sqrt{E_1^{n+1}}}\mathbf{u}^{n+1}\cdot\nabla\phi^{n+1}-\frac{r_h^{n+1}}{\sqrt{E_{1h}^n}}\mathbf{u}_h^n\cdot\nabla\phi_h^n\bigg\|_{-1}\notag\\
		\leq &~\|e_{\mu}^{n+1}\|+C\mathcal{E}_h\|\phi_t\|_{r}+\|R_{\phi}^n\|_{-1}
		\\
		&~+C\tau\left(\|\phi^{n+1}\|_{2}\|\mathbf{u}_t\|_{L^{\infty}(0,T;H^1(\Omega))}+\|A\mathbf{u}^n\|\|\phi_{tt}\|_{L^{\infty}(0,T;L^2(\Omega))}\right)\notag\\
		&~+C\left(\|\phi^n\|_2+\|A_h\mathbf{u}_h^n\|+\|\nabla\mathbf{u}_h^n\|\|\phi_h^n\|_1+\tau\|\phi_t\|_{L^{\infty}(0,T;H^1(\Omega))}\right)\left(\|\Theta_{\mathbf{u}}^n\|\right.\notag\\
		&~\left.+\|\Phi_{\mathbf{u}}\|+\|\Theta_{\phi}^n\|+\|\Phi_{\phi}^n\|+|e_r^{n+1}|\right).\notag
	\end{flalign}
	Thus, we have the error estimate of the term \eqref{eq_T_phi4_boundedness_estimate},
	\begin{equation}
		\label{eq_T_phi4_boundedness_estimate_sub001}
		\begin{aligned}
			\big|T_{\phi 4}\big|\leq &~\tau\|e_{\mu}^{n+1}\|^2+C\tau\|R_{\phi}^n\|_{-1}^2+C\mathcal{E}_h^2\|\phi_t\|_r^2\\
			&~+C\tau^3\left(\|\phi^{n+1}\|_2^2\|\mathbf{u}_t\|_{L^{\infty}(0,T;H^1(\Omega))}^2+\|A\mathbf{u}^n\|^2\|\phi_{tt}\|_{L^{\infty}(0,T;L^2(\Omega))}^2+\|\phi_t\|_{L^{\infty}(0,T;H^1(\Omega))}^2\right)\\
			&~+C\tau\left(\|\phi^n\|_2^2+\|A_h\mathbf{u}_h^n\|^2+\|\nabla\mathbf{u}_h^n\|^2\|\phi_h^n\|_1^2\right)\left(\|\Theta_{\mathbf{u}}^n\|^2+\|\Phi_{\mathbf{u}}\|^2+\|\Theta_{\phi}^n\|^2+\|\Phi_{\phi}^n\|^2+|e_r^{n+1}|^2\right).
		\end{aligned}
	\end{equation}
	Combining the estimates of $T_{\phi 1}\sim T_{\phi 4}$, then summing up $n$ from $0$ to $m$, we obtain
	$$
		\begin{aligned}
			&\|\Theta_{\phi}^{m+1}\|^2+\frac{\tau}{2}\sum_{n=0}^{m}\|\Theta_{\mu}^{n+1}\|^2\leq \|\Theta_{\phi}^0\|^2+C\tau\sum_{n=0}^{m}\left(\|\Phi_{\mu}^{n+1}\|^2+\|R_{\phi}^n\|_{-1}^2\right)+C\mathcal{E}_h^2\|\phi_t\|_r^2\\
			&+C\tau^3\sum_{n=0}^{m}\left(\|\mathbf{u}_t\|_{L^{\infty}(0,T;H^1(\Omega))}^2\|\phi^{n+1}\|_1^2+\|\nabla\mathbf{u}^n\|^2\|\phi_{tt}\|_{L^{\infty}(0,T;L^2(\Omega))}^2\right)\\
			&+C\tau\sum_{n=0}^{m}\left(\|\phi^n\|_1^2+\|\nabla\mathbf{u}_h^n\|^2\right)\left(\|\Phi_{\phi}^n\|^2+\|\Theta_{\phi}^n\|^2+\|\Phi_{\mathbf{u}}^n\|^2+\|\Theta_{\mathbf{u}}^n\|^2\right)\\
			&+C\tau\sum_{n=0}^{m}\left(|e_r^{n+1}|^2+\|\Phi_{\phi}^n\|^2+\|\Theta_{\phi}^n\|^2+\tau^2\|\phi_t\|_{L^{\infty}(0,T;H^1(\Omega))}^2\right)\|\mathbf{u}_h^n\|^2\|\nabla\phi_h^n\|^2\\
			&+C\tau^3\sum_{n=0}^{m}\left(\|\phi^{n+1}\|_2^2\|\mathbf{u}_t\|_{L^{\infty}(0,T;H^1(\Omega))}^2+\|A\mathbf{u}^n\|^2\|\phi_{tt}\|_{L^{\infty}(0,T;L^2(\Omega))}^2+\|\phi_t\|_{L^{\infty}(0,T;H^1(\Omega))}^2\right)\\
			&+C\tau\sum_{n=0}^{m}\left(\|\phi^n\|_2^2+\|A_h\mathbf{u}_h^n\|^2+\|\nabla\mathbf{u}_h^n\|^2\|\phi_h^n\|_1^2\right)\left(\|\Theta_{\mathbf{u}}^n\|^2+\|\Phi_{\mathbf{u}}\|^2+\|\Theta_{\phi}^n\|^2+\|\Phi_{\phi}^n\|^2+|e_r^{n+1}|^2\right).
		\end{aligned}
	$$
	Using the discrete Gr\"{o}nwall's Lemma \ref{lemma_discrete_Gronwall_inequation}, the inequalities \eqref{eq_boundness_Ritz_0001}-\eqref{eq_boundness_Ritz_0004}, the Lemmas \ref{lemma_Stokes_projection_bound}, \ref{lemma_boundedness_E_phi_Deltaphi}, \ref{lemma_truncation_boundedness}, \ref{lemma_Theta_phi_H1}, \ref{lemma_boundedness_er_rho}, \ref{lemma_delta_tilde_u_L2_boundedness}, the Theorem \ref{theorem_e_phi_H1_boundedness} and the triangle inequality,
	we obtain the result of  \eqref{eq_e_phi_mu_L2_error_estimate}.\\	
	Then we prove the optimal $L^2$ error estimate for velocity $\mathbf{u}$. Adding 
	\eqref{eq_error_equations_u_p_splitting} to \eqref{eq_error_equations_u}, we obtain
	\begin{equation}
		\label{eq_add_two_error_equations_u_splitting}
		\begin{aligned}
			\left(\delta_{\tau} \left(\Phi_{\mathbf{u}}^{n+1}+\Theta_{\mathbf{u}}^{n+1}\right),\mathbf{v}_h\right)+\left(\nabla e_{\tilde{\mathbf{u}}}^{n+1},\nabla\mathbf{v}_h\right)-\left(e_p^{n+1},\nabla\cdot\mathbf{v}_h\right)-\left(R_{\mathbf{u}}^{n}+R_p^n,\mathbf{v}_h\right)&\\
			+\left(\frac{\rho^{n+1}}{\sqrt{E_2^{n+1}}}\mathbf{u}^{n+1}\cdot\nabla\mathbf{u}^{n+1}-\frac{\rho_h^{n+1}}{\sqrt{E_{2h}^n}}\mathbf{u}_h^n\cdot\nabla\mathbf{u}_h^n,\mathbf{v}_h\right)&\\
			-\left(\frac{r^{n+1}}{\sqrt{E_1^{n+1}}}\mu^{n+1}\nabla\phi^{n+1}-\frac{r_h^{n+1}}{\sqrt{E_{1h}^n}}\mu_h^n\nabla\phi_h^n,\mathbf{v}_h\right)&~=0.
		\end{aligned}
	\end{equation}
	Taking $\mathbf{v}_h=2\tau\Theta_{\mathbf{u}}^{n+1}$ in \eqref{eq_add_two_error_equations_u_splitting}, and using the \eqref{eq_Stokes_projection_0001}-\eqref{eq_Stokes_projection_0002}, \eqref{eq_error_equations_incompressible_condition}, we obtain
	\begin{equation}
		\label{eq_taking_2tau_theta_u_inner_product_0001}
		\begin{aligned}
			\|\Theta_{\mathbf{u}}^{n+1}\|^2-\|\Theta_{\mathbf{u}}^n\|^2+\|\Theta_{\mathbf{u}}^{n+1}-\Theta_{\mathbf{u}}^n\|^2&+2\tau\|\nabla\Theta_{\mathbf{u}}^{n+1}\|^2=2\tau\left(R_{\mathbf{u}}^n+R_p^n,\Theta_{\mathbf{u}}^{n+1}\right)-2\tau\left(\delta_{\tau}\Phi_{\mathbf{u}}^{n+1},\Theta_{\mathbf{u}}^{n+1}\right)\\
			&
			-\left(\frac{\rho^{n+1}}{\sqrt{E_2^{n+1}}}\mathbf{u}^{n+1}\cdot\nabla\mathbf{u}^{n+1}-\frac{\rho_h^{n+1}}{\sqrt{E_{2h}^n}}\mathbf{u}_h^n\cdot\nabla\mathbf{u}_h^n,\Theta_{\mathbf{u}}^{n+1}\right)\\
			&+\left(\frac{2\tau r^{n+1}}{\sqrt{E_1^{n+1}}}\mu^{n+1}\nabla\phi^{n+1}-\frac{2\tau r_h^{n+1}}{\sqrt{E_{1h}^n}}\mu_h^n\nabla\phi_h^n,\Theta_{\mathbf{u}}^{n+1}\right).
		\end{aligned}
	\end{equation}
	Similar to the estimates of $T_{u2}$, $T_{u3}$ and $T_{u7}$, we obtain
	\begin{equation}
		\label{eq_taking_2tau_theta_u_inner_product_0001_sub001}
		\begin{aligned}
			\big|2\tau\left(R_{\mathbf{u}}^n+R_p^n,\Theta_{\mathbf{u}}^{n+1}\right)\big|\leq \frac{\tau}{3}\|\nabla\Theta_{\mathbf{u}}^{n+1}\|^2+C\tau\left(\|R_{\mathbf{u}}^n\|_{-1}^2+\|R_p^n\|_{-1}^2\right),
		\end{aligned}
	\end{equation}
	\begin{equation}
		\label{eq_taking_2tau_theta_u_inner_product_0001_sub002}
		\begin{aligned}
			\big|-2\tau\left(\delta_{\tau}\Phi_{\mathbf{u}}^{n+1},\Theta_{\mathbf{u}}^{n+1}\right)\big|\leq \frac{\tau}{2}\|\Theta_{\mathbf{u}}^{n+1}\|^2+C\tau\|\delta_{\tau}\Phi_{\mathbf{u}}^{n+1}\|^2,
		\end{aligned}
	\end{equation}
	\begin{equation}
		\label{eq_taking_2tau_theta_u_inner_product_0001_sub003}
		\begin{aligned}
			&\bigg|-\left(\frac{\rho^{n+1}}{\sqrt{E_2^{n+1}}}\mathbf{u}^{n+1}\cdot\nabla\mathbf{u}^{n+1}-\frac{\rho_h^{n+1}}{\sqrt{E_{2h}^n}}\mathbf{u}_h^n\cdot\nabla\mathbf{u}_h^n,\Theta_{\mathbf{u}}^{n+1}\right)\bigg|\\
			\leq&~\frac{\tau}{3}\|\nabla\Theta_{\mathbf{u}}^{n+1}\|^2+C\tau^3\|\mathbf{u}_t\|_{L^{\infty}(0,T;L^2(\Omega))}^2\left(\|A\mathbf{u}^{n+1}\|^2+\|A\mathbf{u}^n\|^2\right)\\
			&+C\tau\left(\|A\mathbf{u}^n\|^2+\|\nabla\mathbf{u}^n\|^2+\|\nabla\mathbf{u}_h^n\|^2\right)\left(\|\Phi_{\mathbf{u}}^{n}\|^2+\|\Theta_{\mathbf{u}}^{n}\|^2\right)\\
			&+C\tau\|\mathbf{u}_h^n\|^2\|\nabla\mathbf{u}_h^n\|^2\left(|e_{\rho}^{n+1}|^2+\|\Phi_{\mathbf{u}}^n\|^2+\|\Theta_{\mathbf{u}}^n\|^2+\tau^2\|\mathbf{u}_t\|_{L^{\infty}(0,T;L^2(\Omega))}^2\right),
		\end{aligned}
	\end{equation}
	and
	\begin{flalign}
		&\bigg|\left(\frac{2\tau r^{n+1}}{\sqrt{E_1^{n+1}}}\mu^{n+1}\nabla\phi^{n+1}-\frac{2\tau r_h^{n+1}}{\sqrt{E_{1h}^n}}\mu_h^n\nabla\phi_h^n,\Theta_{\mathbf{u}}^{n+1}\right)\bigg|\notag\\= &~2\tau\bigg|\left(\mu^{n+1}\nabla\phi^{n+1}-\mu^n\nabla\phi^n+\mu^n\nabla\phi^n-\mu_h^n\nabla\phi_h^n,\Theta_{\mathbf{u}}^{n+1}\right)\notag\\
		&~+\left(\frac{e_r^{n+1}}{\sqrt{E_1^{n+1}}}+\frac{r_h^{n+1}}{\sqrt{E_1^{n+1}}}-\frac{r_h^{n+1}}{\sqrt{E_{1h}^n}}\right)\left(\mu_h^n\nabla\phi_h^n,\Theta_{\mathbf{u}}^{n+1}\right)\bigg|\notag\\
		= &~2\tau\big|\left(\left(\mu^{n+1}-\mu^n\right)\nabla\phi^{n+1}+\mu^n\nabla\left(\phi^{n+1}-\phi^n\right)+e_{\mu}^n\nabla\phi^n+\mu_h^n\nabla e_{\phi}^n,\Theta_{\mathbf{u}}^{n+1}\right)\notag\\
			\label{eq_taking_2tau_theta_u_inner_product_0001_sub004}
		&~+\left(\frac{e_r^{n+1}}{\sqrt{E_1^{n+1}}}+\frac{r_h^{n+1}}{\sqrt{E_1^{n+1}}}-\frac{r_h^{n+1}}{\sqrt{E_{1h}^n}}\right)\left(\mu_h^n\nabla\phi_h^n,\Theta_{\mathbf{u}}^{n+1}\right)\bigg|\\
		\leq &~C\tau\left(\tau^2\|\mu_t\|_{L^{\infty}(0,T;H^1(\Omega))}\|\nabla\phi^{n+1}\|+\tau^2\|\mu^{n}\|_1\|\phi_t\|_{L^{\infty}(0,T;H^1(\Omega))}\right)\|\nabla\Theta_{\mathbf{u}}^{n+1}\|\notag\\
		&~+C\tau\left(\|e_{\mu}^n\|_{-1}\|\nabla\phi^n\|_{1}\|\nabla\Theta_{\mathbf{u}}^{n+1}\|_1+\|\mu_h^n\|_{L^4}\|\nabla e_{\phi}^n\|\|\Theta_{\mathbf{u}}^{n+1}\|_{L^\infty}\right)\notag\\
		&~+\left(\frac{C\tau|e_r^{n+1}|}{\sqrt{C_0}}+\|e_{\phi}^n\|+\tau\|\phi_t\|_{L^{\infty}(0,T;H^1(\Omega))}\right)\|\mu_h^n\|_1\|\nabla\phi_h^n\|\|\nabla\Theta_{\mathbf{u}}^{n+1}\|\notag\\
		\leq &~\frac{\tau}{3}\|\nabla\Theta_{\mathbf{u}}^{n+1}\|^2+C\tau^3\left(\|\mu_t\|_{L^{\infty}(0,T;H^1(\Omega))}^2\|\nabla\phi^{n+1}\|^2+\|\mu^{n}\|_1^2\|\phi_t\|_{L^{\infty}(0,T;H^1(\Omega))}^2\right)\notag\\
		&~+C\tau\left(\|\Phi_{\mu}^n\|_{-1}^2+\|\Theta_{\mu}^n\|_{-1}^2\right)\|\phi^n\|_2^2+C\tau\|\mu_h^n\|_1^2\left(\|\nabla \Phi_{\phi}^n\|^2+\|\nabla\Theta_{\phi}^n\|^2\right)\notag\\
		&~+C\tau\|\mu_h^n\|_1^2\|\nabla\phi_h^n\|^2\left(|e_{r}^{n+1}|^2+\|\Phi_{\phi}^n\|^2+\|\Theta_{\phi}^n\|^2+\tau^2\|\phi_t\|_{L^{\infty}(0,T;H^1(\Omega))}^2\right).\notag
	\end{flalign}
	Combining the above inequalities \eqref{eq_taking_2tau_theta_u_inner_product_0001_sub001}-\eqref{eq_taking_2tau_theta_u_inner_product_0001_sub004} with \eqref{eq_taking_2tau_theta_u_inner_product_0001}, then summing up $n$ from $0$ to $m$, we have
	\begin{flalign}
			\|\Theta_{\mathbf{u}}^{m+1}\|^2&+\tau\sum_{n=0}^{m}\|\nabla\Theta_{\mathbf{u}}^{n+1}\|^2\leq \frac{\tau}{2}\sum_{n=0}^{m}\|\Theta_{\mathbf{u}}^{n+1}\|^2+C\tau\sum_{n=0}^{m}\left(\|R_{\mathbf{u}}^n\|_{-1}^2+\|R_p^n\|_{-1}^2+\|\delta_{\tau}\Phi_{\mathbf{u}}^{n+1}\|^2\right)\notag\\
		&+C\tau^3\sum_{n=0}^{m}\|\mathbf{u}_t\|_{L^{\infty}(0,T;L^2(\Omega))}^2\left(\|A\mathbf{u}^{n+1}\|^2+\|A\mathbf{u}^n\|^2\right)\notag\\
		&+C\tau\sum_{n=0}^{m}\left(\|A\mathbf{u}^n\|^2+\|\nabla\mathbf{u}^n\|^2+\|\nabla\mathbf{u}_h^n\|^2\right)\left(\|\Phi_{\mathbf{u}}^{n}\|^2+\|\Theta_{\mathbf{u}}^{n}\|^2\right)\notag\\
		&+C\tau\sum_{n=0}^{m}\left(\|\mathbf{u}_h^n\|^2\|\nabla\mathbf{u}_h^n\|^2\right)\left(|e_{\rho}^{n+1}|^2+\|\Phi_{\mathbf{u}}^n\|^2+\|\Theta_{\mathbf{u}}^n\|^2+\tau^2\|\mathbf{u}_t\|_{L^{\infty}(0,T;L^2(\Omega))}^2\right)\notag\\
		&+C\tau^3\sum_{n=0}^{m}\left(\|\mu_t\|_{L^{\infty}(0,T;H^1(\Omega))}^2\|\nabla\phi^{n+1}\|^2+\|\mu^{n}\|_1^2\|\phi_t\|_{L^{\infty}(0,T;H^1(\Omega))}^2\right)\notag\\
		&+C\tau\sum_{n=0}^{m}\left(\|\Phi_{\mu}^n\|_{-1}^2+\|\Theta_{\mu}^n\|_{-1}^2\right)\|\phi^n\|_2^2+C\tau\|\mu_h^n\|_1^2\left(\|\nabla \Phi_{\phi}^n\|^2+\|\nabla\Theta_{\phi}^n\|^2\right)\notag\\
		&+C\tau\sum_{n=0}^{m}\|\mu_h^n\|_1^2\|\nabla\phi_h^n\|^2\left(|e_{r}^{n+1}|^2+\|\Phi_{\phi}^n\|^2+\|\Theta_{\phi}^n\|^2+\tau^2\|\phi_t\|_{L^{\infty}(0,T;H^1(\Omega))}^2\right).\notag
	\end{flalign}
	Using the inequalities \eqref{eq_boundness_Ritz_0001}-\eqref{eq_boundness_Ritz_0004}, Lemmas \ref{lemma_Stokes_projection_bound}, \ref{lemma_Laplace_opreator}, \ref{lemma_boundedness_E_phi_Deltaphi}, \ref{lemma_truncation_boundedness},  \ref{lemma_Theta_phi_H1}, \ref{lemma_Theta_u_p_L2}, \ref{lemma_boundedness_er_rho}, the discrete Gr\"{o}nwall's Lemma \ref{lemma_discrete_Gronwall_inequation} and the triangle inequality, we obtain the desired result.
\end{proof}

We next prove the optimal $L^2$ error estimate for pressure, and encounter a term $\|e_{\mathbf{u}}^{n+1}-e_{\mathbf{u}}^n\|_{-1}$. To address this issue, we present the following lemma.
\begin{Lemma}\label{lemma_e_n1n_H_neg_boundedness}
	Under the assumption of regularity \eqref{eq_varibles_satisfied_regularities}, for all $m\geq 0$, we have
	$$
		\sum_{n=0}^{m}\|e_{\mathbf{u}}^{n+1}-e_{\mathbf{u}}^n\|_{-1}^2\leq C\left(\tau^2+h^{2(r+1)}\right),
	$$
	where $C$ is a positive constant independent of $\tau$ and $h$.
\end{Lemma}
\begin{proof}
	Letting $\mathbf{v}_h=\tau A_h^{-1}\left(e_{\mathbf{u}}^{n+1}-e_{\mathbf{u}}^n\right)$ in
	\eqref{eq_add_two_error_equations_u_splitting}, using the \eqref{eq_discrete_laplace_operator_norms_differences_0001} and \eqref{eq_error_equations_incompressible_condition}, we have
	\begin{equation}
		\label{eq_e_u_H_nerg_norm_boundedness}
		\begin{aligned}
			\|e_{\mathbf{u}}^{n+1}-e_{\mathbf{u}}^n\|_{-1}^2&-\left(\Delta_he_{\tilde{\mathbf{u}}}^{n+1},\tau A_h^{-1}\left(e_{\mathbf{u}}^{n+1}-e_{\mathbf{u}}^n\right)\right)=\tau\left(R_{\mathbf{u}}^n+R_p^n, A_h^{-1}\left(e_{\mathbf{u}}^{n+1}-e_{\mathbf{u}}^n\right)\right)\\
			&	-\left(\frac{\tau\rho^{n+1}}{\sqrt{E_2^{n+1}}}\mathbf{u}^{n+1}\cdot\nabla\mathbf{u}^{n+1}-\frac{\tau\rho_h^{n+1}}{\sqrt{E_{2h}^n}}\mathbf{u}_h^n\cdot\nabla\mathbf{u}_h^n,A_h^{-1}\left(e_{\mathbf{u}}^{n+1}-e_{\mathbf{u}}^n\right)\right)\\
			&+\left(\frac{\tau r^{n+1}}{\sqrt{E_1^{n+1}}}\mu^{n+1}\nabla\phi^{n+1}-\frac{\tau r_h^{n+1}}{\sqrt{E_{1h}^n}}\mu_h^n\nabla\phi_h^n,A_h^{-1}\left(e_{\mathbf{u}}^{n+1}-e_{\mathbf{u}}^n\right)\right).
		\end{aligned}
	\end{equation} 
	Using the \eqref{eq_error_equations_u_p_splitting}, we have the following equation
	\begin{equation}
		\label{eq_e_u_H_nerg_norm_boundedness_sub0001}
		\begin{aligned}
			-\left(\Delta_he_{\tilde{\mathbf{u}}}^{n+1},\tau A_h^{-1}\left(e_{\mathbf{u}}^{n+1}-e_{\mathbf{u}}^n\right)\right)=&~\tau\left(e_{\tilde{\mathbf{u}}}^{n+1},-\Delta_hA_h^{-1}\left(e_{\mathbf{u}}^{n+1}-e_{\mathbf{u}}^n\right)\right)\\
			=&~\tau\left(e_{\mathbf{u}}^{n+1},e_{\mathbf{u}}^{n+1}-e_{\mathbf{u}}^n\right)\\
			=&~\frac{\tau}{2}\left(\|e_{\mathbf{u}}^{n+1}\|^2-\|e_{\mathbf{u}}^{n}\|^2+\|e_{\mathbf{u}}^{n+1}-e_{\mathbf{u}}^n\|^2\right).
		\end{aligned}
	\end{equation}
	We next estimate the first term on  the right-hand side of \eqref{eq_e_u_H_nerg_norm_boundedness} as follows,
	\begin{equation}
		\label{eq_e_u_H_nerg_norm_boundedness_sub0002}
		\begin{aligned}
			\big|\tau\left(R_{\mathbf{u}}^n+R_p^n, A_h^{-1}\left(e_{\mathbf{u}}^{n+1}-e_{\mathbf{u}}^n\right)\right)\big|
			\leq \frac{1}{10}\|e_{\mathbf{u}}^{n+1}-e_{\mathbf{u}}^n\|_{-1}^2+C\tau^2\left(\|R_{\mathbf{u}}^n\|_{-1}^2+\|R_p^n\|_{-1}^2\right).
		\end{aligned}
	\end{equation}
	Similar to  the estimate of $T_{u7}$, the second term on  the right-hand side of \eqref{eq_e_u_H_nerg_norm_boundedness} 
	is estimated as follows,
	\begin{equation}
		\label{eq_e_u_H_nerg_norm_boundedness_sub0003}
		\begin{aligned}
			&\bigg|-\left(\frac{\tau\rho^{n+1}}{\sqrt{E_2^{n+1}}}\mathbf{u}^{n+1}\cdot\nabla\mathbf{u}^{n+1}-\frac{\tau\rho_h^{n+1}}{\sqrt{E_{2h}^n}}\mathbf{u}_h^n\cdot\nabla\mathbf{u}_h^n,A_h^{-1}\left(e_{\mathbf{u}}^{n+1}-e_{\mathbf{u}}^n\right)\right)\bigg|\\
			=&~\tau\bigg|\left(\mathbf{u}^{n+1}\cdot\nabla\mathbf{u}^{n+1}-\mathbf{u}^n\cdot\nabla\mathbf{u}^n+\mathbf{u}^n\cdot\nabla\mathbf{u}^n
			-\mathbf{u}_h^n\cdot\nabla\mathbf{u}_h^n,A_h^{-1}\left(e_{\mathbf{u}}^{n+1}-e_{\mathbf{u}}^n\right)\right)\\
			&~+\left(\frac{e_{\rho}^{n+1}}{\sqrt{E_{2}^{n+1}}}+\frac{\rho_h^{n+1}}{\sqrt{E_2^{n+1}}}-\frac{\rho_h^{n+1}}{\sqrt{E_{2h}^n}}\right)\left(\mathbf{u}_h^n\cdot\nabla\mathbf{u}_h^n,A_h^{-1}\left(e_{\mathbf{u}}^{n+1}-e_{\mathbf{u}}^n\right)\right)\bigg|.
		\end{aligned}
	\end{equation}
	Thus, we have 
	\begin{equation}
		\label{eq_e_u_H_nerg_norm_boundedness_sub0004}
		\begin{aligned}
			&\tau\big|\left(\mathbf{u}^{n+1}\cdot\nabla\mathbf{u}^{n+1}-\mathbf{u}^n\cdot\nabla\mathbf{u}^n+\mathbf{u}^n\cdot\nabla\mathbf{u}^n
			-\mathbf{u}_h^n\cdot\nabla\mathbf{u}_h^n,A_h^{-1}\left(e_{\mathbf{u}}^{n+1}-e_{\mathbf{u}}^n\right)\right)\big|\\
			\leq &~\frac{1}{10}\|e_{\mathbf{u}}^{n+1}-e_{\mathbf{u}}^n\|_{-1}^2+C\tau^4\|\mathbf{u}_t\|_{L^{\infty}(0,T;H^1(\Omega))}^2\left(\|\nabla\mathbf{u}^{n+1}\|^2+\|\nabla\mathbf{u}^n\|^2\right)\\
			&~+C\tau^2\left(\|A\mathbf{u}^n\|^2+\|\nabla\mathbf{u}^n\|^2+\|\nabla\mathbf{u}_h^n\|^2\right)\|e_{\mathbf{u}}^n\|^2,
		\end{aligned}
	\end{equation}
	and
	\begin{equation}
		\label{eq_e_u_H_nerg_norm_boundedness_sub0005}
		\begin{aligned}
			&\bigg|\left(\frac{e_{\rho}^{n+1}}{\sqrt{E_{2}^{n+1}}}+\frac{\rho_h^{n+1}}{\sqrt{E_2^{n+1}}}-\frac{\rho_h^{n+1}}{\sqrt{E_{2h}^n}}\right)\left(\mathbf{u}_h^n\cdot\nabla\mathbf{u}_h^n,A_h^{-1}\left(e_{\mathbf{u}}^{n+1}-e_{\mathbf{u}}^n\right)\right)\bigg|\\
			\leq&~\frac{1}{10}\|e_{\mathbf{u}}^{n+1}-e_{\mathbf{u}}^n\|_{-1}^2+C\tau^2\|\mathbf{u}_h^n\|^2\|\nabla\mathbf{u}_h^n\|^2\left(|e_{\rho}^{n+1}|^2+\|\Phi_{\mathbf{u}}^n\|^2+\|\Theta_{\mathbf{u}}^n\|^2+\tau\|\mathbf{u}_t\|_{L^{\infty}(0,T;H^1(\Omega))}^2\right).
		\end{aligned}
	\end{equation}
	Similar to the estimate of $T_{u8}$, the third term on  the right-hand side of \eqref{eq_e_u_H_nerg_norm_boundedness} 
	is estimated as follows,
	\begin{equation}
		\label{eq_e_u_H_nerg_norm_boundedness_sub0006}
		\begin{aligned}
			&\bigg|\left(\frac{\tau r^{n+1}}{\sqrt{E_1^{n+1}}}\mu^{n+1}\nabla\phi^{n+1}-\frac{\tau r_h^{n+1}}{\sqrt{E_{1h}^n}}\mu_h^n\nabla\phi_h^n,A_h^{-1}\left(e_{\mathbf{u}}^{n+1}-e_{\mathbf{u}}^n\right)\right)\bigg|\\= &~\tau\bigg|\left(\mu^{n+1}\nabla\phi^{n+1}-\mu^n\nabla\phi^n+\mu^n\nabla\phi^n-\mu_h^n\nabla\phi_h^n,A_h^{-1}\left(e_{\mathbf{u}}^{n+1}-e_{\mathbf{u}}^n\right)\right)\\
			&~+\left(\frac{e_r^{n+1}}{\sqrt{E_1^{n+1}}}+\frac{r_h^{n+1}}{\sqrt{E_1^{n+1}}}-\frac{r_h^{n+1}}{\sqrt{E_{1h}^n}}\right)\left(\mu_h^n\nabla\phi_h^n,A_h^{-1}\left(e_{\mathbf{u}}^{n+1}-e_{\mathbf{u}}^n\right)\right)\bigg|\\
			= &~\tau\big|\left(\left(\mu^{n+1}-\mu^n\right)\nabla\phi^{n+1}+\mu^n\nabla\left(\phi^{n+1}-\phi^n\right)+e_{\mu}^n\nabla\phi^n+\mu_h^n\nabla e_{\phi}^n,A_h^{-1}\left(e_{\mathbf{u}}^{n+1}-e_{\mathbf{u}}^n\right)\right)\\
			&~+\left(\frac{e_r^{n+1}}{\sqrt{E_1^{n+1}}}+\frac{r_h^{n+1}}{\sqrt{E_1^{n+1}}}-\frac{r_h^{n+1}}{\sqrt{E_{1h}^n}}\right)\left(\mu_h^n\nabla\phi_h^n,A_h^{-1}\left(e_{\mathbf{u}}^{n+1}-e_{\mathbf{u}}^n\right)\right)\bigg|.
		\end{aligned}
	\end{equation}
	According to the inequalities \eqref{eq_boundness_basic_inequalities_0001}, \eqref{eq_boundness_basic_inequalities_0002}, we obtain
	\begin{equation}
		\label{eq_e_u_H_nerg_norm_boundedness_sub0007}
		\begin{aligned}
			&\tau\big|\left(\left(\mu^{n+1}-\mu^n\right)\nabla\phi^{n+1}+\mu^n\nabla\left(\phi^{n+1}-\phi^n\right)+e_{\mu}^n\nabla\phi^n+\mu_h^n\nabla e_{\phi}^n,A_h^{-1}\left(e_{\mathbf{u}}^{n+1}-e_{\mathbf{u}}^n\right)\right)\big|\\
			\leq &~\frac{1}{10}\|e_{\mathbf{u}}^{n+1}-e_{\mathbf{u}}^n\|_{-1}^2+C\tau^4\left(\|\phi^{n+1}\|_1^2\|\mu_t\|_{L^{\infty}(0,T;H^1(\Omega))}^2+\|\mu^{n}\|_1^2\|\phi_{tt}\|_{L^{\infty}(0,T;L^2(\Omega))}^2\right)\\
			&+C\tau^2\left(\|\phi^n\|^2\|e_{\mu}^n\|_1^2+\|\mu_h^n\|_1^2\|e_{\phi}^n\|_1^2\right),
		\end{aligned}
	\end{equation}
	and
	\begin{equation}
		\label{eq_e_u_H_nerg_norm_boundedness_sub0008}
		\begin{aligned}
			&\bigg|\left(\frac{e_r^{n+1}}{\sqrt{E_1^{n+1}}}+\frac{r_h^{n+1}}{\sqrt{E_1^{n+1}}}-\frac{r_h^{n+1}}{\sqrt{E_{1h}^n}}\right)\left(\mu_h^n\nabla\phi_h^n,A_h^{-1}\left(e_{\mathbf{u}}^{n+1}-e_{\mathbf{u}}^n\right)\right)\bigg|\\
			\leq &~C\tau\left(\frac{|e_r^{n+1}|}{\sqrt{C_0}}+\|e_{\phi}^n\|+\tau\|\phi_t\|_{L^{\infty}(0,T;H^1(\Omega))}\right)\|\mu_h^n\|_1\|\phi_h^n\|_1\|e_{\mathbf{u}}^{n+1}-e_{\mathbf{u}}^n\|_{-1}\\
			\leq&~\frac{1}{10}\|e_{\mathbf{u}}^{n+1}-e_{\mathbf{u}}^n\|_{-1}^2+C\tau^2\left(|e_r^{n+1}|^2+\|e_{\phi}^n\|^2+\tau^2\|\phi_t\|_{L^{\infty}(0,T;H^1(\Omega))}^2\right)\|\mu_h^n\|_1^2\|\phi_h^n\|_1^2.
		\end{aligned}
	\end{equation}
	Combining the above inequalities \eqref{eq_e_u_H_nerg_norm_boundedness_sub0001}-\eqref{eq_e_u_H_nerg_norm_boundedness_sub0008} with \eqref{eq_e_u_H_nerg_norm_boundedness}, summing up $n$ from $0$ to $m$, we obtain
	\begin{flalign}
		\frac{\tau}{2}\|e_{\mathbf{u}}^{m+1}\|^2&+\frac{1}{2}\sum_{n=0}^{m}\|e_{\mathbf{u}}^{n+1}-e_{\mathbf{u}}^n\|_{-1}^2\leq 
		\frac{\tau}{2}\|e_{\mathbf{u}}^0\|^2+C\tau^2\sum_{n=0}^{m}\left(\|R_{\mathbf{u}}^n\|_{-1}^2+\|R_p^n\|_{-1}^2\right)\notag\\
		&+C\tau^4\sum_{n=0}^{m}\|\mathbf{u}_t\|_{L^{\infty}(0,T;H^1(\Omega))}^2\left(\|\nabla\mathbf{u}^{n+1}\|^2+\|\nabla\mathbf{u}^n\|^2\right)\notag\\
		&+C\tau^2\sum_{n=0}^{m}\left(\|A\mathbf{u}^n\|^2+\|\nabla\mathbf{u}^n\|^2+\|\nabla\mathbf{u}_h^n\|^2\right)\|e_{\mathbf{u}}^n\|^2\notag\\
		&+C\tau^2\sum_{n=0}^{m}\|\mathbf{u}_h^n\|^2\|\nabla\mathbf{u}_h^n\|^2\left(|e_{\rho}^{n+1}|^2+\|\Phi_{\mathbf{u}}^n\|^2+\|\Theta_{\mathbf{u}}^n\|^2+\tau\|\mathbf{u}_t\|_{L^{\infty}(0,T;H^1(\Omega))}^2\right)\notag\\
		&+C\tau^4\sum_{n=0}^{m}\left(\|\phi^{n+1}\|_1^2\|\mu_t\|_{L^{\infty}(0,T;H^1(\Omega))}^2+\|\mu^{n}\|_1^2\|\phi_{tt}\|_{L^{\infty}(0,T;L^2(\Omega))}^2\right)\notag\\
		&+C\tau^2\sum_{n=0}^{m}\left(\|\phi^n\|^2\|e_{\mu}^n\|_1^2+\|\mu_h^n\|_1^2\|e_{\phi}^n\|_1^2\right)\notag\\
		&+C\tau^2\sum_{n=0}^{m}\left(|e_r^{n+1}|^2+\|e_{\phi}^n\|^2+\tau^2\|\phi_t\|_{L^{\infty}(0,T;H^1(\Omega))}^2\right)\|\mu_h^n\|_1^2\|\phi_h^n\|_1^2.\notag
	\end{flalign}
	Using the Lemmas \ref{lemma_boundedness_E_phi_Deltaphi}, \ref{lemma_truncation_boundedness}, \ref{lemma_boundedness_er_rho},
	and the Theorem \ref{theorem_e_phi_u_L2_error_estimates}, we can derive the desired results.
\end{proof}
Then, we give the following theorem for the optimal $L^2$ error estimate of the pressure $p$.
\begin{Theorem}\label{theorem_p_error_estimate0001}
	Suppose that the system \eqref{eqCHNS01} has a unique solution $\left(\phi,\mu,\mathbf{u},p\right)$ satisfying \eqref{eq_varibles_satisfied_regularities}. 
	Then the fully discrete scheme \eqref{eq_fully_discrete_scheme_phi}-\eqref{eq_fully_discrete_scheme_q} have a unique solution $\left(\phi_h^{n+1},\mu_h^{n+1},\mathbf{u}_h^{n+1},p_h^{n+1}\right) \in \mathcal{X}_h^r$ and the initial error is $0$ such that
	$$
		\tau\sum_{n=0}^{m}\|e_p^{n+1}\|^2\leq C\left(\tau^2+h^{2(r+1)}\right),
	$$
	for some $m\geq 0$  and $C$ is a positive constant independent of $\tau$ and $h$.
\end{Theorem}
\begin{proof}
	According to the \eqref{eq_add_two_error_equations_u_splitting}, using the classical inf-sup condition and the estimations of $T_{u7}$ and $T_{u8}$, we obtain
	\begin{flalign}
		C\|e_p^{n+1}\|\leq&~ \frac{\left(e_p^{n+1},\nabla\cdot\mathbf{v}_h\right)}{\|\nabla\mathbf{v}_h\|}\notag\\
		\leq &~\bigg[\left(\frac{e_{\mathbf{u}}^{n+1}-e_{\mathbf{u}}^n}{\tau},\mathbf{v}_h\right)-\left(R_{\mathbf{u}}^n+R_p^n,\mathbf{v}_h\right)\notag\\
		&+\left(\frac{\rho^{n+1}}{\sqrt{E_2^{n+1}}}\mathbf{u}^{n+1}\cdot\nabla\mathbf{u}^{n+1}-\frac{\rho_h^{n+1}}{\sqrt{E_{2h}^n}}\mathbf{u}_h^n\cdot\nabla\mathbf{u}_h^n,\mathbf{v}_h\right)\notag\\
		\label{eq_e_p_error_estimates_00001}
		&-\left(\frac{r^{n+1}}{\sqrt{E_1^{n+1}}}\mu^{n+1}\nabla\phi^{n+1}-\frac{r_h^{n+1}}{\sqrt{E_{1h}^n}}\mu_h^n\nabla\phi_h^n,\mathbf{v}_h\right)\bigg]/\|\nabla\mathbf{v}_h\|\\
		\leq &\frac{1}{\tau}\|e_{\mathbf{u}}^{n+1}-e_{\mathbf{u}}^n\|_{-1}+\|\nabla e_{\tilde{\mathbf{u}}}^{n+1}\|+\|R_{\mathbf{u}}^n\|_{-1}+\|R_p^n\|_{-1}\notag\\
		&+C\tau\|\mathbf{u}_t\|_{L^{\infty}(0,T;L^2(\Omega))}\left(\|A\mathbf{u}^{n+1}\|+\|A\mathbf{u}^n\|\right)\notag\\
		&+C\tau\left(\|A\mathbf{u}^n\|+\|\nabla\mathbf{u}^n\|+\|\nabla\mathbf{u}_h^n\|\right)\left(\|\Phi_{\mathbf{u}}^{n}\|+\|\Theta_{\mathbf{u}}^{n}\|\right)\notag\\
		&+C\tau\|\mathbf{u}_h^n\|\|\nabla\mathbf{u}_h^n\|\left(|e_{\rho}^{n+1}|+\|\Phi_{\mathbf{u}}^n\|+\|\Theta_{\mathbf{u}}^n\|+\tau\|\mathbf{u}_t\|_{L^{\infty}(0,T;H^1(\Omega))}\right)\notag\\
		&+C\tau\left(\|\mu_t\|_{L^{\infty}(0,T;H^1(\Omega))}\|\nabla\phi^{n+1}\|+\|\mu^{n}\|_1\|\phi_{tt}\|_{L^{\infty}(0,T;L^2(\Omega))}\right)\notag\\
		&~+C\tau\|e_{\phi}^n\|\|\phi^n\|_2+C\tau\|\mu_h^n\|_1\left(\|\nabla \Phi_{\phi}^n\|+\|\nabla\Theta_{\phi}^n\|\right)\notag\\
		&~+C\tau\|\mu_h^n\|_1\|\nabla\phi_h^n\|\left(|e_{r}^{n+1}|+\|\Phi_{\phi}^n\|+\|\Theta_{\phi}^n\|+\tau\|\phi_t\|_{L^{\infty}(0,T;H^1(\Omega))}\right).\notag
	\end{flalign}
	By squaring both sides of the above inequality \eqref{eq_e_p_error_estimates_00001} and multiplying by $\tau$, and summing the results from $n = 0$ to $m$, and using the Lemmas \ref{lemma_boundedness_E_phi_Deltaphi}, \ref{lemma_truncation_boundedness}, \ref{lemma_boundedness_er_rho}, \ref{lemma_e_n1n_H_neg_boundedness}
	and the Theorem \ref{theorem_e_phi_u_L2_error_estimates}, we can derive the desired results.
\end{proof}
\section{Numerical test}\label{section_numerical_test}

This section presents three numerical experiments, each serving a distinct purpose: 

1. To validate the convergence rates stated in Theorems \ref{theorem_e_phi_u_L2_error_estimates} and \ref{theorem_p_error_estimate0001};

2. To simulate the coarsening dynamics using a random initial phase function;

3. To numerically demonstrate the unconditional stability of our scheme.

For all experiments, specific finite elements are chosen: $P_1$ for the phase function $\phi^n_h$ and chemical potential $\mu^n_h$, $P_2$ for $\tilde{\mathbf{u}}_h^n$, and the inf-sup stable pair $(P_2, P_1)$ for the velocity $\mathbf{u}^n_h$ and pressure $p^n_h$.

\subsection{Convergence tests}
The first experiment is conducted on the domain $\Omega = [0,1]\times[0,1]$ with a time frame $T = 0.1$. The physical parameters are set as follows: $M = 0.001$, $\lambda = 0.001$, $\epsilon = 0.04$, $\nu = 0.1$, $C_1 = 0.1$, $C_2 = 0.1$, and $\gamma = 1$. The time step $\tau$ is defined as $\tau = 0.1h^3$ for the first-order scheme. The manufactured solutions are chosen from \cite{2022_ChenYaoyao_CHNS_2022_AMC} and takes the form
$$\left\{
\begin{aligned}
	\phi(t,x,y) &= 2 + \sin(t)\cos(\pi x)\cos(\pi y),  \\
	\mathbf{u}(t,x,y) &= \left[\pi\sin(\pi x)^2\sin(2\pi y)\sin(t), -\pi\sin(\pi y)^2\sin(2\pi x)\sin(t)\right]^{\text{T}},  \\
	p(t,x,y) &= \cos(\pi x)\sin(\pi y)\sin(t),
\end{aligned}\right.
$$
and the exact chemical potential $\mu(t,x,y)$ is obtained by its definition, i.e., 
$$
\begin{aligned}
	\mu(t,x,y) &:= -\lambda \Delta \phi(t,x,y) + F^{\prime}\big(\phi(t,x,y)\big) \\
	&=  -\lambda \Delta \phi(t,x,y) + \frac{1}{\epsilon^2} \left(\phi(t,x,y)^3 - \phi(t,x,y)\right) \\
	&= 2\lambda \pi^2 \cos(\pi x)\cos(\pi y) \sin(t) - \frac{1}{\epsilon^2}\left[\cos(\pi x)\cos(\pi y) \sin(t) - \big(\cos(\pi x)\cos(\pi y) \sin(t) + 2\big)^3 + 2 \right].
\end{aligned}
$$
For simplicity of notations, we adopt the following notations:
\begin{equation}
	\label{eq_error_L2}
	\begin{aligned}
		\|\cdot\|_{\ell^{\infty}(L^2)} := \max\limits_{0\leq n \leq N}\|\cdot\|_{L^2}, \quad 
		|\cdot|_{\ell^{\infty}} := \max\limits_{0\leq n \leq N}|\cdot|, \quad \|\cdot\|_{\ell^{2}(L^2)} := \sqrt{\tau \sum_{n=0}^{N}\|\cdot\|_{L^2}}.
	\end{aligned}
\end{equation}
Therefore, according to Theorems \ref{theorem_e_phi_u_L2_error_estimates} and \ref{theorem_p_error_estimate0001}, the following theoretically optimal orders should be observed:
$$
\begin{aligned}
	\left\|\phi - \phi_h\right\|_{\ell^{\infty}(L^2)} & \approx  \mathcal{O}(h^2), \qquad \left\|\mu - \mu_h\right\|_{\ell^2(L^2)}  \approx  \mathcal{O}(h^2), \\
	\left\|\boldsymbol{u} - \boldsymbol{u}_h\right\|_{\ell^{\infty}(L^2)} & \approx   \mathcal{O}(h^3), \qquad \left\|p - p_h\right\|_{\ell^2(L^2)}  \approx  \mathcal{O}(h^2), \\
	\left\|\nabla(\boldsymbol{u} - \boldsymbol{u}_h)\right\|_{\ell^{\infty}(L^2)} & \approx  \mathcal{O}(h^2), \qquad \left|\rho - \rho_h\right|_{\ell^{\infty}} \approx  \mathcal{O}(\tau).
\end{aligned}
$$
The errors and convergence rates illustrated in \autoref{L2-error-convRates-FirstOrder} and \autoref{H1-error-convRates-FirstOrder} for discretization scheme is consistent with the above \emph{a prior} rates.
\begin{table}[h]
	\centering
	\fontsize{10}{10}
	\begin{threeparttable}
		\caption{$L^2$ errors and convergence orders for discretization scheme.}\label{L2-error-convRates-FirstOrder}
		\begin{tabular}{c|c|c|c|c|c|c|c|c}
			\toprule
			\multirow{2.5}{*}{$h$} &
			\multicolumn{2}{c}{$\left\|\phi-\phi_h\right\|_{\ell^{\infty}(L^2)}$} & \multicolumn{2}{c}{$\left\|\mu-\mu_h\right\|_{\ell^{2}(L^2)}$} & \multicolumn{2}{c}{$\left\|\boldsymbol{u}- \boldsymbol{u}_h\right\|_{\ell^{\infty}(L^2)}$} & \multicolumn{2}{c}{$\left\|p-p_h\right\|_{\ell^{2}(L^2)}$} \cr
			\cmidrule(lr){2-3} \cmidrule(lr){4-5} \cmidrule(lr){6-7}  \cmidrule(lr){8-9}
			& error & rate & error & rate & error & rate & error & rate \cr
			\midrule
			$\frac1{4}$  & 3.10010e$-$04  & -  & 5.44192e$-$02  &  -     & 3.71972e$-$03  & -      & 3.06217e$-$02& -      \cr
			$\frac1{8}$  & 6.04209e$-$05 & 2.26  & 7.04442e$-$03&  2.95  & 5.65230e$-$04  & 2.72 & 3.95282e$-$03& 2.95 \cr
			$\frac1{16}$ & 1.37595e$-$05 & 2.12 & 9.85320e$-$04&  2.84    &7.65473e$-$05  & 2.88& 5.08606e$-$04 &2.96\cr
			$\frac1{32}$ &  3.37201e$-$06  & 2.02 & 1.65165e$-$04 &  2.58   & 9.81105e$-$06  & 2.96& 6.87715e$-$05 & 2.88 \cr
			$\frac1{64}$ &  8.39786e$-$07  & 2.00 & 3.42849e$-$05&  2.27   & 1.23442e$-$06  & 2.99& 1.05115e$-$05 & 2.71 \cr
			\bottomrule
		\end{tabular}
	\end{threeparttable}
\end{table}

\begin{table}[h]
	\centering
	\fontsize{10}{10}
	\begin{threeparttable}
		\caption{$H^1$ errors and convergence orders for discretization scheme.}\label{H1-error-convRates-FirstOrder}
		\begin{tabular}{c|c|c|c|c|c|c|c|c}
			\toprule
			\multirow{2.5}{*}{$h$} &
			 \multicolumn{2}{c}{$\left\|\phi - \phi_h\right\|_{H^1}$} & \multicolumn{2}{c}{$\left\|\mu - \mu_h\right\|_{H^1}$} & \multicolumn{2}{c}{$\left\|\boldsymbol{u}- \boldsymbol{u}_h\right\|_{H^1}$} & \multicolumn{2}{c}{$\left\|p - p_h\right\|_{H^1}$} \cr
			\cmidrule(lr){2-3} \cmidrule(lr){4-5} \cmidrule(lr){6-7}  \cmidrule(lr){8-9}
			& error & rate & error & rate & error & rate & error & rate \cr
			\midrule
			$\frac1{4}$  & 8.4382e$-$02 &  -       & 5.39495e$-$01  & -       & 1.54128e$-$01  & -      & 1.83312e$-$01 & -      \cr
			$\frac1{8}$  & 4.3798e$-$02 &  0.95  & 2.96456e$-$01  & 0.86 &4.37273e$-$02  & 1.82 & 4.87931e$-$02& 1.91 \cr
			$\frac1{16}$ & 2.1785e$-$02 &  1.00  & 1.49348e$-$01  & 0.99  & 1.12124e$-$02  & 1.96 &2.21793e$-$02 & 1.14 \cr
			$\frac1{32}$ & 1.0887e$-$02&1.00 &7.47769e$-$02  & 1.00  & 2.82445e$-$03 & 1.99 &1.09564e$-$02& 1.02 \cr
			$\frac1{64}$ & 5.4432e$-$03&  1.00  &3.73949e$-$02  & 1.00  & 7.07441e$-$04& 2.00 & 5.45581e$-$03 & 1.00\cr
			\bottomrule
		\end{tabular}
	\end{threeparttable}
\end{table}

\subsection{Coarsening dynamics}
In this example, we choose $\Omega = [0,1]\times[0,1]$, and $M = 0.0001$, $\lambda = 0.02$, $\epsilon = 0.01$, $\nu = 1$, $C_1 = 1$, $C_2 = 0.1$, and $\gamma = 1$, with a random initial condition for the phase function with values in $[-0.1, 0.1]$, the initial chemical potential takes the same random values as phase function. The spatial length $h=1/64$, and $\tau = 0.001$. We run this test up to final time $T=5$, and record the snapshots of phase function at $t=0.001, 0.05, 0.1, 0.15, 0.3, 1, 3, 5$, respectively, in \autoref{Figure-CoarDyna}. 
\begin{figure}[h!]
	\centering
	\includegraphics[width=1\linewidth]{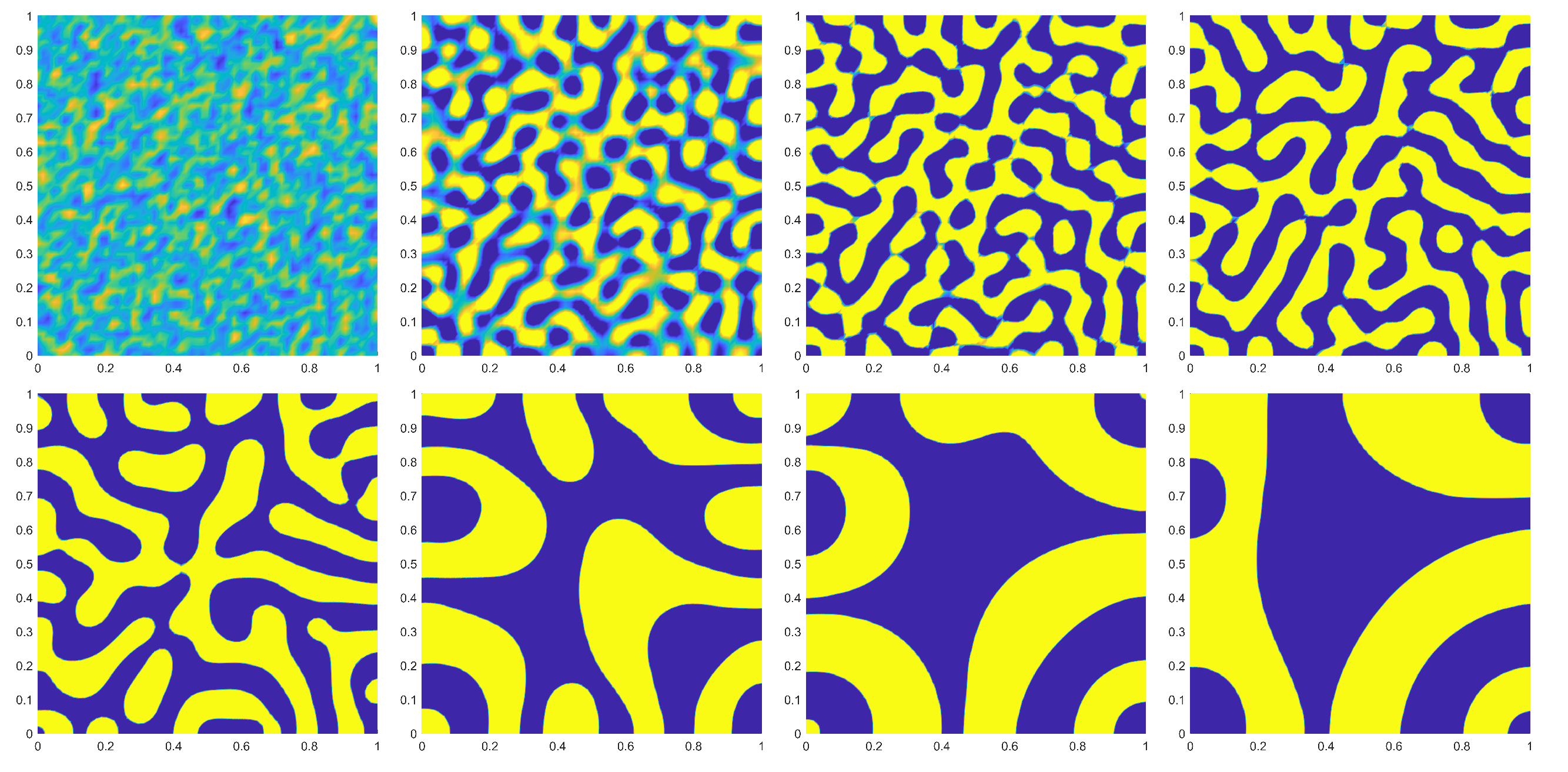}
	\caption{Snapshots of phase function at difference times from left to right row by row with $t=0.001, 0.05, 0.1, 0.15, 0.3, 1, 3, 5$, respectively.}\label{Figure-CoarDyna}
\end{figure}
\subsection{Shape Relaxation with Rotational Boundary Condition}
	In this section, we present an example of shape relaxation within the domain $\Omega = [0, 1] \times [0, 1]$, where the boundary condition is specified as $\mathbf{u} = (y - 0.5, -x + 0.5)$ on $\partial\Omega$. We assess the performance of the scheme \eqref{eq_fully_discrete_scheme_phi}-\eqref{eq_fully_discrete_scheme_q} described by equations \eqref{eq_1st_dis} using critical phase field initial conditions. Specifically, $\phi^0 = 1$ is defined within a polygonal subdomain with reentrant corners, while $\phi^0 = -1$ is set elsewhere in $\Omega$. The initial velocity $\mathbf{u}^0$ is prescribed as $(y - 0.5, -x + 0.5)$. A numerical study of this problem has been conducted in \cite{2008_Kay_David_and_Welford_Richard_Finite_element_approximation_of_a_Cahn_Hilliard_Navier_Stokes_system}. For our simulations, the chosen parameters are: $M = 0.001$, $\lambda = 0.1$, $\epsilon = 0.01$, $\nu = 1$, $\gamma=1$, $C_1=1$, $C_2=0.1$ and $T = 0.5$. The problem is solved using the scheme \eqref{eq_fully_discrete_scheme_phi}-\eqref{eq_fully_discrete_scheme_q} with a $P_1$ element for $\phi_h^{n+1}$ and $\mu_h^{n+1}$, and a $P_2 \times P_1$ element for $\mathbf{u}_h^{n+1}$ and $p_h^{n+1}$. The simulation is run to $T = 0.5$, and phase function snapshots are recorded at times $t = 0, 0.01, 0.02, 0.05, 0.08, 0.2, 0.3, 0.5$, as shown in \autoref{Figure_CHNS_MSAV_FirstOrder_ProjMethods_P1P1P2P1_Gao1_RotaBounCond}.
\begin{figure}[h!]
	\centering
	\includegraphics[width=1\linewidth]{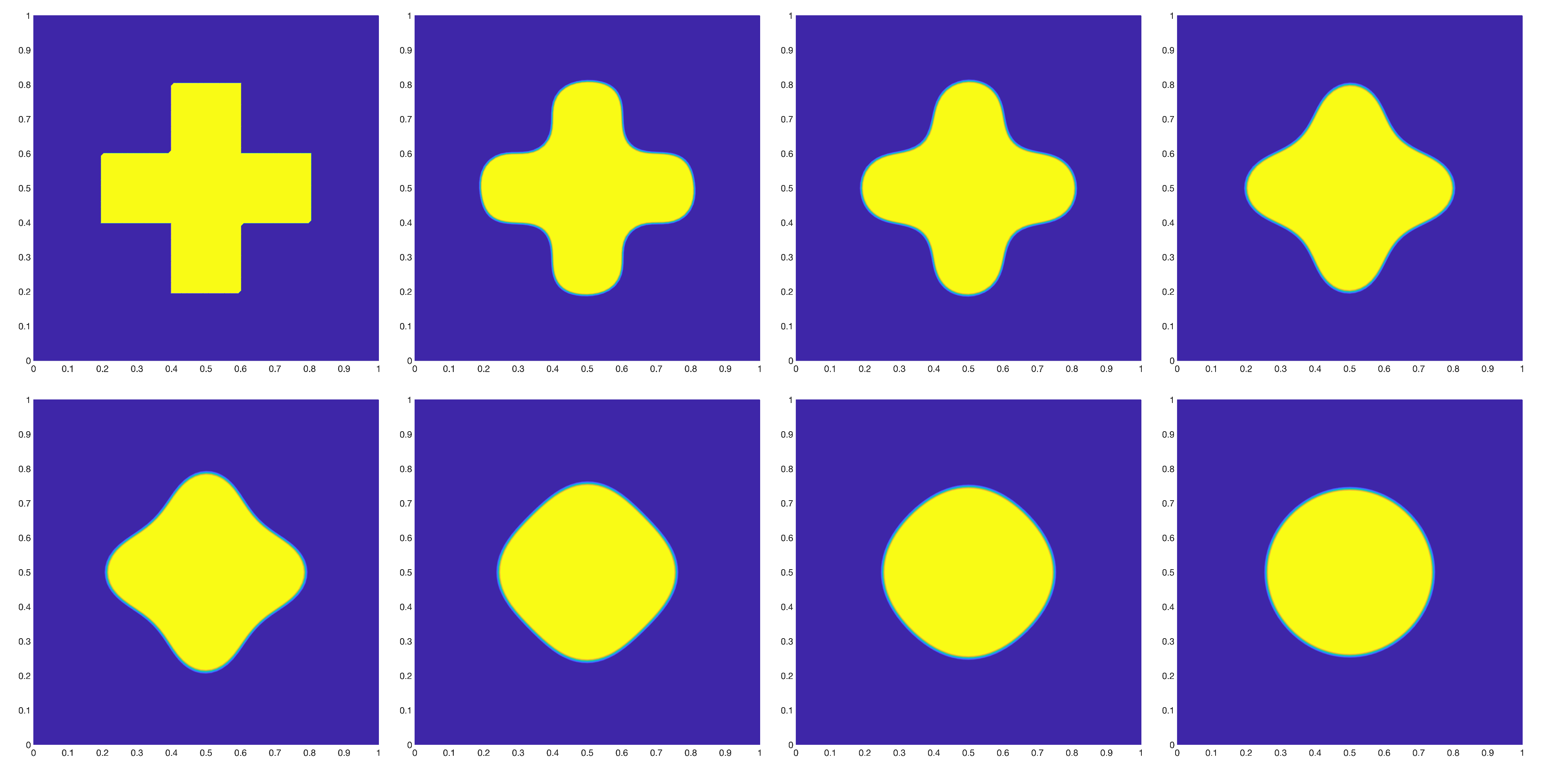}
	\caption{Snapshots of phase function at difference times from left to right row by row with $t=0, 0.01, 0.05, 0.08, 0.1, 0.2, 0.3, 0.5$, respectively.}
	\label{Figure_CHNS_MSAV_FirstOrder_ProjMethods_P1P1P2P1_Gao1_RotaBounCond}
\end{figure}

\subsection{Verification of unconditional stability}
In this final test, we aim to verify the property of unconditional stability of the new scheme here. The same settings as the one in the above coarsening dynamics are chosen but with varying time-step $\tau$. From \autoref{Figure-CoarDyna-Energy-UncondStab}, we observe that our scheme is indeed unconditionally stable in the sense of energy dissipation.
\begin{figure}[h!]
	\centering
	\includegraphics[width=1\linewidth]{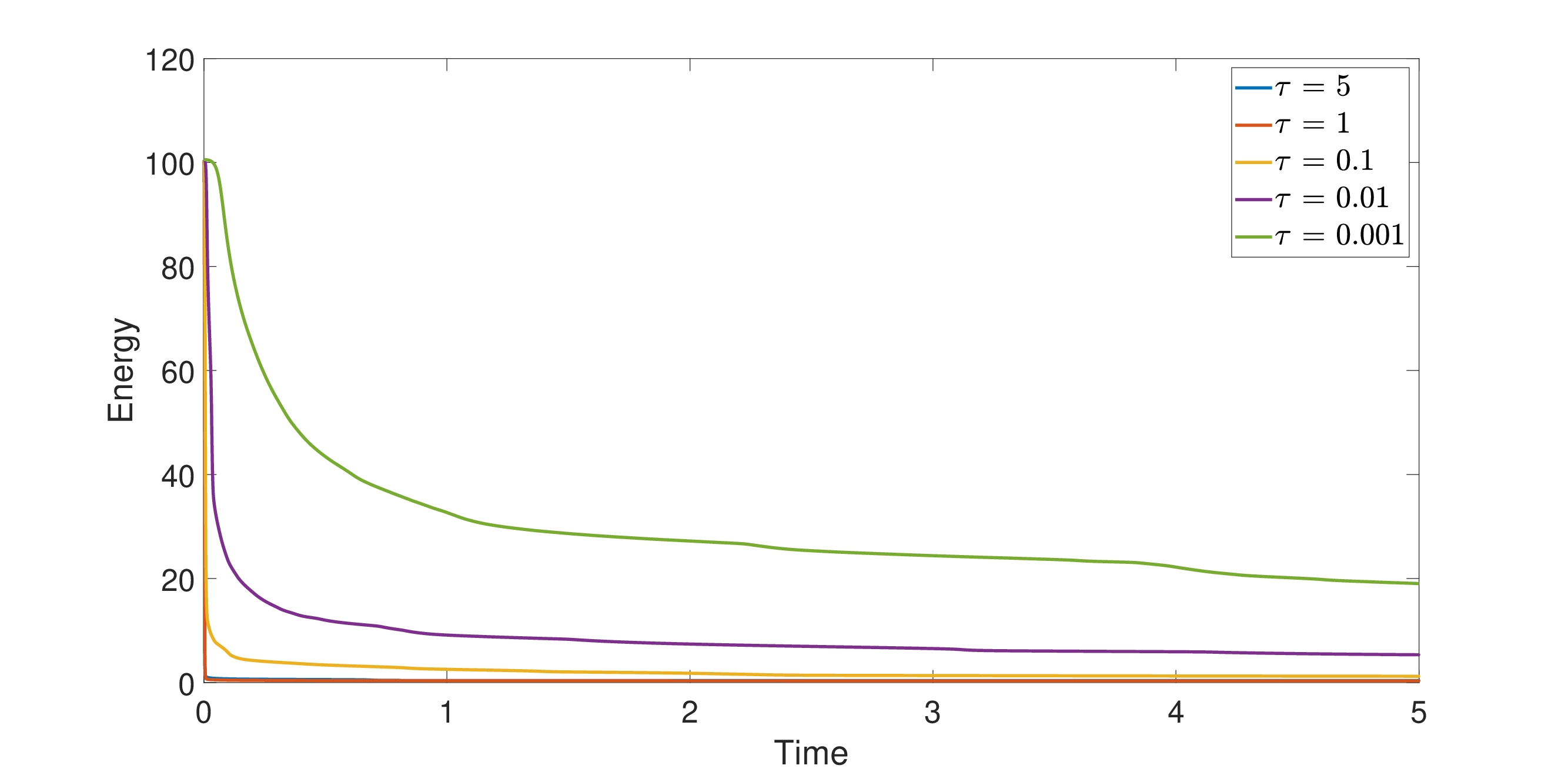}
	\caption{Evolution of the energy with different time-step $\tau$.}\label{Figure-CoarDyna-Energy-UncondStab}
\end{figure}
\section{Conclusion}\label{section_Conclusion}
This manuscript constructs a fully discrete scheme in time and space for the CHNS equations and analyzes the optimal error estimates of the related numerical scheme. The aim of this manuscript is to explore algorithms that are linear, decoupled, and unconditionally energy stable, as well as the optimal error estimates of numerical discretization schemes. 
We obtain that the numerical experimental results indicate that all variables achieve optimal convergence orders. 
In summary, this study conducted a theoretical analysis of the corresponding errors for the numerical scheme of the CHNS equations.

\section*{Acknowledgments}
	This research is supported by the National Natural Science Foundation of China (No.11971337) and the	Natural Science Foundation of Sichuan Province (No. 2025ZNSFSC0070).
	
	\bibliographystyle{unsrt}
	\bibliography{bibfile}
	
\end{document}